\documentclass[12pt,oneside,reqno]{amsart}
\usepackage[latin9]{inputenc}
\usepackage{geometry}
\usepackage{color}
\usepackage{mathrsfs}
\usepackage{amsbsy}
\usepackage{amstext}
\usepackage{amsthm}
\usepackage{amssymb}
\usepackage{graphicx}
\usepackage{setspace}
\makeatletter
\numberwithin{equation}{section}
\theoremstyle{plain}
\newtheorem{thm}{\protect\theoremname}[section]
\theoremstyle{definition}
\newtheorem{defn}[thm]{\protect\definitionname}
\theoremstyle{plain}
\newtheorem{prop}[thm]{\protect\propositionname}
\theoremstyle{remark}
\newtheorem*{rem*}{\protect\remarkname}
\theoremstyle{plain}
\newtheorem{conjecture}[thm]{\protect\conjecturename}
\theoremstyle{remark}
\newtheorem{rem}[thm]{\protect\remarkname}
\theoremstyle{plain}
\newtheorem{cor}[thm]{\protect\corollaryname}
\theoremstyle{remark}
\newtheorem{notation}[thm]{\protect\notationname}
\theoremstyle{plain}
\newtheorem{lem}[thm]{\protect\lemmaname}
\theoremstyle{plain}
\newtheorem*{thm*}{\protect\theoremname}
\theoremstyle{remark}
\newtheorem*{acknowledgement*}{\protect\acknowledgementname}

\usepackage{amsthm}\usepackage{mathrsfs}\usepackage{nameref}\usepackage{cite}\usepackage{upgreek}
\usepackage{stmaryrd}

\numberwithin{equation}{section}

\def\th@plain{\thm@notefont{}\itshape}
\def\th@definition{\thm@notefont{}\normalfont}
\usepackage[colorlinks=true,linkcolor=blue, citecolor = blue]{hyperref}

\makeatother

\providecommand{\acknowledgementname}{Acknowledgement}
\providecommand{\conjecturename}{Conjecture}
\providecommand{\corollaryname}{Corollary}
\providecommand{\definitionname}{Definition}
\providecommand{\lemmaname}{Lemma}
\providecommand{\notationname}{Notation}
\providecommand{\propositionname}{Proposition}
\providecommand{\remarkname}{Remark}
\providecommand{\theoremname}{Theorem}

\begin{document}
\title[Condensation and Metastable Behavior of Non-reversible IP]{Condensation and Metastable Behavior of Non-Reversible Inclusion Processes}
\author{Seonwoo Kim and Insuk Seo }
\address{S. Kim. Department of Mathematical Sciences, Seoul National University,
Republic of Korea.}
\email{ksw6leta@snu.ac.kr}
\address{I. Seo. Department of Mathematical Sciences and R.I.M., Seoul National
University, Republic of Korea.}
\email{insuk.seo@snu.ac.kr}
\begin{abstract}
In this article, we perform quantitative analyses of metastable
behavior of an interacting particle system known as the inclusion
process. For inclusion processes, it is widely believed that the system
nucleates the condensation of particles because of the attractive
nature of the interaction mechanism. The metastable behavior of the
inclusion processes corresponds to the movement of the condensate
on a suitable time scale, and the computation of the corresponding
time scale and the characterization of the scaling limit of the condensate
motion are the main problems in the study of metastability of inclusion
processes. Previously, these problems were solved for reversible inclusion
processes in {[}Bianchi, Dommers, and Giardin\`{a}, Electronic Journal
of Probability, 22: 1-34, 2017{]}, and the main contribution of the
present study is to extend this analysis to a wide class of \textit{non-reversible}
inclusion processes. Non-reversibility is a major obstacle to analyzing
such models, mainly because there is no closed-form expression of
the invariant measure for the general case, and our main achievement
is to overcome this difficulty. In particular, our results demonstrate
that the time scale and limiting process of non-reversible inclusion
processes are quantitatively and qualitatively different from those
of reversible ones, respectively. We emphasize that, to the best of
our knowledge, these results are the \textit{first rigorous quantitative
results in the study of metastability when the invariant measure is
not explicitly known.} In addition, we consider the thermodynamic
limit of metastable behavior of inclusion processes on large torus
as in the paper {[}Armend\'{a}riz, Grosskinsky, and Loulakis, Probability Theory
and Related Fields, 169: 105-175, 2017{]}. For this model, we observe
\textit{three different time scales} according to the level of asymmetry
of the model.
\end{abstract}


\maketitle
\newpage\tableofcontents

\section{\label{sec1}Introduction }

Metastability is a ubiquitous phenomenon that occurs in various stochastic
systems, such as the small random perturbation of dynamical systems
\cite{B-E-G-K,F-W,L-Mar-Seo,L-Seo,L-Seo diff,Rez-Seo}, low-temperature
ferromagnetic systems \cite{B-B-I,B-M,L-Seo Potts,N-Z}, and interacting
particle systems consisting of sticky particles \cite{A-G-L,A-C-R,B-L ZRP,B-D-G,C-C-G,C-G-J,Go-L,Gr-R-V 11,Gr-R-V 13,Gr-S-S,L TAZRP,Seo NRZRP}.
The present study focuses on quantitative analyses of the metastable
behavior of an interacting particle system known as the \textit{inclusion
process}, whose precise mathematical formulation is given in Section
\ref{sec21}.

\subsection{\label{sec11}Condensation of inclusion processes}

Systems of particles interacting under the attractive interaction
mechanism exhibit a phenomenon known as \textit{condensation}, i.e.,
a macroscopically significant portion of the particles is concentrated
at a site (cf. Definition \ref{def12}). Over the last decade, comprehensive
studies have been conducted to understand this phenomenon, especially
for two representative stochastic particle systems: zero-range processes
\cite{A-G-L,B-L ZRP,Go-L,Gr-S-S,J-M-P,L TAZRP,Seo NRZRP} and inclusion
processes \cite{A-C-R,B-D-G,C-C-G,C-G-J,Gr-R-V 11,Gr-R-V 13}. These
studies have mainly focused on the following objectives:
\begin{itemize}
\item Establishing the existence of condensation by demonstrating that a
large portion of the particles is located at only one site with dominating
probability under the invariant measure of the dynamics.
\item Analyzing the metastable behavior of the condensate: once the appearance
of the condensate has been successfully established, the next objective
is to investigate the \textit{dynamical movements of the condensate}.
Successive movements of the condensate can be regarded as metastable
transitions studied in the context of metastability (cf. \cite{B-L TM,B-L TM2,B-E-G-K}).
\end{itemize}
In this study, we attempt to achieve these objectives for inclusion
processes, especially non-reversible ones, for which the invariant
measure cannot be written in a closed-form.

\subsubsection*{Condensation of inclusion processes}

The inclusion process is an interacting particle system that is expected
to exhibit condensation, and it recently attracted considerable
interest in the study of metastability. The study of condensation
of inclusion processes originated from the work \cite{Gr-R-V 11}
of Grosskinsky, Redig, and Vafayi, who demonstrated this
under either of the following conditions: \textit{reversibility} or
\textit{double stochasticity} of the underlying random walk. Indeed, in \cite{C-G}, the authors proved that a large class of particle systems including the inclusion process have explicit product-type formulas for the invariant distribution under either of these conditions (cf. \cite[Theorem 2.1]{C-G}). Their proof of condensation is based solely on these formulas. \textbf{\textit{The first contribution of the present study is to
prove the condensation of a wide class of inclusion processes without
such an explicit expression of the invariant measure. Moreover, we
obtain sharp asymptotics for the mass of each metastable valley. }}

The metastable behavior of inclusion processes was firstly analyzed
in \cite{B-D-G,Gr-R-V 13} for the reversible case. Meanwhile, owing
to the lack of a closed-form expression for the invariant measure,
the metastable behavior of non-reversible inclusion processes has
not been analyzed rigorously thus far; the only existing study is
\cite{C-C-G}, in which an asymmetric (i.e., non-reversible) model
on a torus was addressed by computational methods. \textbf{\textit{The
second contribution of this study is to derive rigorous results on
the metastable behavior of non-reversible inclusion processes}}, which is done by
developing a sequence of novel computations for inclusion processes.

Further, we consider the thermodynamic limit of the condensate as
in \cite{A-G-L}, for which the underlying lattice structure grows
together with the number of particles. Thus, a suitable time-space
rescaling of the condensate motion is expected to converge to a certain
diffusion process. \textbf{\textit{The third contribution of this
study is to demonstrate the existence of three different time scales
for the thermodynamic limit according to the level
of non-reversibility of the process and to characterize the limiting
process in a precise manner.}} We remark that, to the best of our
knowledge, such an interesting phenomenon has never been observed
in any other model.

\subsubsection*{Main difficulty: non-reversibility }

The main challenge in the problems that we are going to consider in
this study originates from the non-reversibility of processes. Quantitative
analysis of the metastable behavior of non-reversible processes is
a long-standing open question in the research of metastability because
of the following two main difficulties associated with such processes:
\begin{enumerate}
\item Absence of the variational principle known as the Dirichlet--Thomson
principle, which enables us to estimate the potential-theoretic quantities
such as the capacity between metastable sets.
\item Absence of the explicit form of the invariant measure.
\end{enumerate}
The first difficulty was recently resolved in \cite{G-L,S}, in which
the Dirichlet--Thomson principle for non-reversible Markov chains
was established, and a manual for applying these generalized tools
in the context of metastability was also developed in \cite{L-Seo}.
On the basis of these studies, numerous results of the analysis of
metastability of non-reversible processes were presented in \cite{L TAZRP,L-Mar-Seo,L-Seo Potts,L-Seo,L-Seo diff,Seo NRZRP}.
We remark that the models considered in these studies have a closed-form
expression for the invariant measure; hence, the second difficulty
mentioned above is not applicable.

In contrast to the first difficulty mentioned above, the second one
remains a major obstacle. This is not merely a technical issue, as
all existing tools for the analysis of metastable behavior use highly
precise knowledge about the invariant measure in a neighborhood of
the transition path between metastable sets. Therefore, general models
without the closed-form expression of the invariant measure have not
been addressed thus far. For instance, the Eyring--Kramers-type result
for non-reversible diffusions considered by Freidlin and Wentzell
\cite{F-W} remains unresolved because of such a difficulty. \textbf{\textit{We
emphasize that the present study provides the first metastability
result that overcomes this difficulty.}}

\subsection{\label{sec12}Zero-range processes}

The most investigated particle system in the context of the condensation
phenomenon is the sub-critical zero-range processes. We refer to \cite{Seo NRZRP}
and the references therein for a comprehensive account of the long
history of the investigation of condensation in zero-range processes.
Here, we briefly review a part of this history to understand the state
of the art of studies on the metastable behavior of interacting particle
systems and to compare zero-range processes with inclusion processes.

Condensation in the zero-range process was firstly observed in \cite{J-M-P};
since then, it has taken nearly a decade to complete to answer most
of the relevant questions with sufficient generality. Among the various
studies, we review those on the analysis of metastable behavior. First,
Beltr\`{a}n and Landim \cite{B-L ZRP} analyzed the metastable behavior
of reversible zero-range processes. As non-reversible zero-range processes
have the same (closed-form) invariant measure as reversible ones,
the analysis of metastability was extended to the non-reversible case
in \cite{L TAZRP} and \cite{Seo NRZRP} on the basis of recent technologies
for the analysis of the metastability of non-reversible processes.
From these successful studies on non-reversible zero-range processes,
one can infer that the study of non-reversible inclusion processes
can be reduced to the study of the invariant measure. However, the
main problem is that no existing tool can be applied without the closed-form
of the invariant measure. In this study, we will overcome this problem
by introducing a new way of analyzing inclusion processes.

\begin{figure}
\includegraphics[scale=0.19]{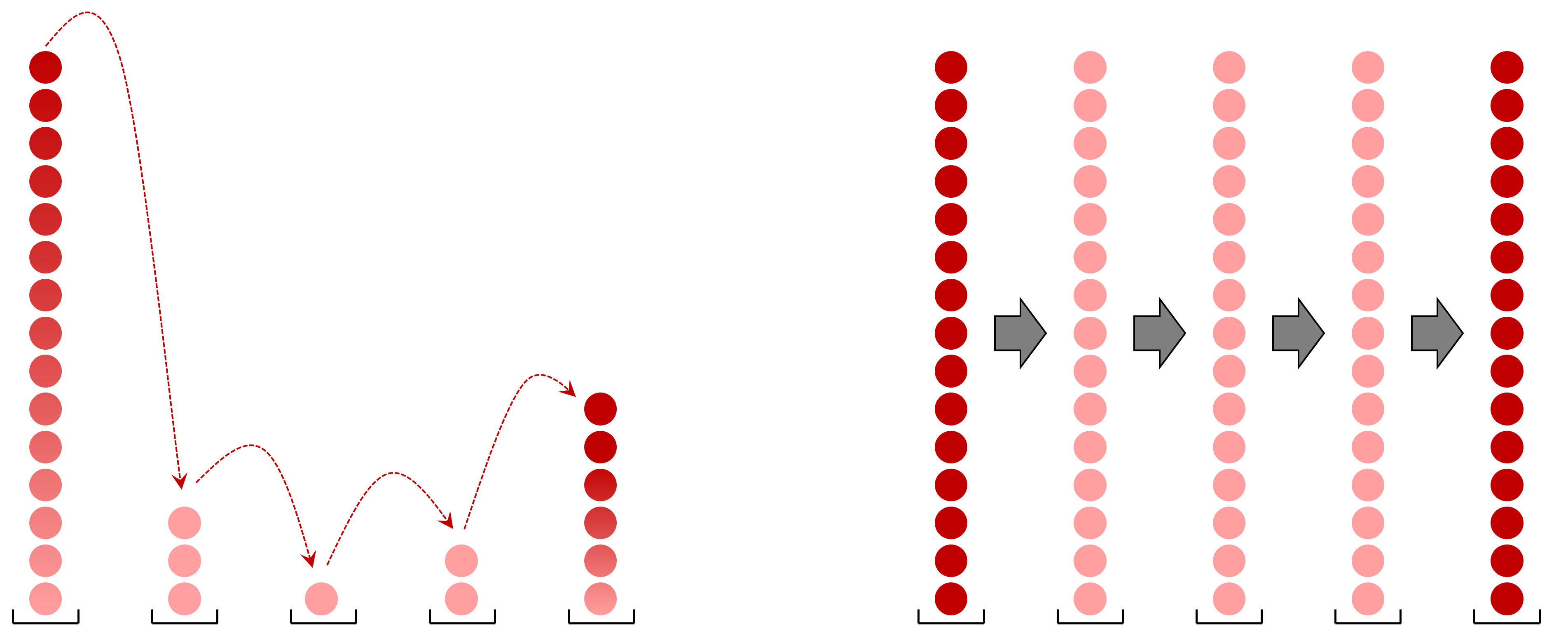}\caption{\label{fig0}Difference between the movements of the condensate in
zero-range processes (left) and inclusion processes (right).}
\end{figure}

\subsubsection*{Comparison between zero-range process and inclusion process}

We conclude this section with a brief explanation of the intrinsic
difference between the metastable behavior of zero-range processes
and that of inclusion processes. Figure \ref{fig0} shows a visualization
of this difference. First, we explain the mechanism for the transition
of the condensate for the zero-range process. Initially, a few particles
are detached from the condensation of the zero-range process. These
particles wander momentarily and finally form a small new condensate
at a site that might be far away from the original condensate. Then,
the movement of the full condensation is completed by sending particles
from the original condensate to this new one as in Figure \ref{fig0}-(left).
Because of this mechanism, the condensate for the zero-range process
has long-range movements. Meanwhile, the condensate of the simple inclusion process cannot exhibit  long-range movements; it has to move to its neighboring sites. To move the condensate to a distant site, the condensate has to move along neighboring sites successively, as shown in Figure \ref{fig0}-(right). Our
results will formulate this difference in a concrete form.

\section{\label{sec2}Condensation of Inclusion Processes}

In this section, we introduce the inclusion process and explain the
condensation phenomenon in a more concrete form. More precisely, we
formulate the inclusion process in Section \ref{sec21} and then review
the known condensation results in Section \ref{sec22}. Here, we remark
that our new results will be presented in Section \ref{sec3}.

\subsection{\label{sec21}Inclusion processes}

The inclusion process is a particle system consisting of \textit{interacting
random walks on a finite state set $S$}. Thus, we should start by
introducing the underlying random walk on $S$ constituting the inclusion
process.
\begin{defn}[Underlying random walk]
\label{def11}The underlying random walk
is a continuous-time, irreducible Markov chain $(X(t))_{t\in[0,\,\infty)}$ on $S$ with jump rate
$r(\cdot,\,\cdot):S\times S\rightarrow[0,\,\infty)$. Let $m(\cdot)$
denote the invariant measure of the Markov chain $X(\cdot)$. For
the simplicity of the discussion, we set $r(x,\,x)=0$ for all $x\in S$.
\end{defn}

The inclusion process is defined as a continuous-time Markov chain
on the set\footnote{In this article, $\mathbb{N}$ includes $0$, i.e., $\mathbb{N}=\{n\in\mathbb{Z}:n\ge0\}.$}
$\mathcal{H}_{N}\subseteq\mathbb{N}^{S}$ defined by
\[
\mathcal{H}_{N}=\bigg\{\eta=(\eta_{x})_{x\in S}\in\mathbb{N}^{S}:\sum_{x\in S}\eta_{x}=N\bigg\}\;.
\]
Here, $\eta_{x}$ can be regarded as the number of particles at the
site $x\in S$; hence, $\eta$ represents the particle configuration
on $S$. For $\eta\in\mathcal{H}_{N}$ and $x,\,y\in S$,\footnote{Writing $u,\,v\in T$ or $\{u,\,v\}\subseteq T$ implicitly implies
that $u$ and $v$ are different.} let $\sigma^{x,\,y}\eta\in\mathcal{H}_{N}$ denote the configuration
obtained by sending a particle, if possible, from $x$ to $y$ in
$\eta$. In other words, for $\eta$ with $\eta_{x}\ge1$, define
\[
(\sigma^{x,\,y}\eta)_{z}=\begin{cases}
\eta_{x}-1 & \text{if }z=x\;,\\
\eta_{y}+1 & \text{if }z=y\;,\\
\eta_{z} & \text{otherwise}\;,
\end{cases}
\]
and we set $\sigma^{x,\,y}\eta=\eta$ if $\eta_{x}=0$. Now, we are
ready to define the inclusion process.
\begin{defn}[Inclusion process]
Let $\{d_{N}\}_{N\ge1}$ be a sequence of positive real numbers converging
to $0$. The inclusion process $\{\eta_{N}(t)\}_{t\ge0}$ is a continuous-time
Markov chain on $\mathcal{H}_{N}$ associated with the generator $\mathcal{L}_{N}$
given by
\begin{equation}
(\mathcal{L}_{N}F)(\eta)=\sum_{x,\,y\in S}\eta_{x}(d_{N}+\eta_{y})r(x,\,y)\{F(\sigma^{x,\,y}\eta)-F(\eta)\}\label{gen}
\end{equation}
for all $\eta\in\mathcal{H}_{N}$ and $F:\mathcal{H}_{N}\rightarrow\mathbb{R}$.
\end{defn}

Now, we briefly explain the dynamics of the inclusion process. According
to the generator \eqref{gen}, for a configuration $\eta$, a particle
moves from site $x$ to site $y$ at the rate
\[
\eta_{x}(d_{N}+\eta_{y})r(x,\,y)=\eta_{x}\eta_{y}r(x,\,y)+d_{N}\eta_{x}r(x,\,y)\;.
\]
Thus, we can divide the dynamics into two components. The first one
corresponding to the term $\eta_{x}\eta_{y}r(x,\,y)$ denotes the
attractive interaction of the system, as this term increases with
$\eta_{y}$, which means that particles are more likely to move to
more occupied sites. Meanwhile, the dynamics corresponding to the
term $d_{N}\eta_{x}r(x,\:y)$ denotes the diffusive behavior of the
particles. However, if the parameter $d_{N}$ is sufficiently small,
this diffusive dynamics is dominated by the attractive interaction;
consequently, we can expect condensation of the particles at one site.
However, it is the second type of dynamics that gives rise to the
transition of the condensate; when this diffusive effect accumulates
for a sufficiently long time, we may observe the movement of the condensate
to another one.

We conclude this subsection by introducing several notations regarding
the inclusion processes as follows:
\begin{itemize}
\item Let $\mathbf{r}_{N}(\cdot,\,\cdot)$ denote the transition rate kernel
associated with the inclusion process $\eta_{N}(\cdot)$, i.e.,
\begin{equation}
\mathbf{r}_{N}(\eta,\,\eta')=\begin{cases}
\eta_{x}(d_{N}+\eta_{y})r(x,\,y) & \text{if }\eta'=\sigma^{x,\,y}\eta\;,\\
0 & \text{otherwise\;.}
\end{cases}\label{rn}
\end{equation}
\item Let $\lambda_{N}(\cdot)$ denote the corresponding holding rate:
\begin{equation}
\lambda_{N}(\eta)=\sum_{\eta'\in\mathcal{H}_{N}}\mathbf{r}_{N}(\eta,\,\eta')\;,\label{holding}
\end{equation}
\item Let $\mathbf{p}_{N}(\cdot,\cdot)$ denote the jump probability kernel:
\begin{equation}
\mathbf{p}_{N}(\eta,\,\eta')=\frac{\mathbf{r}_{N}(\eta,\,\eta')}{\lambda_{N}(\eta)}\;.\label{pn}
\end{equation}
\item Let $\mathbb{P}_{\eta}=\mathbb{P}_{\eta}^{N}$ and $\mathbb{E}_{\eta}=\mathbb{E}_{\eta}^{N}$
be the law and the expectation with respect to the process $\eta_{N}(\cdot)$
starting from $\eta$, respectively.
\item We can readily verify that the inclusion process defined above is
an irreducible Markov chain on $\mathcal{H}_{N}$ and thus has a unique
invariant measure on $\mathcal{H}_{N}$. Let $\mu_{N}$ denote this
unique invariant measure.
\end{itemize}

\subsection{\label{sec22}Condensation of inclusion processes}

In this subsection, we summarize all the known results regarding the
condensation of inclusion processes.

\subsubsection*{Condensation on metastable sets}

To describe this condensation phenomenon in a more concrete form,
we introduce the metastable set. In the context of inclusion processes,
this metastable set is very simple, i.e., it is just a singleton set.
For $x\in S$, let $\xi_{N}^{x}\in\mathcal{H}_{N}$ denote the configuration
with all the particles located at $x$, i.e.,
\[
(\xi_{N}^{x})_{y}=\begin{cases}
N & \text{if }y=x\;,\\
0 & \text{otherwise\;.}
\end{cases}
\]
For each $x\in S$, define the set $\mathcal{E}_{N}^{x}$ by
\[
\mathcal{E}_{N}^{x}=\{\xi_{N}^{x}\}=\{\eta\in\mathcal{H}_{N}:\eta_{x}=N\}\;.
\]
This set is metastable in the sense that not only the rate of escaping
from this set is extremely low but also the likelihood of returning
to this set immediately after escape is extremely high. For the inclusion process, we may to define the metastable valleys $\mathcal{E}_{N}^{x}$ as singletons, because these configurtaions are sufficient to dominate the invariant distribution of the dynamics (cf. Proposition \ref{p14}).

For a subset $R$ of $S$, define
\[
\mathcal{E}_{N}(R)=\bigcup_{x\in R}\mathcal{E}_{N}^{x}\;\;\;\text{and\;\;\;}\mathcal{E}_{N}=\mathcal{E}_{N}(S)\;.
\]
With these terminologies, we are now ready to formulate the condensation
in a concrete form.
\begin{defn}[Condensation]
\label{def12}The inclusion process exhibits \textbf{\textit{\textcolor{black}{condensation}}} if
\[
\lim_{N\rightarrow\infty}\mu_{N}(\mathcal{E}_{N})=1\;;
\]
and to exhibit condensation on $R\subseteq S$ if
\[
\lim_{N\rightarrow\infty}\mu_{N}(\mathcal{E}_{N}(R))=1\;.
\]
If the condensation occurs, we define the \textit{maximal condensing
set} as
\begin{equation}
S_{\star}=\left\{ x\in S:\limsup_{N\rightarrow\infty}\mu_{N}(\mathcal{E}_{N}^{x})>0\right\} \ne\emptyset\;.\label{sstar}
\end{equation}
Hence, $S_{\star}$ denotes the smallest set on which the condensation
occurs. Finally, we write the remainder set as
\[
\Delta_{N}=\mathcal{H}_{N}\setminus\mathcal{E}_{N}(S_{\star})\;.
\]
\end{defn}

\subsubsection*{Formula for invariant measure: two special conditions}

Now, we introduce two special conditions for the underlying random walk defined
in Definition \ref{def11} that enable us to write the invariant measure
in an explicit form.\\
\\
\textbf{(Rev)} The underlying random walk $X(\cdot)$ is \textit{reversible
}with respect to its invariant measure, i.e.,
\begin{equation}
m(x)r(x,\,y)=m(y)r(y,\,x)\text{ for all }x,\,y\in S\;,\label{revcon}
\end{equation}
such that the inclusion process is also reversible with respect to
its invariant measure $\mu_{N}(\cdot)$.\\
\\
\textbf{(UI)} The invariant measure $m(\cdot)$ for the underlying
random walk $X(\cdot)$ is the\textit{ uniform measure on $S$.}\\

To explain the invariant measure for these cases, we define several
notations. On the basis of the invariant measure $m(\cdot)$ for the
underlying random walk, we introduce the following notations:
\begin{equation}
M_{*}=\max\{m(x):x\in S\}\;\;\;\text{and\;\;\;}S_{\mathrm{max}}=\{x\in S:m(x)=M_{*}\}\;.\label{mmax}
\end{equation}
Finally, we introduce an auxiliary function $w_{N}:\mathbb{N}\rightarrow(0,\,\infty)$
as
\[
w_{N}(n)=\frac{\Gamma(n+d_{N})}{n!\Gamma(d_{N})}\;,\quad n\in\mathbb{N}\;,
\]
where $\Gamma$ denotes the usual gamma function. Then, we deduce
the following formula under \textbf{(Rev)} or \textbf{(UI)}.
\begin{prop}
\label{p13}Under the condition \textbf{\textup{(Rev)}} or \textbf{\textup{(UI)}},
the invariant measure $\mu_{N}(\cdot)$ can be written as
\begin{equation}
\mu_{N}(\eta)=\frac{1}{Z_{N}}\prod_{x\in S}\left(\frac{m(x)}{M_{*}}\right)^{\eta_{x}}w_{N}(\eta_{x})\text{ for all }\eta\in\mathcal{H}_{N}\label{mundef}
\end{equation}
where the partition function $Z_{N}$ is given by
\[
Z_{N}=\sum_{\eta\in\mathcal{H}_{N}}\prod_{x\in S}\left(\frac{m(x)}{M_{*}}\right)^{\eta_{x}}w_{N}(\eta_{x})\;.
\]
\end{prop}

We remark that $\frac{m(x)}{M_{*}}=1$ for all $x\in S$ under the
condition \textbf{(UI)}. The proof for the case \textbf{(Rev)} is
straightforward, as the following detailed balance condition holds:
\[
\mu_{N}(\eta)\mathbf{r}_{N}(\eta,\,\eta')=\mu_{N}(\eta')\mathbf{r}_{N}(\eta',\,\eta)\;.
\]
This implies that the inclusion process is also reversible with respect
to $\mu_{N}(\cdot)$. For the case \textbf{(UI)}, the proof is presented
in \cite[Theorem 2.1(a)]{Gr-R-V 11}; nevertheless, we provide a short
proof in Section \ref{sec5} for the completeness of the study. Based
on the explicit formula \eqref{mundef}, the following result is established
in \cite[Proposition 2.1]{B-D-G}.
\begin{prop}
\label{p14}Suppose that $\mu_{N}(\cdot)$ admits the formula \eqref{mundef}
and $\lim_{N\rightarrow\infty}d_{N}\log N=0$. Then, it holds
that
\[
\lim_{N\rightarrow\infty}\mu_{N}(\mathcal{E}_{N}^{x})=\frac{1}{|S_{\mathrm{max}}|}\text{ for all }x\in S_{\mathrm{max}}\;.
\]
In other words, the inclusion process exhibits the condensation on
$S_{\mathrm{max}}$; moreover, $S_{\star}=S_{\mathrm{max}}$. In particular,
for the case \textbf{\textup{(UI)}}, we have $S_{\star}=S$.
\end{prop}

Here, we emphasize that the proof of this proposition is based entirely
on the formula \eqref{mundef}. Without this expression, proving the
condensation phenomenon becomes a completely non-trivial task; we
confront this difficulty in this study.

\section{\label{sec3}Main Results}

In this section, we explain the main results obtained in this article.
Our primary concern is the metastable behavior of the condensate of
the inclusion process. Rigorous analysis of this metastable behavior
was previously restricted to the inclusion process satisfying \textbf{(Rev)}.
We will extend these results to
\begin{enumerate}
\item inclusion processes satisfying \textbf{(UI)} (cf. Section \ref{sec32}),
\item inclusion processes for which jump rate $r(\cdot,\,\cdot)$ is uniformly
positive (cf. Section \ref{sec33}),
\item inclusion processes in the thermodynamic limit regime for which the
underlying graph ($d$-dimensional discrete torus) grows together
with the number of particles (cf. Section \ref{sec34}).
\end{enumerate}
For these cases, the inclusion process can be non-reversible. In particular,
for case (2), even the invariant measure cannot be written in an explicit
form; hence, the existence of the condensation is unknown. We shall
establish this existence of the condensation in Theorem \ref{t311}.

\subsection{\label{sec31}Description of metastable behavior}

Before explaining our main results, we briefly review the canonical
methodology developed in \cite{B-L TM,B-L TM2} for the description
of the metastable behavior of the stochastic systems as a convergence
of the so-called trace process. We explain this methodology in the
context of inclusion processes for the convenience of the readers.
The successive movements of condensate in the inclusion process can
be regarded as a transition among the metastable sets $\mathcal{E}_{N}^{x}=\{\xi_{N}^{x}\}$,
$x\in S_{\star}$, as the condensation at a site $x\in S\setminus S_{\star}$
will be mollified on a shorter time scale than the condensation at
a site $x\in S_{\star}$. Hence, by identifying the state $\xi_{N}^{x}$
with $x$ and ignoring short excursions on $\Delta_{N}=\mathcal{H}_{N}\setminus\mathcal{E}_{N}(S_{\star})$,
the resulting dynamics converges to a Markov chain on $S_{\star}$
after suitable time-rescaling. In the context of metastability theory,
this procedure is a canonical way of describing the metastable behavior,
and it has been systematically established in \cite{B-L TM,B-L TM2}
on the basis of the martingale approach. In this approach, the procedure
explained above is understood as the convergence of the so-called
trace process. Hence, we now introduce the trace process in the context
of the inclusion process.
\begin{defn}[Trace process of the inclusion process]
Fix a non-empty set $\mathcal{G}\subseteq\mathcal{H}_{N}$ and define
a (random) non-decreasing function by
\[
T^{\mathcal{G}}(t)=\int_{0}^{t}\mathbf{1}\left\{ \eta_{N}(s)\in\mathcal{G}\right\} ds\;.
\]
Let $S^{\mathcal{G}}(t)$ be its generalized inverse:
\[
S^{\mathcal{G}}(t)=\sup\left\{ s\ge0:T^{\mathcal{G}}(s)\le t\right\} \;.
\]
Then, the trace process $\eta_{N}^{\mathcal{G}}(\cdot)$ on $\mathcal{G}$
is defined by
\begin{equation}
\eta_{N}^{\mathcal{G}}(t)=\eta_{N}(S^{\mathcal{G}}(t))\;.\label{etr}
\end{equation}
The trace process $\eta_{N}^{\mathcal{G}}(t)$ on $\mathcal{G}$ is
obtained from $\eta_{N}(t)$ by turning off the clock when $\eta_{N}(\cdot)$
does not belong to $\mathcal{G}$, since $S^{\mathcal{G}}(\cdot)$
freezes the clock when the process $\eta_{N}(\cdot)$ escapes from
$\mathcal{G}$ and turns it back when $\eta_{N}(\cdot)$ returns to
$\mathcal{G}$. Therefore, the process $\eta_{N}^{\mathcal{G}}(\cdot)$
becomes a random process on $\mathcal{G}$ whose trajectory is obtained
from that of $\eta(\cdot)$ by removing its excursions on $\mathcal{G}^{c}$.
Then, it is well known that $\eta_{N}^{\mathcal{G}}(\cdot)$ is a
Markov chain on $\mathcal{G}$ (cf. \cite[Proposition 6.1]{B-L TM}).
\end{defn}

\subsubsection*{Description of movements of condensate}

On the basis of the trace process constructed above, we are now ready
to rigorously formulate the metastable behavior of the inclusion processes.
Denote simply by
\[
\eta_{N}^{\star}(\cdot)=\eta_{N}^{\mathcal{E}_{N}(S_{\star})}(\cdot)
\]
the trace process on the metastable set $\mathcal{E}_{N}(S_{\star})=\{\xi_{N}^{x}:x\in S_{\star}\}$.
For the sake of simplicity, define an identification function $\Psi:\mathcal{E}_{N}(S_{\star})\rightarrow S_{\star}$
as
\[
\Psi(\xi_{N}^{x})=x\text{ for }x\in S_{\star}\;.
\]
Using this function, we define a process $\{Y_{N}(t)\}_{t\ge0}$ on
$S_{\star}$ by
\begin{equation}
Y_{N}(t)=\Psi(\eta_{N}^{\star}(t))\;.\label{yn}
\end{equation}
Thus, the process $Y_{N}(\cdot)$ is obtained by taking the label
of the metastable set at which the process $\eta_{N}^{\star}(\cdot)$
is staying. Since $\eta_{N}^{\star}(\cdot)$ is a Markov chain, the
process $Y_{N}(\cdot)$ is a Markov chain on $S_{\star}$ as well. This special property for the inclusion process is due to the fact that metastable valleys are singletons in our case (cf. Section \ref{sec22}). This fact also helps our proof of the main metastability argument; see the remark after the proof of Proposition \ref{main0}.

Now, the long-time movement of the condensate can be characterized
by proving the convergence of the process $Y_{N}(\cdot)$ with a proper
acceleration factor $\theta_{N}$ to a certain limiting Markov chain
on $S_{\star}$. Let $\{Y(t)\}_{t\ge0}$ denote a continuous-time
Markov chain on $S_{\star}$, which is the candidate for the limiting
Markov chain.
\begin{defn}[Description of metastable behavior]
\label{def21}Suppose that the inclusion process exhibits condensation
in the sense of Definition \ref{def12}. Then, the dynamical movement
of the condensate of an inclusion process is said to be \textbf{\textit{described
by a Markov chain $\{Y(t)\}_{t\ge0}$ on $S_{\star}$ with scale $\theta_{N}$}} (which may not diverge to infinity) if the law of the process $Y_{N}(\theta_{N}\cdot)$ starting from
$\xi_{N}^{x}$ converges to that of $Y(\cdot)$ starting from $x$
for all $x\in S_{\star}$, and if
\begin{equation}
\lim_{N\rightarrow\infty}\sup_{\eta\in\mathcal{E}_{N}(S_{\star})}\mathbb{E}_{\eta}\left[\int_{0}^{T}\mathbf{1}\left\{ \eta_{N}(\theta_{N}s)\notin\mathcal{E}_{N}(S_{\star})\right\} ds\right]=0\text{ for all }T>0\;.\label{neg}
\end{equation}
\end{defn}

\begin{rem*}
Note that the condition \eqref{neg} implies that the inclusion process
does not spend too much time outside the metastable sets and hence
guarantees that there exist only fast transitions between the metastable
sets. In general models, proving \eqref{neg} is not a trivial issue;
however, in the inclusion process case, it directly follows from the
definition of condensation (Definition \ref{def12}), as one can see
from Proposition \ref{pneg}.
\end{rem*}
The main objective of this study is to prove the requirements of Definition
\ref{def21} for a wide class of non-reversible inclusion processes.
We also remark that this has been achieved for reversible inclusion
processes in \cite{B-D-G}. Now, we review this result along with
some conjectures regarding the non-reversible case.

\subsubsection*{Movements of condensate: reversible and non-reversible cases}

Now, we explain the known result and the conjectures for the limiting
chain $Y(\cdot)$ and the factor $\theta_{N}$ appearing in Definition
\ref{def21}.

First, we define a Markov chain $Y^{\mathrm{rv}}(t)$ on $S_{\star}$
(cf. Proposition \ref{p14}) with rate
\begin{equation}
a^{\mathrm{rv}}(x,\,y)=r(x,\,y)\text{ for all }x,\,y\in S_{\star}\;.\label{arv}
\end{equation}
Note that $r(\cdot,\,\cdot)$ is the jump rate of the underlying random
walk $X(\cdot)$; thus, $Y^{\mathrm{rv}}(\cdot)$ can be regarded
as the restricted Markov chain of $X(\cdot)$ on $S_{\star}$. We
also remark that $Y^{\mathrm{rv}}(\cdot)$ is not necessarily an irreducible
chain. Further, we define
\[
\theta_{N}^{\mathrm{rv}}=\frac{1}{d_{N}}\;.
\]
Then, in the terminology of Definition \ref{def21}, the following
result has been established in \cite[Theorem 2.3]{B-D-G}.
\begin{thm}
\label{trev}Suppose that the underlying random walk is reversible
with respect to its invariant measure $\mu_{N}$ and that $\lim d_{N}\log N=0$.
Then, the movement of the condensate is described by a Markov chain
$Y^{\mathrm{rv}}(\cdot)$ on $S_{\star}=S_{\mathrm{max}}$ (cf. \eqref{mmax})
with scale $\theta_{N}^{\mathrm{rv}}$.
\end{thm}

For the non-reversible case, we expect a completely different result
compared to the reversible case. Suppose that we have characterized
the set $S_{\star}$. Define $Y^{\mathrm{nrv}}(\cdot)$ as a Markov
chain on $S_{\star}$ with rate
\begin{equation}
a^{\mathrm{nrv}}(x,\,y)=\left[r(x,\,y)-r(y,\,x)\right]\mathbf{1}\left\{ r(x,\,y)>r(y,\,x)\right\} \text{ for }x,\,y\in S_{\star}\;;\label{anrv}
\end{equation}
and define the time scale as
\begin{equation}
\theta_{N}^{\mathrm{nrv}}=\frac{1}{Nd_{N}}\;.\label{thetanrv}
\end{equation}

\begin{conjecture}
\label{cnrv}Suppose that $\lim d_{N}\log N=0$. Then, the movement
of the condensate is described by the Markov chain $Y^{\mathrm{nrv}}(\cdot)$
with scale $\theta_{N}^{\mathrm{nrv}}$.
\end{conjecture}

The proof of Theorem \ref{trev} for the reversible case obtained
in \cite{B-D-G} is based on the potential theory of reversible Markov
chains. Hence, it is tempting to adopt the recently developed potential
theory of non-reversible Markov chains \cite{G-L,S} to investigate
the non-reversible case. Indeed, we are able to do so in the case
if the invariant measure $\mu_{N}(\cdot)$ admits the formula \eqref{mundef}.
However, instead of following this traditional approach, we try to
directly estimate the so-called mean-jump rate by exploiting several
model-dependent features of the inclusion process. This is mainly
because we wish to tackle the general case without the formula \eqref{mundef}
on $\mu_{N}$. Indeed, one of the main difficulties in the study of
the non-reversible case is the lack of such an explicit formula for
$\mu_{N}$; in this case, it is even unclear what $S_{\star}$ is.
Specifying $S_{\star}$ itself seems to be an extremely difficult
problem.
\begin{rem}
\label{rmkrvnrv}In general, it is anticipated that the metastable
transition of non-reversible dynamics occurs faster than that of its
reversible counterpart. For instance, such a phenomenon has been verified
for the stochastic discrete gradient descent \cite{L-Seo}, small
random perturbation of dynamical systems \cite{L-Mar-Seo}, and zero-range
processes \cite{B-L ZRP,L TAZRP,Seo NRZRP}. These results show that
the non-reversible dynamics is faster than the reversible one by a
constant (i.e., $O(1)$) factor, while Conjecture \ref{cnrv} indicates
that the non-reversible dynamics of the inclusion process is expected
to be $O(N)$ times faster than the reversible one. This observation is supported by \cite[Section 4.3]{C-C-G} which performed heuristic computations for the inclusion process on one-dimensional tori in the thermodynamic limit.
\end{rem}

Finally, suppose that the relation $r(x,\,y)=r(y,\,x)$ holds for
all $x,\,y\in S_{\star}$. In this case, we have $a^{\mathrm{nrv}}(x,\,y)=0$
for all $x,\,y\in S_{\star}$; hence, Conjecture \ref{cnrv} implies
that the scale $\theta_{N}^{\mathrm{nrv}}=\frac{1}{Nd_{N}}$ is too
short to observe the transitions. We expect that the correct scale
for this case is $\theta_{N}^{\mathrm{rv}}=\frac{1}{d_{N}}$.
\begin{conjecture}
\label{cnrv2}Suppose that $r(x,\,y)=r(y,\,x)$ for all $x,\,y\in S_{\star}$
and $\lim d_{N}\log N=0$. Then, the movement of the condensate
is described by the Markov chain $Y^{\mathrm{rv}}(\cdot)$ on $S_{\star}$
with scale $\theta_{N}^{\mathrm{rv}}$.
\end{conjecture}

Here, we emphasize that Theorem \ref{trev} is a special case of this
conjecture. To see this, observe that $S_{\star}=S_{\mathrm{max}}$
for the reversible case; thus, we have
\[
r(x,\,y)=\frac{m(y)}{m(x)}r(y,\,x)=\frac{M_{*}}{M_{*}}r(y,\,x)=r(y,\,x)\;\;\;\text{for all }x,\,y\in S_{\star}\;.
\]
This implies that, if the previous conjecture is true, the scale $\theta_{N}^{\mathrm{rv}}$
and the limiting Markov chain $Y^{\mathrm{rv}}(\cdot)$ appear in
the reversible case because $r(\cdot,\,\cdot)$ is symmetric on $S_{\star}=S_{\mathrm{max}}$,
and the reversibility is not a fundamental reason.

In this study, we verify the validity of Conjectures \ref{cnrv} and
\ref{cnrv2} for wide-class of non-reversible inclusion processes.

\subsubsection*{Comments on the convergence of finite-dimensional distributions}

Before proceeding to the main results of this article, we remark on
the mode of convergence regarding Definition \ref{def21}. Although
the convergence of the trace process is natural in the study of metastability,
an alternative description has been presented \cite{L-L-M}, which
does not need to recall the trace process in the description and is
hence more intuitive to understand. To see this, fix a cemetery state $\mathfrak{o}$ and define a map $\widehat{\Psi}:\mathcal{H}_{N}\rightarrow S_{\star}\cup\{\mathfrak{o}\}$
as
\[
\widehat{\Psi}(\eta)=\begin{cases}
x & \text{if }\eta=\xi_{N}^{x}\text{ with }x\in S_{\star}\;,\\
\mathfrak{o} & \text{otherwise\;.}
\end{cases}
\]
Then, define a process $\{\widehat{Y}_{N}(t)\}_{t\ge0}$ on $S_{\star}\cup\{\mathfrak{o}\}$
by
\[
\widehat{Y}_{N}(t)=\widehat{\Psi}(\eta_{N}(t))\;.
\]
In other words, we trace each metastable configuration to its label and all the other configurations to the cemetery state $\mathfrak{o}$.

\begin{defn}
\label{def25}The dynamical movement of the condensate of an inclusion
process is said to be \textbf{\textit{described by a Markov chain
$\{Y(t)\}_{t\ge0}$ on $S_{\star}$ with scale $\theta_{N}$ in the
finite-dimensional marginal sense}} if, for all $k\in\mathbb{N}$,
we have
\[
\lim_{N\rightarrow\infty}\mathbb{P}_{\xi_{N}^{x}}\left[\widehat{Y}_{N}(\theta_{N}t_{1})\in A_{1},\,\dots,\,\widehat{Y}_{N}(\theta_{N}t_{k})\in A_{k}\right]=\mathbf{P}_{x}\left[Y(t_{1})\in A_{1},\,\dots,\,Y(t_{k})\in A_{k}\right]
\]
for all $0\le t_{1}<\cdots<t_{k}$ and $A_{1},\,\dots,\,A_{k}\subseteq S_{\star}$,
where $\mathbf{P}_{x}$ denotes the law of the process $Y(\cdot)$
starting from $x$.
\end{defn}

To establish this convergence of marginal distributions from that
of the trace process defined in Definition \ref{def21}, it is known
from \cite[Proposition 2.1]{L-L-M} that the verification of the following
technical condition is sufficient:
\begin{equation}
\lim_{\delta\rightarrow0}\limsup_{N\rightarrow\infty}\sup_{2\delta\le s\le3\delta}\sup_{\eta\in\mathcal{E}_{N}(S_{\star})}\mathbb{P}_{\eta}\left[\eta_{N}(\theta_{N}s)\notin\mathcal{E}_{N}(S_{\star})\right]=0\;.\label{neg2}
\end{equation}
For the inclusion process, this condition is straightforward to check
(cf. Proposition \ref{pneg}); thus, the convergence of the trace
process immediately implies the convergence of the finite-dimensional
distributions.

\subsection{\label{sec32}Main result 1: inclusion processes under condition
(UI)}

In this subsection, we explain our result of the analysis of the metastable
behavior of the inclusion process under the condition \textbf{(UI)}.
For this case, as mentioned in Proposition \ref{p13}, the invariant
measure admits the expression \eqref{mundef}; therefore condensation
occurs owing to Proposition \ref{p14}. Moreover, as the invariant
measure for the underlying random walk is uniform, we have $S_{\star}=S_{\mathrm{max}}=S$,
i.e., condensation occurs on the entire state set $S$.

The metastable behavior of the inclusion process for this case was
known only when $r(\cdot,\,\cdot)$ is completely symmetric (as in
case (1) of the theorem below). The following theorem extends this
result for the general case under \textbf{(UI)}.
\begin{thm}
\label{t21}Suppose that the underlying random walk satisfies the
condition \textbf{\textup{(UI)}} and that $\lim_{N\rightarrow\infty}d_{N}\log N=0$.
\begin{enumerate}
\item Suppose that $r(x,\,y)=r(y,\,x)$ for all $x,\,y\in S$. Then, Conjecture
\ref{cnrv2} holds.
\item Suppose that $r(x,\,y)\neq r(y,\,x)$ for some $x,\,y\in S$. Then,
Conjecture \ref{cnrv} holds.
\end{enumerate}
\end{thm}

We remark that, for case (1), the underlying random walk is reversible;
hence, this result is a consequence of \cite{B-D-G} (i.e., of Theorem
\ref{trev} of the current paper). Our new result focuses on case
(2), which is essentially the first rigorous analysis of the metastable
behavior of non-reversible inclusion processes. The proof of this
result relies on careful analysis of the mean-jump rates established
in Section \ref{sec4}. We explain the proof in Section \ref{sec5}.

\subsubsection*{Inclusion processes on torus}

An interesting example satisfying condition \textbf{(UI)} is the simple
random walk on the discrete torus. Suppose that the underlying random
walk is a simple random walk on the torus $\mathbb{T}_{L}=\mathbb{Z}/(L\mathbb{Z})$
with jump rate
\[
r(x,\,y)=\begin{cases}
p & \text{if }y=x+1\text{ (mod }L)\;,\\
1-p & \text{if }y=x-1\text{ (mod }L)\;,\\
0 & \text{otherwise\;.}
\end{cases}
\]
As the uniform measure on $\mathbb{T}_{L}$ is the invariant measure
for this random walk, the condition \textbf{(UI)} is valid. We can
prove that the dynamical transition of the condensate can be described
as follows. For the simplicity we may assume that $p\ge1/2$ since
the case $p\le1/2$ can be treated in the same manner.
\begin{cor}
Suppose that $\lim_{N\rightarrow\infty}d_{N}\log N=0$. Then, the
dynamical movement of the condensate for the inclusion process on
$\mathbb{T}_{L}$ defined above is described by the following limiting
Markov chain and the time scale:
\begin{enumerate}
\item for $p=1/2$, a Markov chain $\{Y^{\mathrm{sym}}(t)\}_{t\ge0}$ with
jump rate
\[
a^{\mathrm{sym}}(x,\,y)=\begin{cases}
1/2 & \text{if }|y-x|=1\;,\\
0 & \text{otherwise\;,}
\end{cases}
\]
and scale $\theta_{N}^{\mathrm{rv}}=1/d_{N}$.
\item for $p>1/2$, a Markov chain $\{Y^{\mathrm{asym}}(t)\}_{t\ge0}$ with
jump rate
\[
a^{\mathrm{asym}}(x,\,y)=\begin{cases}
2p-1 & \text{if }y=x+1\;,\\
0 & \text{otherwise\;,}
\end{cases}
\]
and scale $\theta_{N}^{\mathrm{nrv}}=1/(Nd_{N})$.
\end{enumerate}
\end{cor}

We note that the transition scale for the asymmetric case is $1/(Nd_{N})$,
and it is $O(N)$ times faster than that of the symmetric case, i.e.,
$1/d_{N}$. This observation verifies the statement in Remark \ref{rmkrvnrv}.
Furthermore, it is interesting that the limiting dynamics for the
partially asymmetric case (i.e., $p\in(1/2,\,1)$) is totally asymmetric.

\subsection{\label{sec33}Main result 2: inclusion processes with uniformly positive
rates}

As mentioned earlier in the introduction, the condensation of inclusion
processes without condition\textbf{ (Rev)} or \textbf{(UI)} is unknown.
For instance, whether condensation occurs on $S$, i.e., $\lim_{N\rightarrow\infty}\mu_{N}(\mathcal{E}_{N})=1$,
is an open question. This is mainly because of the lack of the explicit
formula of $\mu_{N}$. Under suitable assumptions, we now describe
both static and dynamical analyses of condensation in such general
cases.

\subsubsection*{Metastable behavior for general non-reversible inclusion processes}

We assume first that the occurrence of the condensation has been verified,
and then focus on the analysis of the metastable behavior. We will
return to the condensation issue later in this subsection.

To prove Conjecture \ref{cnrv}, we should first characterize $S_{\star}$.
To this end, let us consider an auxiliary Markov chain $(Z_{1}(t))_{t\ge0}$
on $S$ with jump rate
\begin{equation}
b(x,\,y)=\left[r(x,\,y)-r(y,\,x)\right]\mathbf{1}\left\{ r(x,\,y)>r(y,\,x)\right\} \text{ for all }x,\,y\in S\;,\label{bxy}
\end{equation}
which is an extension of $a^{\mathrm{nrv}}(\cdot,\,\cdot)$ defined
in \eqref{anrv} to the set $S$. Let $S_{0}$ denote the set of recurrent
states (including absorbing states; refer to Figure \ref{fig2}) of
the Markov chain $Z_{1}(\cdot)$. We say that \textit{$S_{0}$ has
only one irreducible component }if the Markov chain $Z_{1}(\cdot)$
restricted to $S_{0}$ is irreducible, i.e., for any $x,\,y\in S_{0}$,
there exists some $k\ge1$ such that
\[
\sum_{z_{1},\,\dots,\,z_{k-1}\in S_{0}}b(x,\,z_{1})b(z_{1},\,z_{2})\cdots b(z_{k-1},\,y)>0\;.
\]
This assumption is equivalent to the uniqueness of the invariant measure
for $Z_{1}(\cdot)$, and for such a case $S_{0}$ is the support of
the invariant measure. Then, the following result describes the metastable
behavior of the inclusion process when $S_{0}$ has only one irreducible
component.

\begin{figure}
\includegraphics[scale=0.21]{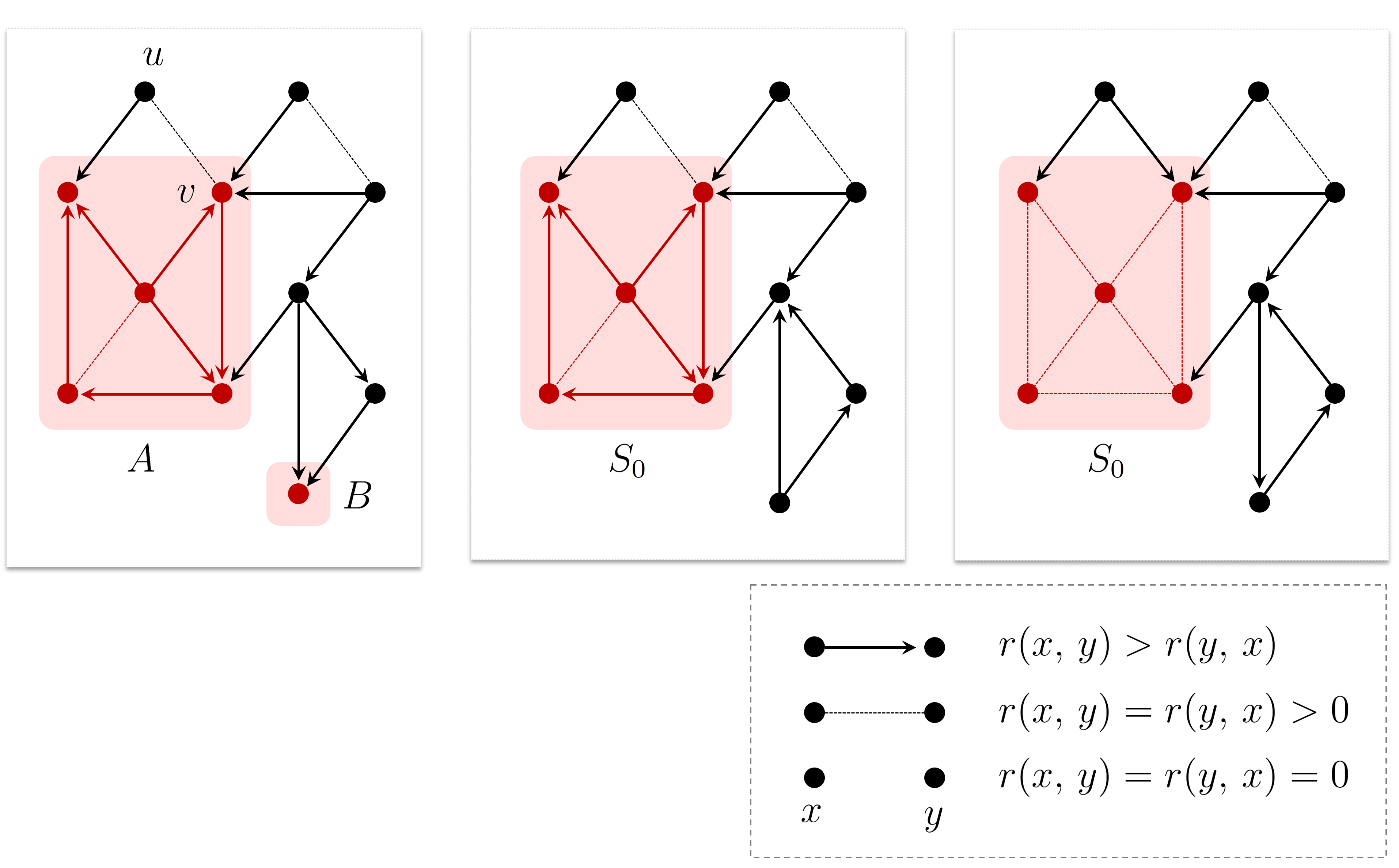}\caption{\label{fig2}(Left) The set $S_{0}$ is given by $S_{0}=A\cup B.$
In this case, $Z_{1}(\cdot)$ restricted to $S_{0}$ has two irreducible
components $A$ and $B$; thus, it does not satisfy the condition
of Theorem \ref{t213}. The set $A$ is semi-attracting since $r(u,\,v)=r(v,\,u)$.
(Middle) The set $S_{0}$ satisfies the condition of Theorem \ref{t213},
since $S_{0}$ has only one irreducible component with respect to
$Z_{1}(\cdot)$. (Right) The set $S_{0}$ is attracting; hence $S_{0}$
satisfies all the conditions of Theorem \ref{t214}. }
\end{figure}

\begin{thm}
\label{t213}Suppose that condensation occurs and that $S_{0}$ defined
above has only one irreducible component. Then, $S_{\star}=S_{0}$
and Conjecture \ref{cnrv} holds.
\end{thm}

Now, we turn to Conjecture \ref{cnrv2}. To this end, we assume that
$r(x,\,y)=r(y,\,x)$ for all $x,\,y\in S_{0}$. Then, consider another
auxiliary Markov chain $(Z_{2}(t))_{t\ge0}$ on $S_{0}$ whose rate
between $x\in S_{0}$ and $y\in S_{0}$ is just $r(x,\,y)$. We need
to introduce additional simple concepts to state our result.
\begin{notation}
\label{n32}The set $A\subseteq S$ is called \textit{attracting}
if it holds that $r(x,\,y)<r(y,\,x)$ for all $x\in A$ and $y\in A^{c}$
with $r(x,\,y)+r(y,\,x)>0$. Moreover, $A$ is called \textit{semi-attracting}
if it holds that $r(x,\,y)\le r(y,\,x)$ for all $x\in A$ and $y\in A^{c}$
with $r(x,\,y)+r(y,\,x)>0$. We refer to Figure \ref{fig2} for the
illustration.
\end{notation}

Note that attracting sets are semi-attracting as well. For the symmetric
case, we obtain the following result.
\begin{thm}
\label{t214}Suppose that condensation occurs and that the Markov
chain $(Z_{2}(t))_{t\ge0}$ on $S_{0}$ defined above is irreducible.
Further, assume that $S_{0}$ is attracting. Then, $S_{\star}=S_{0}$
and Conjecture \ref{cnrv2} holds.
\end{thm}

\begin{rem*}
Since $S_{\star}=S_{0}$ in this case, the Markov chain $Z_{2}(\cdot)$
is indeed $Y^{\mathrm{rv}}(\cdot)$. The condition that $S_{0}$ is
attracting is required to guarantee that $S_{0}$ is the set of states
at which the transition occurs.
\end{rem*}
As a consequence of Theorems \ref{t213} and \ref{t214}, we can provide
the following non-trivial asymptotic limit of $\mu_{N}(\xi_{N}^{x})$
for $x\in S_{\star}=S_{0}$.
\begin{thm}
\label{t215}Under the conditions of Theorem \ref{t213} (resp. Theorem
\ref{t214}), it holds that
\[
\lim_{N\rightarrow\infty}\mu_{N}(\xi_{N}^{x})=\nu(x)\text{ for all }x\in S_{\star}
\]
where $\nu(\cdot)$ is the unique invariant measure of the irreducible
Markov chain $Y^{\mathrm{nrv}}(\cdot)$ (resp. $Y^{\mathrm{rv}}(\cdot)$).
\end{thm}

\begin{rem}
Several remarks regarding the irreducibility of $Z_{1}(\cdot)$ and
$Z_{2}(\cdot)$ on $S_{0}$ are stated below.
\begin{enumerate}
\item When there exist multiple irreducible components of $Z_{1}(\cdot)$
on $S_{0}$, a certain linear combination of the invariant measure
on each component is expected to equal the limit of $\mu_{N}$ on
$\mathcal{E}_{N}$. However, at this moment, it is unclear as to which
linear combination is the correct one. Moreover, characterizing $S_{\star}$
is not possible at this moment. The sites in $S\setminus S_{0}$ will
be discarded in the long-time limit; however, it is unclear as to
which sites of $S_{0}$ will survive, partially or completely, in
the accelerated process. We shall not pursue this problem in the present
study, and is left as a topic for future research.
\item The reversible case in which there exist multiple irreducible components
of $Z_{2}(\cdot)$ on $S_{0}$ has been investigated in \cite{B-D-G}
for a specific form of the underlying graph. In these longer scaling
limits, each irreducible component is expected to act as a single
element in the limiting dynamics, and the long-time movement will
occur among these component-wise elements. If the graph distance between
these components is exactly $2$, then the transition occurs in the
second scale $N/d_{N}^{2}$. If the distance is greater than $2$,
then the transition occurs in the third scale $N^{2}/d_{N}^{3}$.
However, such generality has not been analyzed even for the reversible
inclusion process on general graphs.
\end{enumerate}
\end{rem}

\subsubsection*{Condensation}

Previously, we analyzed the metastable behavior of inclusion processes
by assuming that the condensation occurs. However, without the closed-form
expression for the invariant measure, the verification of the condensation
is not a simple task. Here, we prove the existence of condensation
under the following assumption:

\medskip{}

\noindent \textbf{(UP)} The jump rate of the underlying random walk
is\textit{ uniformly positive }in the sense that
\begin{equation}
r(x,\,y)>0\text{ for all }x,\,y\in S\;.\label{upr}
\end{equation}

\medskip{}

With this assumption, we can establish the existence of condensation
for inclusion processes. We emphasize that this is the first verification
of the condensation for the inclusion process without explicit formula
\eqref{mundef} for $\mu_{N}$.
\begin{thm}
\label{t311}Suppose that the assumption \textbf{\textup{(UP)}} holds
and
\begin{equation}
\lim_{N\rightarrow\infty}d_{N}N^{|S|+2}(\log N)^{|S|-3}=0\;.\label{fcv}
\end{equation}
Then, the condensation occurs for the inclusion process, i.e.,
\begin{equation}
\lim_{N\rightarrow\infty}\mu_{N}(\mathcal{E}_{N})=1\;.\label{trrposmaineq}
\end{equation}
\end{thm}

\begin{rem}
We remark that the stringent condition \eqref{fcv} appeared because the estimates used in Section \ref{sec7} are partially sub-optimal. We conjecture that Theorem \ref{t311} still holds when \eqref{fcv} is substituted by the standard condition $\lim_{N\rightarrow\infty}d_N\log N=0$ without changing the current setting. We expect that refining the arguments carried out in Section \ref{sec7} regarding the analysis of the inner core (cf. Notation \ref{no61}) of the configuration space is crucial to get an optimal result. We believe that a totally different idea is required to get such an optimal result.
\end{rem}

The following corollary is now immediate.
\begin{cor}
Theorems \ref{t213}, \ref{t214}, and \ref{t215} hold under the
conditions \textbf{\textup{(UP)}} and \eqref{fcv}.
\end{cor}

The proof of Theorem \ref{t311} is given in Section \ref{sec7} and
relies on the results on mean-jump rates established in Section \ref{sec4}
along with a weak result on the nucleation of condensation stated
below in Theorem \ref{t211}.

In general, the nucleation regime explains the typical behavior of
particles, starting from an arbitrary distribution among sites to
condensation at a sole site. The only rigorous result regarding the
nucleation was obtained in \cite{Gr-R-V 13}, where it was proved
that the nucleation procedure of the inclusion process satisfying
both \textbf{(Rev)} and\textbf{ (UI) }can be explained by a Wright--Fisher-type
slow-fast diffusion. We refer to \cite{Gr-R-V 13} for further information
on nucleation; although our nucleation result explained hereafter
is much weaker, it is the first quantitative result in the study of
nucleation of non-reversible inclusion processes. For $\mathcal{A}\subseteq\mathcal{H}_{N}$,
let $\tau_{\mathcal{A}}=\tau_{\mathcal{A}}^{N}$ denote the hitting
time of the set $\mathcal{A}$ with respect to the inclusion process
$\eta_{N}(\cdot)$, and let $\delta>0$ be an arbitrary fixed number.
Define
\[
\mathcal{U}_{N}=\{\eta\in\mathcal{H}_{N}:\eta_{x}\le\delta\log N\text{ for some }x\in S\}\;.
\]
Then, the nucleation result can be formulated as follow.
\begin{thm}
\label{t211}Suppose that the assumption \textbf{\textup{(UP)}} holds
and $\lim_{N\rightarrow\infty}d_{N}\frac{N^{2}}{(\log N)^{2}}=0$.
Then, there exists a constant $C=C(\delta)>0$ such that
\[
\sup_{\eta\in\mathcal{H}_{N}}\mathbb{E}_{\eta}\left[\tau_{\mathcal{U}_{N}}\right]\le CN\;.
\]
\end{thm}

Suppose that the inclusion process starts from a configuration containing
$\Omega(N)$\footnote{A number asymptotically lying between $c_{1}N$ and $c_{2}N$.}
particles at all sites. Then, the first stage of the nucleation of
condensation is to empty a site, which can be deduced by studying
the typical path to the set $\{\eta\in\mathcal{H}_{N}:\eta_{x}=0\text{ for some }x\in S\}$
and examining the mean of the hitting time. The theorem above provides
a weak form of such a result, and its proof will be given in Section
\ref{sec74}. It is strongly expected that the actual scale of the
nucleation of particles is $O(\log N)$, which serves as an important
topic of future research.

\subsection{\label{sec34}Main result 3: inclusion processes in the thermodynamic
limit regime}

In the previous models, we fixed the state space $S$. In this subsection,
we consider a slightly different model for which the space given by
the multi-dimensional discrete torus grows together with the number
of particles. Then, a suitable time-space rescaling of the movements
of the condensate converges to a \textit{continuous process }on a
multi-dimensional torus; this type of result is referred to as the
\textit{thermodynamic limit of condensation }(cf. \cite{A-G-L}).

The thermodynamic limit of condensation has been thoroughly studied
for zero-range processes in \cite{A-G-L,Rez-Seo NRZRP}. In \cite{A-G-L},
the thermodynamic limit of condensation of the symmetric zero-range
process on the torus has been investigated by the martingale approach,
and in \cite{Rez-Seo NRZRP}, it has been generalized to the asymmetric
zero-range process on the multi-dimensional torus via a new approach
based on the solution of a Poisson equation. For the simple inclusion
process, the thermodynamic limit of the inclusion process whose underlying
random walk is either a symmetric or totally asymmetric random walk
on the one-dimensional torus has been investigated in \cite{C-C-G}.
The authors used exquisitely constructed heuristic simulations to
derive various time scales related to the nucleation regime of the
process, which is divided into four parts: nucleation, coarsening,
saturation, and stationary. Readers may refer to \cite{C-C-G} for
further details.

Our contribution to the study of condensation in the thermodynamic
limit regime is \textbf{\textit{to establish the scaling limit of
the movement of condensation, and we find three different time scales
according to the level of asymmetry}}\textit{.} We explain these results
in the remainer of this subsection.

\subsubsection*{Model}

We start by introducing our model, which is distinguished from previous
models by the characteristic that the underlying state space is growing.
Recall that $\mathbb{T}_{L}=\mathbb{Z}/L\mathbb{Z}$ denotes a discrete
torus of length $L$. Now, we consider the inclusion process consisting
of $N$ interacting particles that move according to a random walk
on the multi-dimensional torus $\mathbb{T}_{L}^{d}$ where $L$ and
$N$ grow together such that
\begin{equation}
L\rightarrow\infty\;,\quad N=N_{L}\rightarrow\infty\;,\;\text{and}\quad\frac{N}{L^{d}}\rightarrow\rho\;\text{ for some }\rho>0\;.\label{LNlim}
\end{equation}
Henceforth, we assume that $\rho>0$ is fixed and regard $N$ as a
variable that is dependent on $L$; hence, the only control variable
is $L$. With this convention, the condition \eqref{LNlim} implies
that the total density is maintained to be close to $\rho$ as $L\rightarrow\infty$.

To get a scaling limit, we will assume that the underlying system
is a translation-invariant random walk on $\mathbb{T}_{L}^{d}$, i.e.,
the jump rate of the underlying random walk on $\mathbb{T}_{L}^{d}$
is given by
\begin{equation}
r(x,\,y)=h(y-x)\label{trate}
\end{equation}
for some non-negative function $h:\mathbb{Z}^{d}\rightarrow[0,\,\infty)$
with compact support, i.e., there exists $M>0$ such that $h(x)=0$
if $|x|>M$. We assume that this random walk is irreducible, i.e.,
the support of $h$ spans $\mathbb{Z}^{d}$.
\begin{rem}
\label{remt}Now, we state several remarks on this model:
\begin{enumerate}
\item It should be emphasized that the simple nearest-neighbor random walk
on $\mathbb{T}_{L}^{d}$ is an example of the translation-invariant
random walk.
\item By the translation invariance, it can immediately be verified that
the random walk satisfies the condition \textbf{(UI)}, i.e., the invariant
measure $m$ of the underlying random walk is the uniform measure
on $\mathbb{T}_{L}^{d}$. Moreover, this random walk is reversible
with respect to this invariant measure only when the function $h$
is \textit{symmetric}, i.e., $h(x)=h(-x)$ for each $x\in\mathbb{Z}^{d}$.
\item Throughout the remainder of this subsection, we shall implicitly assume
$L>2M$ so that the state space $\mathbb{T}_{L}^{d}$ is much larger
than the support of $h$.
\end{enumerate}
\end{rem}

The inclusion process $\{\eta_{L}(t)\}_{t\ge0}$ on $\mathbb{T}_{L}^{d}$
consisting of $N$ particles where $N$ and $L$ satisfy \eqref{LNlim}
is defined as a continuous-time Markov chain on the configuration
space given by
\[
\mathcal{H}_{L}=\bigg\{\eta\in\mathbb{N}^{\mathbb{T}_{L}^{d}}:\sum_{x\in\mathbb{T}_{L}^{d}}\eta_{x}=N\bigg\}\;.
\]
If the inclusion process consists of the translation-invariant underlying
random walks described above, then the generator corresponding to
the inclusion process is defined, for $\mathbf{f}:\mathcal{H}_{L}\rightarrow\mathbb{R}$,
by
\[
(\mathscr{L}_{L}\mathbf{f})(\eta)=\sum_{x,\,y\in\mathbb{T}_{L}^{d}}\eta_{x}(d_{L}+\eta_{y})r(x,\,y)\{\mathbf{f}(\sigma^{x,\,y}\eta)-\mathbf{f}(\eta)\}\;;\quad\eta\in\mathcal{H}_{L}\;,
\]
where $\{d_{L}\}_{L=1}^{\infty}$ is a sequence of positive real numbers
converging to $0$. Let $\mathbb{P}_{\eta}^{L}$ and $\mathbb{E}_{\eta}^{L}$
denote the law and expectation with respect to the process $\eta_{L}(\cdot)$
starting at $\eta$, respectively.

\subsubsection*{Condensation}

We are primarily interested in the limiting behavior of the condensate
of the model explained above as $L$ tends to infinity. As before,
define the metastable set corresponding to the condensation of the
inclusion process as
\[
\mathcal{E}_{L}^{x}=\{\xi_{L}^{x}\}\text{ for each }x\in\mathbb{T}_{L}^{d}\;,
\]
where $\xi_{L}^{x}$ denotes the configuration containing all the
particles at site $x\in\mathbb{T}_{L}^{d}$. Write
\begin{equation}
\mathcal{E}_{L}=\bigcup_{x\in\mathbb{T}_{L}^{d}}\mathcal{E}_{L}^{x}\;.\label{tel}
\end{equation}
Let $\mu_{L}(\cdot)$ denote the invariant measure for this model.
As this model satisfies the condition \textbf{(UI)} as mentioned in
(2) of Remark \ref{remt}, we can use Proposition \ref{p13} to write
the invariant measure as
\begin{equation}
\mu_{L}(\eta)=\frac{1}{Z_{L}}\prod_{x\in\mathbb{T}_{L}^{d}}w_{L}(\eta_{x})\;,\quad\eta\in\mathcal{H}_{L}\;,\label{tinv}
\end{equation}
where
\[
w_{L}(n)=\frac{\Gamma(n+d_{L})}{n!\Gamma(d_{L})}\;,\quad n\in\mathbb{N}\text{ and }Z_{L}=\sum_{\eta\in\mathcal{H}_{L}}\prod_{x\in\mathbb{T}_{L}^{d}}w_{L}(\eta_{x})\;.
\]
Owing to this expression, we can prove the occurrence of condensation
provided that $d_{L}$ converges to $0$ sufficiently fast.
\begin{thm}
\label{tlmain}Suppose that $\lim_{L\rightarrow\infty}d_{L}L^{d}\log L=0$.
Then, we have
\[
\lim_{L\rightarrow\infty}\mu_{L}(\mathcal{E}_{L})=1\;.
\]
Consequently, by the symmetry of the invariant measure \eqref{tinv},
we have
\[
\mu_{L}(\mathcal{E}_{L}^{x})=(1+o_{L}(1))\frac{1}{L^{d}}\text{ for all }x\in\mathbb{T}_{L}^{d}\;.
\]
\end{thm}

We remark that this result has been recently proved in \cite[Proposition 2]{C-G-J}
using the technique of size-biased sampling. However, we propose an
alternative proof of this theorem in Section \ref{sec81} for the
completeness of the article.

\subsubsection*{Description of metastable behavior}

Now, we turn to the dynamics of the condensate. In this model, we
rescale the state space so that we can identify $x\in\mathbb{T}_{L}^{d}$
as a point $L^{-1}x\in\mathbb{T}^{d}$. By rescaling the time appropriately,
we expect the dynamics of the condensate to converge to a process
on $\mathbb{T}^{d}$ as $L\rightarrow\infty$. Our result presented
below verifies that three different time scales appear according to
the level of asymmetry of the underlying random walk. To rigorously
formulate this result, we start by defining a map $\Theta_{L}:\mathcal{E}_{L}\rightarrow\mathbb{T}^{d}$
by
\[
\Theta_{L}(\xi_{L}^{x})=\frac{x}{L}\;,\quad x\in\mathbb{T}_{L}^{d}\;.
\]
Define a process $\{Y_{L}(t)\}_{t\ge0}$ on $\mathbb{T}^{d}$ by
\[
Y_{L}(t)=\Theta_{L}(\eta_{L}^{\mathcal{E}_{L}}(t))\;,
\]
where $\eta_{L}^{\mathcal{E}_{L}}(\cdot)$ is the trace process of
$\eta_{L}(\cdot)$ on the set $\mathcal{E}_{L}$. The following is
a variant of Definition \ref{def21}.
\begin{defn}
\label{def218}The movement of the condensate of the inclusion process
on $\mathbb{T}_{L}^{d}$ defined above is said to \textbf{\textit{be
described by a process $\{Y(t)\}_{t\ge0}$ on $\mathbb{T}^{d}$ with
scale $\theta_{L}$}} if the following conditions hold simultaneously.
\begin{enumerate}
\item For each sequence $(x_{L})_{L=1}^{\infty}$ such that $x_{L}\in\mathbb{T}_{L}^{d}$
for all $L\ge1$ and $\lim_{L\rightarrow\infty}(x_{L}/L)=u$, the
law of the rescaled trace process $Y_{L}(\theta_{L}\cdot)$ starting
from $\xi_{L}^{x_{L}}$ converges to that of the process $Y(\cdot)+u$
on $\mathbb{T}^{d}$.
\item The excursions outside $\mathcal{E}_{L}$ are negligible at the scale
$\theta_{L}$ in the sense that
\begin{equation}
\lim_{L\rightarrow\infty}\sup_{\eta\in\mathcal{E}_{L}}\mathbb{E}_{\eta}^{L}\left[\int_{0}^{T}\mathbf{1}\left\{ \eta_{L}(\theta_{L}s)\notin\mathcal{E}_{L}\right\} ds\right]=0\text{ for all }T>0\;.\label{tneg}
\end{equation}
\end{enumerate}
\end{defn}

\subsubsection*{Main results for thermodynamic limit of metastable behavior}

Let $v$ denote the mean displacement (hence, the velocity) of the
underlying random walk:
\[
v=\sum_{y\in\mathbb{Z}^{d}}h(y)y\;.
\]
We decompose the model into three cases as follows:
\begin{enumerate}
\item If $v\neq0$, the model is referred to as\textit{ totally asymmetric}.
\item If $v=0$ and $h$ is not symmetric, then the model is referred to
as \textit{mean-zero asymmetric}.
\item If $v=0$ and $h$ is symmetric, then the model is referred to as
\textit{symmetric}.
\end{enumerate}
Then, the relevant time scales for these three cases are different,
as we will see below. The following is the first main result.
\begin{thm}[The first time scale for the totally asymmetric case]
\label{tt1}Suppose that $v\ne0$ and assume that $\lim d_{L}L^{d+1}\log L=0$.
Then, the movement of the condensate of the inclusion process on $\mathbb{T}_{L}^{d}$
is described by the deterministic motion $V(t)=\rho vt$ with scale
$\theta_{L}=1/(d_{L}L^{d-1})$.
\end{thm}

Note that the limiting dynamics $V(t)$ obtained in the last theorem
is non-degenerate only when the dynamics is totally asymmetric, i.e.,
$v\neq0$. Hence, if $v=0$, we have to wait for more time to observe
the transitions of the condensation. Now, we formulate this result
in a rigorous form. For each $y\in\mathbb{R}^{d}$, let $y\otimes y$
denote the outer product, i.e., $y\otimes y=yy^{\dagger}$.\footnote{Given a matrix $A$, let $A^{\dagger}$ denote the transpose of $A$.}
Hence, $y\otimes y$ is a $d\times d$ matrix. Consider a non-negative
symmetric matrix $\mathbb{S}_{1}$ given by
\[
\mathbb{S}_{1}=\rho\sum_{y\in\mathbb{Z}^{d}:h(y)>h(-y)}(h(y)-h(-y))\,y\otimes y
\]
and let $\Sigma_{1}$ denote its square root.\footnote{Let $U^{\dagger}\Lambda U$ denote the diagonalization of the symmetric
matrix $\mathbb{S}_{1}$, where $\Lambda=\text{diag}(\lambda_{1},\,\dots,\,\lambda_{d}).$
Define $\Lambda^{1/2}=\text{diag}(\lambda_{1}^{1/2},\,\dots,\,\lambda_{d}^{1/2})$
which is well defined since $\mathbb{S}_{1}$ is non-negative definite.
Then, $\Sigma_{1}$ is defined by $U^{\dagger}\Lambda^{1/2}U$. Note
that $\Sigma_{1}\Sigma_{1}=\mathbb{S}_{1}$.}
\begin{thm}[The second time scale for the mean-zero asymmetric case]
\label{tt2}Suppose that $v=0$ and assume that $\lim d_{L}L^{d+2}\log L=0$.
Then, the movement of the condensate of the inclusion process on $\mathbb{T}_{L}^{d}$
is described by the Brownian motion with diffusion matrix $\Sigma_{1}$
and scale $\theta_{L}=1/(d_{L}L^{d-2})$.
\end{thm}

This theorem explains the diffusive behavior of condensation when
the underlying random walk is mean-zero such that the local drift
at the time scale $1/(d_{L}L^{d-1})$ is canceled out. However, note
that the matrix $\mathbb{S}_{1}$, and hence $\Sigma_{1}$ is a zero
matrix when the underlying random walk is symmetric. This indicates
that we still have to wait for more time to observe the macroscopic
movements of the condensate for the symmetric case. Indeed, we should
wait for much longer to observe these movements. To formulate this,
define a positive definite matrix $\mathbb{S}_{2}$ by
\[
\mathbb{S}_{2}=\sum_{y\in\mathbb{Z}^{d}}h(y)\,y\otimes y\;,
\]
and let $\Sigma_{2}$ denote its square root.
\begin{thm}[The third time scale for the symmetric case]
\label{tt3}Suppose that $h(x)=h(-x)$ for all $x\in\mathbb{Z}^{d}$
and assume that $\lim d_{L}L^{2d+2}\log L=0$. Then, the movement
of the condensate of the inclusion process on $\mathbb{T}_{L}^{d}$
is described by the Brownian motion with diffusion matrix $\Sigma_{2}$
and scale $\theta_{L}=L^{2}/d_{L}$.
\end{thm}

The proofs of Theorems \ref{tt1}, \ref{tt2}, and \ref{tt3} are
given in Section \ref{sec8}. We conclude this section with several
remarks on these theorems regarding the metastable behavior of the
inclusion process in thermodynamic limit regime.
\begin{rem}
$ $
\begin{enumerate}
\item It should be noted that the limiting particle density $\rho$ affects
the limiting dynamics of the asymmetric cases. This is mainly because
the higher density facilitates the first escape of one particle from
a condensate. Subsequently, the movement of the remaining particles
occurs instantaneously because of the asymmetry of the system. However,
for the symmetric case, this acceleration of the first jump by the
higher $\rho$ is canceled out by the fact that we have to move more
particles to the adjacent site for the higher $\rho$. These two effects
are exactly matched for the symmetric case; consequently, the limiting
dynamics becomes independent of $\rho$.
\item Our conditions on $d_{L}$ appeared in Theorems \ref{tt1}, \ref{tt2}, and \ref{tt3} are sub-optimal for technical
reasons. We believe that all the results must hold under the condition
$\lim d_{L}L^{d}\log L=0$ as in Theorem \ref{tlmain}.
\item Condensation of the zero-range process in the thermodynamic regime
exhibits phase transition in terms of $\rho$ (e.g., see \cite{A-G-L}).
More precisely, there exists $\rho_{c}>0$ such that condensation
occurs if and only if $\rho>\rho_{c}$. However, in the inclusion
process, we do not observe such a phenomenon. We refer to \cite[Proposition 1]{C-G-J}
for further details.
\end{enumerate}
\end{rem}

\section{\label{sec4}Movements of condensate: General Results}

In this section, we present general results regarding the metastable
behavior of the inclusion process on the basis of the martingale approach
of Beltr\'{a}n and Landim developed in \cite{B-L TM,B-L TM2}. The primary
contribution of this approach is to reduce the analysis of metastable
behavior to an investigation of the scaling limit of the so-called
\textit{mean-jump rate} of the trace process on metastable sets. In
the reversible case, the mean-jump rate can be estimated on the basis
of the Dirichlet--Thomson principle of the potential theory. For
instance, it leads us to rigorous results for the metastable behavior
of the reversible zero-range process \cite{A-G-L,B-L ZRP} and the
reversible inclusion process \cite{B-D-G}. Based on recent developments
\cite{G-L,S} of the non-reversible version of the Dirichlet--Thomson
principle along with the martingale approach developed in \cite{B-L TM2,L-Seo}
for the non-reversible case, the metastable behavior has also been
analyzed for the totally asymmetric zero-range process on the discrete
torus in \cite{L TAZRP} and for general non-reversible zero-range
processes in \cite{Seo NRZRP}. We emphasize that the explicit form
of the invariant measure played a crucial role in these studies.

Although we do not have such a formula, we will provide estimate of
the mean-jump rate in this section (cf. Proposition \ref{main}).
To overcome the lack of knowledge about the invariant measure, we
exploit the fact that the metastable set of inclusion process is a
singleton and the mean-jump rate is thus reduced to a jump rate between
these singletons.

The results obtained in this section directly imply Theorem \ref{t21}
regarding the metastable behavior of the inclusion processes under \textbf{(UI)},
as we are aware of the appearance of condensation for this case. We
explain this in Section \ref{sec5}. However, for the general case,
considerable effort is required to prove the existence of condensation
to apply the results obtained in this section. This will be done under
the condition\textbf{ (UP)} in Sections \ref{sec7} and \ref{sec6}.
We also discuss the thermodynamic limit in Section \ref{sec8} on
the basis of the results obtained in this section.

\subsection{\label{sec41}Applications of the martingale approach}

Here, we explain the application of the martingale approach
for the inclusion setting.

\subsubsection*{Preliminary: negligibility of excursions on $\Delta_{N}$}

As a preliminary step, we first verify the two conditions given by
\eqref{neg} and \eqref{neg2} for the inclusion process under static
condensation.
\begin{prop}
\label{pneg}Suppose that the inclusion process exhibits condensation
and let $S_{\star}$ be the maximal condensing set defined in \eqref{sstar}.
Then, for any sequence $(\alpha_{N})_{N=1}^{\infty}$ of positive
real numbers, we have
\begin{align}
\lim_{N\rightarrow\infty}\sup_{\eta\in\mathcal{E}_{N}(S_{\star})}\mathbb{E}_{\eta}\left[\int_{0}^{T}\mathbf{1}\left\{ \eta_{N}(\alpha_{N}s)\in\Delta_{N}\right\} ds\right] & =0\text{ for all }T>0\;,\label{et31}\\
\lim_{\delta\rightarrow0}\limsup_{N\rightarrow\infty}\sup_{2\delta\le s\le3\delta}\sup_{\eta\in\mathcal{E}_{N}(S_{\star})}\mathbb{P}_{\eta}\left[\eta_{N}(\alpha_{N}s)\in\Delta_{N}\right] & =0\;,\label{et32}
\end{align}
where $\Delta_{N}=\mathcal{E}_{N}(S_{\star})^{c}$. In other words,
the two conditions given by \eqref{neg} and \eqref{neg2} hold.
\end{prop}

\begin{proof}
For $x\in S_{\star}$, we have
\[
\mathbb{P}_{\xi_{N}^{x}}\bigg[\eta_{N}(\alpha_{N}s)\in\Delta_{N}\bigg]\le\frac{1}{\mu_{N}(\mathcal{E}_{N}^{x})}\mathbb{P}_{\mu_{N}}\bigg[\eta_{N}(\alpha_{N}s)\in\Delta_{N}\bigg]=\frac{\mu_{N}(\Delta_{N})}{\mu_{N}(\mathcal{E}_{N}^{x})}=o(1)\;,
\]
where the last identity follows from $\mu_{N}(\Delta_{N})=o(1)$ and
\eqref{sstar}. Now, \eqref{et31} directly follows from the Fubini
theorem, as does \eqref{et32}.
\end{proof}

\subsubsection*{Application of the martingale approach to inclusion processes}

For $A\subseteq S$, we consider the trace process $\eta_{N}^{\mathcal{E}_{N}(A)}(\cdot)$
defined in \eqref{etr}, which is a Markov chain on $\mathcal{E}_{N}(A)$.
Denote the jump rate of this Markov chain by $\mathbf{r}_{N}^{A}(\cdot,\,\cdot):\mathcal{E}_{N}(A)\times\mathcal{E}_{N}(A)\rightarrow[0,\,\infty)$.
Such a jump rate is called the \textit{mean-jump rate} in the context
of metastability theory. With this notation, the following is a consequence
of the martingale approach \cite{B-L TM,B-L TM2}.
\begin{prop}
\label{main0}Suppose that the inclusion process exhibits condensation
and let $S_{\star}$ be the maximal condensing set. In addition, we
suppose that
\begin{equation}
\lim_{N\rightarrow\infty}\theta_{N}\mathbf{r}_{N}^{S_{\star}}(\xi_{N}^{x},\,\xi_{N}^{y})=a(x,\,y)\text{ \;for all }x,\,y\in S_{\star}\label{fc}
\end{equation}
for some $a:S_{\star}\times S_{\star}\rightarrow[0,\,\infty)$. Then,
the movement of the condensate is described by a Markov chain on $S_{\star}$
with rate $a(\cdot,\,\cdot)$ and scale $\theta_{N}$. Moreover, the
same description holds in the sense of the finite-dimensional marginal
explained in Definition \ref{def25}.
\end{prop}

\begin{proof}
We refer to \cite[Theorem 2.7]{B-L TM} and \cite[Theorem 2.1]{B-L TM2}
for the first part of the proposition. The requirements of these theorems
hold because of Definition \ref{def12} and \eqref{fc}. The second
part of the proposition follows from \eqref{et32} of Proposition
\ref{pneg} and \cite[Proposition 2.1]{L-L-M}.
\end{proof}

\begin{rem}
We remark that the condition (\textbf{H1}) in \cite[Theorem 2.7]{B-L TM} and \cite[Theorem 2.1]{B-L TM2} is not included in our discussion. This is due to the fact that (\textbf{H1}) is obvious in our context because metastable valleys are singletons (Section \ref{sec22}).
\end{rem}

\subsubsection*{Estimation of the mean-jump rate}

In view of Proposition \ref{main0}, the analysis of the metastable
behavior of the inclusion process is reduced to find a suitable scaling
limit of the form \eqref{fc} for the mean-jump rates. Such a scaling
limit stated as Proposition \ref{main} below is the main result of
this section. Write
\[
\ell_{N}=d_{N}\log N+q^{N}\;,
\]
where $q\in(0,\,1)$ is a fixed constant that will be specified later
in \eqref{defq}.
\begin{prop}
\label{main}Suppose that $\lim_{N\rightarrow\infty}d_{N}\log N=0$.
Fix a non-empty set $A\subseteq S$ and define
\[
r_{N}^{A}(x,\,y)=\frac{1}{d_{N}N}\mathbf{r}_{N}^{A}(\xi_{N}^{x},\,\xi_{N}^{y})\text{ for }x,\,y\in A\;.
\]
\begin{enumerate}
\item If $A$ is a semi-attracting set, we have
\begin{equation}
r_{N}^{A}(x,\,y)=\begin{cases}
\left(1+O\big(\frac{1}{N}+\ell_{N}\big)\right)\,(r(x,\,y)-r(y,\,x)) & \text{if }r(x,\,y)>r(y,\,x)\;,\\
O\big(\frac{1}{N}+\ell_{N}\big) & \text{if }r(x,\,y)<r(y,\,x)\;,\\
\frac{1}{N}r(x,\,y)+O\big(\frac{1}{N}+\ell_{N}\big) & \text{if }r(x,\,y)=r(y,\,x)\;.
\end{cases}\label{emain-1}
\end{equation}
\item If $A$ is an attracting set, we have
\begin{equation}
r_{N}^{A}(x,\,y)=\begin{cases}
(1+O(\ell_{N}))\,(r(x,\,y)-r(y,\,x)) & \text{if }r(x,\,y)>r(y,\,x)\;,\\
O(\ell_{N}) & \text{if }r(x,\,y)<r(y,\,x)\;,\\
\frac{1}{N}r(x,\,y)+O(\ell_{N}) & \text{if }r(x,\,y)=r(y,\,x)\;.
\end{cases}\label{emain}
\end{equation}
\end{enumerate}
\end{prop}

It should be noted that the only difference between \eqref{emain-1}
and \eqref{emain} is the appearance of the additional $O(1/N)$-order
error term. Note that the error term $O(1/N)$ can be ignored when
we consider the time scale $\theta_{N}^{\mathrm{nrv}}=1/(Nd_{N})$.
Hence, in view of the following lemma, part (1) of the previous theorem
provides sufficient control regarding the proof of Theorem \ref{t213}.
\begin{lem}
\label{lem314}The set $S_{0}$ defined right after \eqref{bxy} is
a semi-attracting set.
\end{lem}

\begin{proof}
By contrast, suppose that some $x\in S_{0}$ and $y\in S_{0}^{c}$
satisfy $r(x,\,y)>r(y,\,x)$. Pick an invariant measure $\pi$ of
$Z_{1}(\cdot)$ such that $\pi(x)>0$ and $\pi(y)=0$. Then, we have
\[
0=\sum_{z\in S}\pi(y)b(y,\,z)=\sum_{z\in S}\pi(z)b(z,\,y)\ge\pi(x)b(x,\,y)=\pi(x)(r(x,\,y)-r(y,\,x))>0\;,
\]
which is a contradiction.
\end{proof}
Meanwhile, we cannot afford this error when we consider the time scale
$\theta_{N}^{\mathrm{rv}}=1/d_{N}$. Hence, we need to assume the
attractiveness of $A$ in Theorem \ref{t214} to eliminate this error
in \eqref{emain}.

The remainder of this section is devoted to proving Proposition \ref{main}.
We establish several preliminary estimates in Section \ref{sec42},
and the proof of Proposition \ref{main} is then given in Section
\ref{sec43}.

\subsection{\label{sec42}Hitting times on the tubes }

A set playing a significant role in the estimate of $\mathbf{r}_{N}^{A}(\xi_{N}^{x},\,\xi_{N}^{y})$
for $x,\,y\in A$ is the \textbf{\textit{tube}} $\mathcal{A}_{N}^{x,\,y}$
between $\xi_{N}^{x}$ and $\xi_{N}^{y}$ defined hereafter, as the
transition from $\xi_{N}^{x}$ to $\xi_{N}^{y}$ takes place only
along this tube with dominating probability.
\begin{notation}[Tube between metastable sets]
\label{nota}Here, we gather all the relevant notation related to
the tube that will be frequently used in the remainder of this study.
We refer to Figure \ref{fig21} for the illustration of the notation
introduced here.
\begin{figure}
\includegraphics[scale=0.22]{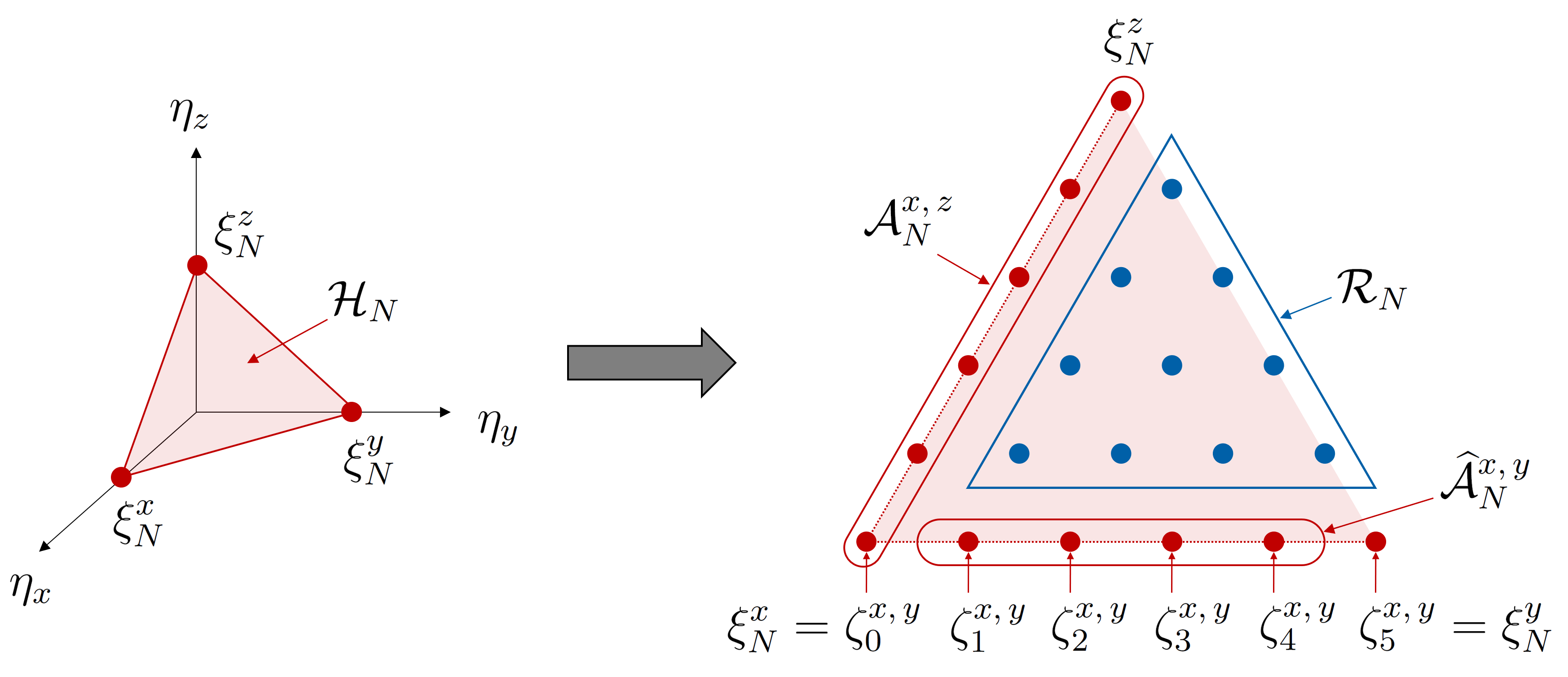}\caption{\label{fig21}Visualization of the objects introduced in Notation
\ref{nota} when $S=\{x,\,y,\,z\}$ and $r(y,\,z)=r(z,\,y)=0$.}
\end{figure}

\begin{itemize}
\item For $x,\,y\in S$, the \textbf{\textit{tube}} $\mathcal{A}_{N}^{x,\,y}$
between $\xi_{N}^{x}$ and $\xi_{N}^{y}$ is defined by
\[
\mathcal{A}_{N}^{x,\,y}=\{\eta\in\mathcal{H}_{N}:\eta_{x}+\eta_{y}=N\}\;.
\]
Note that this tube contains $\xi_{N}^{x}$ and $\xi_{N}^{y}$. Let
$\widehat{\mathcal{A}}_{N}^{x,\,y}$ denote the set obtained from
$\mathcal{A}_{N}^{x,\,y}$ by removing these two extremal configurations:
\[
\widehat{\mathcal{A}}_{N}^{x,\,y}=\{\eta\in\mathcal{A}_{N}^{x,\,y}:\eta_{x},\eta_{y}\ge1\}=\mathcal{A}_{N}^{x,\,y}\setminus\{\xi_{N}^{x},\,\xi_{N}^{y}\}\;.
\]
\item We denote the set $\mathcal{A}_{N}^{x,\,y}$ by $\{\zeta_{0}^{x,\,y},\,\zeta_{1}^{x,\,y},\,\dots,\,\zeta_{N}^{x,\,y}\}$\footnote{Indeed, it should be denoted by $\{\zeta_{0,\,N}^{x,\,y},\,\zeta_{1,\,N}^{x,\,y},\,\cdots,\,\zeta_{N,\,N}^{x,\,y}\}$;
however, we have ignored the dependency on $N$ in the notation.}, where $\zeta_{i}^{x,\,y}$ is defined by
\begin{equation}
(\zeta_{i}^{x,\,y})_{z}=\begin{cases}
N-i & \text{for }z=x\;,\\
i & \text{for }z=y\;,\\
0 & \text{otherwise\;},
\end{cases}\;\;\text{for }0\le i\le N\;.\label{ena}
\end{equation}
Note that $\zeta_{0}^{x,\,y}=\xi_{N}^{x}$, $\zeta_{N}^{x,\,y}=\xi_{N}^{y}$,
and $\widehat{\mathcal{A}}_{N}^{x,\,y}=\{\zeta_{1}^{x,\,y},\,\dots,\,\zeta_{N-1}^{x,\,y}\}.$
\item If $x,\,y\in S$ satisfy $r(x,\,y)+r(y,\,x)>0$, then we write $x\sim y$.
With this notation, we write
\begin{equation}
\mathcal{A}_{N}=\bigcup_{x,\,y\in S\,:\,x\sim y}\mathcal{A}_{N}^{x,\,y}\;\text{\;}\;\text{and\;\;\;}\widehat{\mathcal{A}}_{N}=\bigcup_{x,\,y\in S\,:\,x\sim y}\widehat{\mathcal{A}}_{N}^{x,\,y}\;.\label{an}
\end{equation}
Note that $\mathcal{A}_{N}=\widehat{\mathcal{A}}_{N}\cup\mathcal{E}_{N}$.
The remainder set is denoted by $\mathcal{R}_{N}$:
\[
\mathcal{R}_{N}=\mathcal{H}_{N}\setminus\mathcal{A}_{N}\;.
\]
\item Finally, we define several constants for convenience:
\begin{align}
R_{1} & =\min\{r(x,\,y):x,\,y\in S\text{ such that }r(x,\,y)>0\}>0\;,\nonumber \\
R_{2} & =\max\{r(x,\,y):x,\,y\in S\}\;,\label{const}\\
\Lambda & =\max\{\lambda(x):x\in S\}\;,\nonumber
\end{align}
where $\lambda(x)=\sum_{y\in S}r(x,\,y)$ denotes the holding rate
of the underlying random walk. For $x,\,y\in S$ satisfying $x\sim y$,
we write
\begin{equation}
q_{x,\,y}=\frac{\min\{r(x,\,y),\,r(y,\,x)\}}{\max\{r(x,\,y),\,r(y,\,x)\}}\in[0,\,1]\;.\label{qdef}
\end{equation}
Then, we define
\begin{equation}
q=\max\{q_{x,\,y}:x,\,y\in S,\,x\sim y\text{ and }r(x,\,y)\neq r(y,\,x)\}<1\;.\label{defq}
\end{equation}
\end{itemize}
\end{notation}

For $\mathcal{C}\subseteq\mathcal{H}_{N}$, let $\tau_{\mathcal{C}}$
denote the hitting time of the set $\mathcal{C}$. If the set $\mathcal{C}=\{\eta\}$
is a singleton, we write $\tau_{\{\eta\}}$ simply as $\tau_{\eta}$.
In the remainder of this section, we fix $A\subseteq S$ and $x,\,y\in A$.
Then, we define an event $E_{0}=E_{0}^{y,\,A}$ by
\[
E_{0}=\{\tau_{\mathcal{E}_{N}^{y}}=\tau_{\mathcal{E}_{N}(A)}\}\;.
\]
Now, we provide a sequence of lemmas regarding the probability of
the event $E_{0}$. We remark that these lemmas are also valid for
a wide class of events that depend only on the hitting times of subsets
of $(\widehat{\mathcal{A}}_{N}^{x,\,y})^{c}$ such as $\{\tau_{\mathcal{E}_{N}^{x}}<\tau_{\mathcal{E}_{N}^{y}}\}$.

The first lemma asserts that, provided $d_{N}$ is sufficiently small,
the inclusion process on $\widehat{\mathcal{A}}_{N}^{x,\,y}$ behaves
as a nearest-neighbor random walk whose jump rate from $\zeta_{i}^{x,\,y}$
to $\zeta_{i+1}^{x,\,y}$ is $r(x,\,y)$ and from $\zeta_{i+1}^{x,\,y}$
to $\zeta_{i}^{x,\,y}$ is $r(y,\,x)$, especially when we are only
concerned with the event $E_{0}$.
\begin{lem}
\label{tubetrest}Suppose that $x,\,y\in S$ satisfy $x\sim y$. Then,
there exists $C>0$ such that
\[
\bigg|\,\mathbb{P}_{\zeta_{i}^{x,\,y}}\left[E_{0}\right]-\frac{r(x,\,y)}{r(x,\,y)+r(y,\,x)}\mathbb{P}_{\zeta_{i+1}^{x,\,y}}\left[E_{0}\right]-\frac{r(y,\,x)}{r(x,\,y)+r(y,\,x)}\mathbb{P}_{\zeta_{i-1}^{x,\,y}}\left[E_{0}\right]\bigg|\le C\frac{d_{N}N}{i(N-i)}
\]
for all $i\in\llbracket1,\,N-1\rrbracket$\footnote{For $a,\,b\in\mathbb{\mathbb{Z}}$, the interval $\llbracket a,\,b\rrbracket$
denotes $[a,\,b]\cap\mathbb{Z}$.}.
\end{lem}

\begin{proof}
Recall $\mathbf{r}_{N}(\cdot,\,\cdot)$, $\lambda_{N}(\cdot)$, and
$\mathbf{p}_{N}(\cdot,\,\cdot)$ from \eqref{rn}, \eqref{holding},
and \eqref{pn}, respectively. Then, we can write
\begin{align}
\mathbf{r}_{N}(\zeta_{i}^{x,\,y},\,\zeta_{i+1}^{x,\,y}) & =(N-i)(d_{N}+i)\,r(x,\,y)\;,\label{ern1}\\
\mathbf{r}_{N}(\zeta_{i}^{x,\,y},\,\zeta_{i-1}^{x,\,y}) & =i\,(d_{N}+N-i)\,r(y,\,x)\;,\label{ern2}\\
\mathbf{r}_{N}(\zeta_{i}^{x,\,y},\,\sigma^{x,\,z}\zeta_{i}^{x,\,y}) & =(N-i)\,d_{N}\,r(x,\,z)\;\;\;;\quad z\ne x,\,y\;,\nonumber \\
\mathbf{r}_{N}(\zeta_{i}^{x,\,y},\,\sigma^{y,\,z}\zeta_{i}^{x,\,y}) & =i\,d_{N}\,r(y,\,z)\;\;\;;\quad z\ne x,\,y\;.\nonumber
\end{align}
Thus, the holding rate at $\zeta_{i}^{x,\,y}$ is given by
\begin{equation}
\lambda_{N}(\zeta_{i}^{x,\,y})=i(N-i)\{r(x,\,y)+r(y,\,x)\}+d_{N}\{(N-i)\lambda(x)+i\lambda(y)\}\;.\label{eh}
\end{equation}
Hence, by \eqref{ern1} and \eqref{eh},
\begin{align}
 & \bigg|\,\mathbf{p}_{N}(\zeta_{i}^{x,\,y},\,\zeta_{i+1}^{x,\,y})-\frac{r(x,\,y)}{r(x,\,y)+r(y,\,x)}\bigg|\nonumber \\
=\; & \frac{d_{N\,}r(x,\,y)|(N-i)\{r(x,\,y)+r(y,\,x)\}-\{(N-i)\lambda(x)+i\lambda(y)\}|}{[i(N-i)\{r(x,\,y)+r(y,\,x)\}+d_{N}\{(N-i)\lambda(x)+i\lambda(y)\}]\{r(x,\,y)+r(y,\,x)\}}\nonumber \\
\le\; & \frac{2R_{2}\Lambda}{R_{1}^{2}}d_{N}\frac{N}{i(N-i)}\;,\label{epe1}
\end{align}
where the last line follows from the definition \eqref{const}. Similarly,
by \eqref{ern2} and \eqref{eh},
\begin{equation}
\bigg|\,\mathbf{p}_{N}(\zeta_{i}^{x,\,y},\,\zeta_{i-1}^{x,\,y})-\frac{r(y,\,x)}{r(x,\,y)+r(y,\,x)}\bigg|\le\frac{2R_{2}\Lambda}{R_{1}^{2}}d_{N}\frac{N}{i(N-i)}\;.\label{epe2}
\end{equation}
The last two bounds imply that
\begin{equation}
\sum_{z:z\ne x,\,y}\mathbf{p}_{N}(\zeta_{i}^{x,\,y},\,\sigma^{x,\,z}\zeta_{i}^{x,\,y})+\sum_{z:z\ne x,\,y}\mathbf{p}_{N}(\zeta_{i}^{x,\,y},\,\sigma^{y,\,z}\zeta_{i}^{x,\,y})\le\frac{4R_{2}\Lambda}{R_{1}^{2}}d_{N}\frac{N}{i(N-i)}\;.\label{epe3}
\end{equation}
By the Markov property, we have
\begin{align*}
\mathbb{P}_{\zeta_{i}^{x,\,y}}\left[E_{0}\right]= & \;\mathbf{p}_{N}(\zeta_{i}^{x,\,y},\,\zeta_{i+1}^{x,\,y})\mathbb{P}_{\zeta_{i+1}^{x,\,y}}\left[E_{0}\right]+\mathbf{p}_{N}(\zeta_{i}^{x,\,y},\,\zeta_{i-1}^{x,\,y})\mathbb{P}_{\zeta_{i-1}^{x,y}}\left[E_{0}\right]\\
 & \;+\sum_{z:z\ne x,\,y}\mathbf{p}_{N}(\zeta_{i}^{x,\,y},\,\sigma^{x,\,z}\zeta_{i}^{x,\,y})\mathbb{P}_{\sigma^{x,\,z}\zeta_{i}^{x,\,y}}\left[E_{0}\right]\\
 & \;+\sum_{z:z\ne x,\,y}\mathbf{p}_{N}(\zeta_{i}^{x,\,y},\,\sigma^{y,\,z}\zeta_{i}^{x,\,y})\mathbb{P}_{\sigma^{y,\,z}\zeta_{i}^{x,\,y}}\left[E_{0}\right]\;.
\end{align*}
Finally, inserting the estimates \eqref{epe1}, \eqref{epe2}, and
\eqref{epe3} into the last identity completes the proof.
\end{proof}
On the basis of the previous estimate, we can estimate the probabilities
$\mathbb{P}_{\zeta_{1}^{x,\,y}}\left[E_{0}\right]$ and $\mathbb{P}_{\zeta_{N-1}^{x,\,y}}\left[E_{0}\right]$
in terms of $\mathbb{P}_{\xi_{N}^{x}}\left[E_{0}\right]$ and $\mathbb{P}_{\xi_{N}^{y}}\left[E_{0}\right]$.
We divide this estimate into three cases according to the relation
between $r(x,\,y)$ and $r(y,\,x)$ as follows.
\begin{lem}
\label{tubelem1}Suppose that $x,\,y\in S$ satisfy $r(x,\,y)>r(y,\,x)>0$.
Then, it holds that
\begin{align*}
\bigg|\,\mathbb{P}_{\zeta_{1}^{x,\,y}}\left[E_{0}\right]-\frac{q_{x,\,y}-q_{x,\,y}^{N}}{1-q_{x,\,y}^{N}}\mathbb{P}_{\xi_{N}^{x}}\left[E_{0}\right]-\frac{1-q_{x,\,y}}{1-q_{x,\,y}^{N}}\mathbb{P}_{\xi_{N}^{y}}\left[E_{0}\right]\bigg| & =O(d_{N}\log N)\;\;\text{and }\\
\bigg|\,\mathbb{P}_{\zeta_{N-1}^{x,\,y}}\left[E_{0}\right]-\frac{q_{x,\,y}^{N-1}-q_{x,\,y}^{N}}{1-q_{x,\,y}^{N}}\mathbb{P}_{\xi_{N}^{x}}\left[E_{0}\right]-\frac{1-q_{x,\,y}^{N-1}}{1-q_{x,\,y}^{N}}\mathbb{P}_{\xi_{N}^{y}}\left[E_{0}\right]\bigg| & =O(d_{N}\log N)\;.
\end{align*}
\end{lem}

\begin{proof}
Following \eqref{qdef} and Lemma \ref{tubetrest}, it holds for $i\in\llbracket1,\,N-1\rrbracket$
that
\begin{equation}
\bigg|\,\mathbb{P}_{\zeta_{i}^{x,\,y}}\left[E_{0}\right]-\frac{1}{1+q_{x,\,y}}\mathbb{P}_{\zeta_{i+1}^{x,\,y}}\left[E_{0}\right]-\frac{q_{x,\,y}}{1+q_{x,\,y}}\mathbb{P}_{\zeta_{i-1}^{x,\,y}}\left[E_{0}\right]\bigg|\le C\frac{d_{N}N}{i(N-i)}\;.\label{emm1}
\end{equation}
Write
\[
b_{i}=\mathbb{P}_{\zeta_{i-1}^{x,\,y}}\left[E_{0}\right]-\mathbb{P}_{\zeta_{i}^{x,\,y}}\left[E_{0}\right]\text{ for}\;i\in\llbracket1,\,N\rrbracket\;;
\]
so that we can rewrite \eqref{emm1} as
\[
\bigg|\,b_{i+1}-q_{x,\,y}b_{i}\bigg|\le C\frac{d_{N}N}{i(N-i)}(1+q_{x,\,y})\;,
\]
and therefore, for $i\in\llbracket1,\,N\rrbracket$,
\begin{align*}
|\,b_{i}-q_{x,\,y}^{i-1}b_{1}| & \le Cd_{N}N(1+q_{x,\,y})\sum_{j=1}^{i-1}\frac{q_{x,\,y}^{i-1-j}}{j(N-j)}\;.
\end{align*}
Since $\mathbb{P}_{\xi_{N}^{x}}\left[E_{0}\right]-\mathbb{P}_{\xi_{N}^{y}}\left[E_{0}\right]=b_{1}+\cdots+b_{N}$,
the previous bound implies that

\begin{align*}
\left|\,\mathbb{P}_{\xi_{N}^{x}}\left[E_{0}\right]-\mathbb{P}_{\xi_{N}^{y}}\left[E_{0}\right]-\sum_{i=1}^{N}q_{x,\,y}^{i-1}b_{1}\right| & \le Cd_{N}N(1+q_{x,\,y})\sum_{i=1}^{N}\sum_{j=1}^{i-1}\frac{q_{x,\,y}^{i-1-j}}{j(N-j)}\\
 & =Cd_{N}N(1+q_{x,\,y})\sum_{j=1}^{N-1}\frac{1}{j(N-j)}\sum_{i=j+1}^{N}q_{x,\,y}^{i-1-j}\\
 & \le Cd_{N}\sum_{j=1}^{N-1}\left[\frac{1}{j}+\frac{1}{N-j}\right]\frac{1+q_{x,\,y}}{1-q_{x,\,y}}\le Cd_{N}\log N\;.
\end{align*}
From this computation, we can deduce that
\[
\left|\,b_{1}-\frac{1-q_{x,\,y}}{1-q_{x,\,y}^{N}}(\mathbb{P}_{\xi_{N}^{x}}\left[E_{0}\right]-\mathbb{P}_{\xi_{N}^{y}}\left[E_{0}\right])\right|\le Cd_{N}\log N\;.
\]
By inserting $b_{1}=\mathbb{P}_{\xi_{N}^{x}}\left[E_{0}\right]-\mathbb{P}_{\zeta_{1}^{x,\,y}}\left[E_{0}\right]$,
we obtain the first estimate of the lemma. The second one can be proved
similarly.
\end{proof}
Now, we consider the second case in which the jump from $y$ to $x$
is excluded.
\begin{lem}
\label{tubelem2}Suppose that $x,\,y\in S$ satisfy $r(x,\,y)>r(y,\,x)=0$.
Then, it holds that
\[
\big|\,\mathbb{P}_{\zeta_{1}^{x,\,y}}\left[E_{0}\right]-\mathbb{P}_{\xi_{N}^{y}}\left[E_{0}\right]\big|=O(d_{N}\log N)\;\text{ and\;\;}\big|\,\mathbb{P}_{\zeta_{N-1}^{x,\,y}}\left[E_{0}\right]-\mathbb{P}_{\xi_{N}^{y}}\left[E_{0}\right]\big|=O(d_{N})\;.
\]
\end{lem}

\begin{proof}
By Lemma \ref{tubetrest}, it holds that
\[
\big|\,\mathbb{P}_{\zeta_{i}^{x,\,y}}\left[E_{0}\right]-\mathbb{P}_{\zeta_{i+1}^{x,\,y}}\left[E_{0}\right]\big|\le C\frac{d_{N}N}{i(N-i)}\text{ for all }i\in\llbracket1,\,N-1\rrbracket\;.
\]
By inserting $i=N-1$, we immediately obtain the second estimate.
For the first estimate, it suffices to apply the triangle inequality
such that
\[
\,\big|\mathbb{P}_{\zeta_{1}^{x,\,y}}\left[E_{0}\right]-\mathbb{P}_{\xi_{N}^{y}}\left[E_{0}\right]\big|\le\sum_{i=1}^{N-1}C\frac{d_{N}N}{i(N-i)}=O(d_{N}\log N)\;.
\]
This completes the proof of the first estimate.
\end{proof}
Now, we consider the last case, i.e., the symmetric case.
\begin{lem}
\label{tubelem3}Suppose that $x,\,y\in S$ satisfy $r(x,\,y)=r(y,\,x)>0$.
Then, it holds that
\begin{equation}
\bigg|\mathbb{P}_{\zeta_{1}^{x,\,y}}\left[E_{0}\right]-\frac{N-1}{N}\mathbb{P}_{\xi_{N}^{x}}\left[E_{0}\right]-\frac{1}{N}\mathbb{P}_{\xi_{N}^{y}}\left[E_{0}\right]\bigg|=O(d_{N}\log N)\;.\label{eqt1}
\end{equation}
\end{lem}

\begin{proof}
For $i\in\llbracket1,\,N-1\rrbracket$, write
\begin{equation}
c_{i}=\mathbb{P}_{\zeta_{i-1}^{x,\,y}}\left[E_{0}\right]-\mathbb{P}_{\zeta_{i}^{x,\,y}}\left[E_{0}\right]-\frac{1}{N}(\mathbb{P}_{\xi_{N}^{x}}\left[E_{0}\right]-\mathbb{P}_{\xi_{N}^{y}}\left[E_{0}\right])\;.\label{ci}
\end{equation}
Then, we can observe that
\begin{equation}
c_{1}+\cdots+c_{N}=0\label{sci}
\end{equation}
and that the left-hand side of \eqref{eqt1} is $|c_{1}|$. Thus,
it suffices to show that $|c_{1}|=O(d_{N}\log N)$.

By Lemma \ref{tubetrest}, it holds that
\[
\bigg|\mathbb{P}_{\zeta_{i}^{x,\,y}}\left[E_{0}\right]-\frac{1}{2}\mathbb{P}_{\zeta_{i+1}^{x,\,y}}\left[E_{0}\right]-\frac{1}{2}\mathbb{P}_{\zeta_{i-1}^{x,\,y}}\left[E_{0}\right]\bigg|\le C\frac{d_{N}N}{i(N-i)}\text{ for all }i\in\llbracket1,\,N-1\rrbracket\;.
\]
By \eqref{ci} , this inequality can be written as
\[
|c_{i}-c_{i+1}|\le C\frac{d_{N}N}{i(N-i)}\;.
\]
Therefore, by the triangle inequality, we obtain
\[
|c_{1}-c_{i}|\le\sum_{j=1}^{i-1}|c_{j}-c_{j+1}|\le Cd_{N}\sum_{j=1}^{i-1}\frac{N}{j(N-j)}\le Cd_{N}\log N\;.
\]
Hence, by \eqref{sci},
\[
|Nc_{1}|=\left|Nc_{1}-(c_{1}+\cdots+c_{N})\right|\le\sum_{i=2}^{N}|c_{1}-c_{i}|\le Cd_{N}N\log N\;.
\]
This completes the proof of $|c_{1}|=O(d_{N}\log N)$.
\end{proof}

\subsection{\label{sec43}Proof of Proposition\textmd{ }\ref{main}}
\begin{proof}[Proof of Proposition \ref{main}]
Fix $A\subseteq S$ and fix $x,\,y\in A$. By \cite[Corollary 6.2]{B-L TM},
we can write the jump rate $\mathbf{r}_{N}^{A}(\xi_{N}^{x},\,\xi_{N}^{y})$
as
\begin{align}
\mathbf{r}_{N}^{A}(\xi_{N}^{x},\,\xi_{N}^{y})= & \mathbf{r}_{N}(\xi_{N}^{x},\,\xi_{N}^{y})+\sum_{\eta\in\mathcal{H}_{N}\setminus\mathcal{E}_{N}(A)}\mathbf{r}_{N}(\xi_{N}^{x},\,\eta)\mathbb{P}_{\eta}\left[\tau_{\mathcal{E}_{N}^{y}}=\tau_{\mathcal{E}_{N}(A)}\right]\nonumber \\
= & \sum_{z:z\ne x}Nd_{N}\,r(x,\,z)\mathbb{P}_{\zeta_{1}^{x,\,z}}\left[E_{0}\right]\;.\label{tracejr}
\end{align}
Hence, it suffices to estimate $\mathbb{P}_{\zeta_{1}^{x,\,z}}\left[E_{0}\right]$
for $z\neq x$ with $r(x,\,z)>0$ to estimate $\mathbf{r}_{N}^{A}(\xi_{N}^{x},\,\xi_{N}^{y})$.

Suppose first that $z\ne y$. Then, we divide the estimate of $\mathbb{P}_{\zeta_{1}^{x,\,z}}\left[E_{0}\right]$
into two cases:\\
\\
\textbf{(Case 1: $z\in A$)} Since $\mathbb{P}_{\xi_{N}^{z}}[E_{0}]=0$,
we deduce from Lemmas \ref{tubelem1}, \ref{tubelem2}, and \ref{tubelem3}
that
\begin{equation}
\mathbb{P}_{\zeta_{1}^{x,\,z}}[E_{0}]=O(d_{N}\log N)\;.\label{tjr1}
\end{equation}
\\
\textbf{(Case 2: $z\notin A$)} We divide this case into two as following:
\begin{itemize}
\item If $A$ is attracting, we have $r(x,\,z)<r(z,\,x)$. Thus by Lemma
\ref{tubelem1} we obtain
\begin{equation}
\mathbb{P}_{\zeta_{1}^{x,\,z}}[E_{0}]=O(q^{N})\mathbb{P}_{\xi_{N}^{z}}[E_{0}]+O(d_{N}\log N)=O(\ell_{N})\;.\label{tjr1-1}
\end{equation}
\item If $A$ is semi-attracting, we only have $r(x,\,z)\le r(z,\:x)$.
Thus by Lemmas \ref{tubelem1} and \ref{tubelem3} we obtain
\begin{equation}
\mathbb{P}_{\zeta_{1}^{x,\,z}}[E_{0}]=O\Big(\frac{1}{N}+q^{N}\Big)\,\mathbb{P}_{\xi_{N}^{z}}[E_{0}]+O(d_{N}\log N)=O\Big(\frac{1}{N}+\ell_{N}\Big)\;.\label{tjr1-2}
\end{equation}
\end{itemize}
Now it remains to estimate $\mathbb{P}_{\zeta_{1}^{x,\,y}}\left[E_{0}\right]$
when $r(x,\,y)\neq0$ to estimate \eqref{tracejr}. To this end, we
consider four cases separately:
\begin{enumerate}
\item $r(x,\,y)>r(y,\,x)>0$: By Lemma \ref{tubelem1} and the fact that
\begin{equation}
\mathbb{P}_{\xi_{N}^{x}}[E_{0}]=0\text{ and }\mathbb{P}_{\xi_{N}^{y}}[E_{0}]=1\;,\label{err11}
\end{equation}
we have that
\[
\left|\,\mathbb{P}_{\zeta_{1}^{x,\,y}}[E_{0}]-\frac{q_{x,\,y}-q_{x,\,y}^{N}}{1-q_{x,\,y}^{N}}\cdot0-\frac{1-q_{x,\,y}}{1-q_{x,\,y}^{N}}\cdot1\right|=O(d_{N}\log N)\;.
\]
Thus, we have that
\begin{equation}
\mathbb{P}_{\zeta_{1}^{x,\,y}}[E_{0}]=\frac{1-q_{x,\,y}}{1-q_{x,\,y}^{N}}+O(d_{N}\log N)=(1+O(\ell_{N}))(1-q_{x,\,y})\;.\label{tjr2}
\end{equation}
\item $r(y,\,x)>r(x,\,y)>0$: By Lemma \ref{tubelem1} and \eqref{err11},
\[
\bigg|\,\mathbb{P}_{\zeta_{1}^{x,\,y}}[E_{0}]-\frac{1-q_{x,\,y}^{N-1}}{1-q_{x,\,y}^{N}}\cdot0-\frac{q_{x,\,y}^{N-1}-q_{x,\,y}^{N}}{1-q_{x,\,y}^{N}}\cdot1\bigg|=O(d_{N}\log N)\;.
\]
Therefore, we obtain that
\begin{equation}
\mathbb{P}_{\zeta_{1}^{x,\,y}}[E_{0}]=\frac{q_{x,\,y}^{N-1}-q_{x,\,y}^{N}}{1-q_{x,\,y}^{N}}+O(d_{N}\log N)=O(\ell_{N})\;.\label{tjr3}
\end{equation}
\item $r(x,\,y)>r(y,\,x)=0$: By Lemma \ref{tubelem2} and \eqref{err11},
\begin{equation}
\mathbb{P}_{\zeta_{1}^{x,\,y}}[E_{0}]=1+O(d_{N}\log N)\;.\label{tjr4}
\end{equation}
\item $r(x,\,y)=r(y,\,x)>0$: By Lemma \ref{tubelem3} and \eqref{err11},
\[
\bigg|\mathbb{P}_{\zeta_{1}^{x,\,y}}[E_{0}]-\frac{N-1}{N}\cdot0-\frac{1}{N}\cdot1\bigg|=O(d_{N}\log N)\;.
\]
Hence, we can conclude that
\begin{equation}
\mathbb{P}_{\zeta_{1}^{x,\,y}}[E_{0}]=\frac{1}{N}+O(d_{N}\log N)\;.\label{tjr5}
\end{equation}
\end{enumerate}
Finally, we can combine \eqref{tjr1}-\eqref{tjr5} along with the
identity \eqref{tracejr} to complete the proof of the proposition.
\end{proof}

\section{\label{sec5}Metastable Behavior of Inclusion Processes under Condition
\textbf{(UI)}}

In this section, we investigate the metastable behavior of the inclusion
process under the condition \textbf{(UI)}. We first show that the
invariant measure $\mu_{N}(\cdot)$ admits the expression \eqref{mundef}.
\begin{proof}[Proof of Proposition \ref{p13} for case \textbf{\textup{(UI)}}]
It suffices to prove that, for $\eta\in\mathcal{H}_{N}$,
\begin{equation}
\sum_{x,\,y\in S:\eta_{y}\ge1}\mu_{N}(\sigma^{y,\,x}\eta)(\sigma^{y,\,x}\eta)_{x}(d_{N}+(\sigma^{y,\,x}\eta)_{y})r(x,\,y)=\mu_{N}(\eta)\sum_{x,y\in S}\eta_{y}(d_{N}+\eta_{x})r(y,\,x)\;.\label{prop2.1}
\end{equation}
Calculating the left-hand side of \eqref{prop2.1}, it holds that
\begin{align*}
 & \sum_{x,\,y\in S\,:\,\eta_{y}\ge1}\mu_{N}(\sigma^{y,\,x}\eta)(\sigma^{y,\,x}\eta)_{x}(d_{N}+(\sigma^{y,\,x}\eta)_{y})r(x,\,y)\\
=\; & \sum_{y\in S\,:\,\eta_{y}\ge1}\,\sum_{x\in S}\mu_{N}(\sigma^{y,\,x}\eta)(\eta_{x}+1)(d_{N}+\eta_{y}-1)r(x,\,y)\\
=\; & \mu_{N}(\eta)\sum_{y\,:\,\eta_{y}\ge1}\,\sum_{x\in S}\eta_{y}(d_{N}+\eta_{x})r(x,\,y)=\mu_{N}(\eta)\sum_{x,\,y\in S}(\eta_{x}\eta_{y}+d_{N}\eta_{y})r(x,\,y)\;.
\end{align*}
Comparing to the right-hand side of \eqref{prop2.1}, it suffices
to show that
\[
\sum_{x,\,y\in S}\eta_{y}r(x,\,y)=\sum_{x,\,y\in S}\eta_{y}r(y,\,x)\;.
\]
This identity holds since
\[
\sum_{x,\,y\in S}\eta_{y}r(x,\,y)=\sum_{y\in S}\eta_{y}\sum_{x\in S}r(x,\,y)\stackrel{\text{\textbf{(UI)}}}{=}\sum_{y\in S}\eta_{y}\sum_{x\in S}r(y,\,x)=\sum_{x,\,y\in S}\eta_{y}r(y,\,x)\;.
\]
\end{proof}
Now, we can prove Theorem \ref{t21} by gathering the results obtained
so far.
\begin{proof}[Proof of Theorem \ref{t21}]
As we mentioned before, part (1) follows from the investigation
of the reversible case. Hence, we shall only concentrate on part (2).
By Propositions \ref{p13} and \ref{p14}, we know that condensation
occurs on the entire set $S$, i.e., $S=S_{\star}$. Then, the condition
\eqref{fc} of Proposition \ref{main0} follows from Proposition \ref{main}
with $A=S$, with $\theta_{N}=\frac{1}{Nd_{N}}$ and
\[
a(x,\,y)=\left[r(x,\,y)-r(y,\,x)\right]\mathbf{1}\left\{ r(x,\,y)>r(y,\,x)\right\} \text{ for all }x,\,y\in S\;.
\]
These scale and limiting chain correspond to \eqref{anrv} and \eqref{thetanrv}
of Conjecture \ref{cnrv}, and the proof is completed.
\end{proof}

\section{\label{sec6}Metastable Behavior of Inclusion Processes with Condensation}

In this section, we are concerning on the metastable behavior of the
condensate of non-reversible inclusion processes under the condition
that the condensation occurs, namely Theorems \ref{t213}, \ref{t214}
and \ref{t215}. By assuming several irreducibility conditions on
the limiting Markov chain, we derive the followings in this section
based on the results obtained in Section \ref{sec4}:
\begin{itemize}
\item the characterization of the maximal condensing set $S_{\star}\subseteq S$,
\item the asymptotic limit of $\mu_{N}(\xi_{N}^{x})$ for $x\in S_{\star}$
as $N\rightarrow\infty$,
\item the limiting Markov chain on $S_{\star}$ describing the movement
of condensate.
\end{itemize}
We prove these main results in Section \ref{sec62} based on a lemma
introduced in \ref{sec61}.

\subsection{\label{sec61}A preliminary lemma}

In this short subsection, we introduce an elementary lemma. We believe
that this result is not new, but we include the full proof since we
were not able to find an exact reference that states the exact result
that we need.
\begin{lem}
\label{trrmatconv}Let $(Z_{N}(\cdot))_{N=1}^{\infty}$ be a sequence
of continuous-time Markov chains on a finite set $\mathfrak{S}$.
Denote the jump rate of $Z_{N}(\cdot)$
by $a_{N}(\cdot,\,\cdot)$ and fix an invariant measure $\pi_{N}(\cdot)$ of $Z_{N}(\cdot)$ for each $N$. Suppose
in addition that
\begin{equation}
\lim_{N\rightarrow\infty}a_{N}(x,\,y)=a(x,\,y)\text{ for all }x,\,y\in\mathfrak{S}\;.\label{e71}
\end{equation}
Then each limit point of $\{\pi_{N}\}$ becomes an invariant measure
for the Markov chain $Z(\cdot)$ with jump rate $a(\cdot,\,\cdot)$.
Moreover, if $Z(\cdot)$ admits the unique invariant measure $\pi$,
then we have that
\begin{equation}
\lim_{N\rightarrow\infty}\pi_{N}(x)=\pi(x)\text{ for all }x,\,y\in\mathfrak{S}\;.\label{e72}
\end{equation}
\end{lem}

\begin{rem*}
In the second statement above, note that we did not assume the irreducibility
of $Z(\cdot)$. However, the uniqueness of the invariant measure for
$Z(\cdot)$ is a crucial condition for this statement.
\end{rem*}
\begin{proof}
Suppose that a subsequence $(\pi_{N_{k}})_{k=1}^{\infty}$ converges
to $\pi_{0}$. Note that $\pi_{0}$ must be a probability measure
on $\mathfrak{S}$ as well. Since $\pi_{N_{k}}$ is an invariant measure
for the chain $Z_{N_{k}}$, we have
\[
\sum_{y\in\mathfrak{S}}\pi_{N_{k}}(x)a_{N_{k}}(x,\,y)=\sum_{y\in\mathfrak{S}}\pi_{N_{k}}(y)a_{N_{k}}(y,\,x)\text{ for all }x,\,y\in\mathfrak{S}\;.
\]
By letting $k\rightarrow\infty$ at the last identity, we obtain that
\[
\sum_{y\in\mathfrak{S}}\pi_{0}(x)a(x,\,y)=\sum_{y\in\mathfrak{S}}\pi_{0}(y)a(y,\,x)\text{ for all }x,\,y\in\mathfrak{S}\;.
\]
Therefore, $\pi_{0}$ is an invariant measure of $Z(\cdot)$. This
concludes the first statement.

Next we consider the second statement. Since $\{\pi_{N}:N\in\mathbb{N}\}$
is a bounded subset of $\mathbb{R}^{\mathfrak{S}}$, we know that
this set is precompact. Moreover, we have shown above that every convergent
subsequence converges to an invariant measure of $Z(\cdot)$, which
should be $\pi$ by the uniqueness assumption for this case. This
completes the proof.
\end{proof}

\subsection{\label{sec62}Proof of main results}

Now, we are ready to prove Theorems \ref{t213}, \ref{t214}, and
\ref{t215}. We consider the asymmetric case and the symmetric case
separately. Recall two Markov chains $(Z_{1}(t))_{t\ge0}$ and $(Z_{2}(t))_{t\ge0}$
and the set $S_{0}\subseteq S$ from Section \ref{sec33}. We start
with the asymmetric case.
\begin{proof}[Proof of Theorem \ref{t213} and the asymmetric case of Theorem \ref{t215}]
To start the proof, we first prove Theorem \ref{t213} by using Proposition \ref{main0}.
It suffices to verify \eqref{fc} and the fact that $S_{\star}=S_{0}$.
We recall the invariant measure $\nu$ of $Y^{\mathrm{nrv}}(\cdot)$
on $S_{0}$ (cf. Theorem \ref{t215}), and the rate $b(\cdot,\,\cdot):S\times S\rightarrow[0,\,\infty)$
defined in \eqref{bxy}. Recalling the remark after Notation \ref{n32},
the set $S_{0}$ is semi-attracting. Thus, by Proposition \ref{main},
we know that
\[
\lim_{N\rightarrow\infty}\theta_{N}^{\mathrm{nrv}}\mathbf{r}_{N}^{S}(\xi_{N}^{x},\,\xi_{N}^{y})=b(x,\,y)\text{ \;for all }x,\,y\in S\;,
\]
where $\theta_{N}^{\mathrm{nrv}}=\frac{1}{Nd_{N}}$. We assumed that
the Markov chain $Z_{1}(\cdot)$ with jump kernel $b$ has the only
irreducible component $S_{0}$, and this guarantees the uniqueness
of the invariant measure of $Z_{1}(\cdot)$, which will be denoted
by $\pi$. Since the invariant measure of the trace process $\eta_{N}^{\mathcal{E}_{N}}(\theta_{N}^{\mathrm{nrv}}\cdot)$
is the conditioned measure $\mu_{N}(\cdot\,|\mathcal{E}_{N})=\mu_{N}(\cdot)/\mu_{N}(\mathcal{E}_{N})$
on $\mathcal{E}_{N}$, we can deduce from Lemma \ref{trrmatconv}
that
\[
\lim_{N\rightarrow\infty}\frac{\mu_{N}(\xi_{N}^{x})}{\mu_{N}(\mathcal{E}_{N})}=\pi(x)\;\;\text{for all }x\in S\;.
\]
Since condensation occurs, i.e., $\lim_{N\rightarrow\infty}\mu_{N}(\mathcal{E}_{N})=1$,
we obtain that
\begin{equation}
\lim_{N\rightarrow\infty}\mu_{N}(\xi_{N}^{x})=\pi(x)\;.\label{e73}
\end{equation}
Since $S_{0}$ is the unique irreducible component of the chain $Z_{1}(\cdot)$,
we know that $\pi(x)=0$ for $x\in S\setminus S_{0}$, and that $\pi(x)>0$
for $x\in S_{0}$. From this and \eqref{e73}, we can conclude that
$S_{\star}=S_{0}$. Next, using Proposition \ref{main} with $A=S_{0}$,
we obtain
\[
\lim_{N\rightarrow\infty}\theta_{N}^{\mathrm{nrv}}\mathbf{r}_{N}^{S_{0}}(\xi_{N}^{x},\,\xi_{N}^{y})=b(x,\,y)\text{ \;for all }x,\,y\in S_{0}\;,
\]
Hence, the jump rate of the speeded-up trace process $\eta_{N}^{\mathcal{E}_{N}(S_{0})}(\theta_{N}^{\mathrm{nrv}}\cdot)$
converges to $b(\cdot,\,\cdot)$, by identifying $\xi_{N}^{x}$ with
$x$, which gives \eqref{fc}. This concludes Theorem \ref{t213}.

Finally, note that $\pi$ conditioned on the irreducible component
$S_{0}$ is the invariant measure of the Markov chain $Z_{1}(\cdot)$
conditioned on $S_{0}$, which is indeed the Markov chain $Y^{\mathrm{nrv}}(\cdot)$
defined in the paragraph preceding \eqref{anrv}. Thus, we can conclude
that $\pi(x)=\nu(x)$ for $x\in S_{0}$ as well. This and \eqref{e73}
finish the proof of the asymmetric case of Theorem \ref{t215}.
\end{proof}
Now, we consider the symmetric case, for which the time scale is now
$1/d_{N}$ instead of $1/(Nd_{N})$.
\begin{proof}[Proof for Theorem \ref{t214} and the symmetric case of Theorem \ref{t215}]
As in the previous proof, we use Proposition \ref{main0}; hence, we shall demonstrate
\eqref{fc} and the fact that $S_{\star}=S_{0}$.

We first prove that condensation occurs on $S_{0}$. By Proposition
\ref{main}, we know that
\[
\lim_{N\rightarrow\infty}\frac{1}{Nd_{N}}\mathbf{r}_{N}^{S}(\xi_{N}^{x},\,\xi_{N}^{y})=b(x,\,y)\text{ \;for all }x,\,y\in S\;.
\]
Here, the Markov chain $Z_{1}(\cdot)$ does not necessarily admit
a unique invariant measure. Nevertheless, all the invariant measures
of $Z_{1}(\cdot)$ do share the characteristic that they should vanish
on $S\setminus S_{0}$, which is clear from the definition of $S_{0}$.
Hence, it follows from the first statement of Lemma \ref{trrmatconv}
that
\[
\lim_{N\rightarrow\infty}\frac{\mu_{N}(\xi_{N}^{x})}{\mu_{N}(\mathcal{E}_{N})}=0\;\;\text{for all }x\in S\setminus S_{0}\;.
\]
By the above and the assumption of condensation on $S$, we have
$\lim_{N\rightarrow\infty}\mu_{N}(\mathcal{E}_{N}(S_{0}))=1$, so
that condensation occurs on $S_{0}$.

Next, using part (1) of Proposition \ref{main} with $A=S_{0}$, which
is possible since $S_{0}$ is assumed to be attracting, we obtain
that
\[
\lim_{N\rightarrow\infty}\theta_{N}^{\mathrm{rv}}\mathbf{r}_{N}^{S_{0}}(\xi_{N}^{x},\,\xi_{N}^{y})=r(x,\,y)\text{ \;for all }x,\,y\in S_{0}\;,
\]
which establishes \eqref{fc}.

Since the Markov chain $Z_{2}(\cdot)$ on $S_{0}$ with jump kernel
$b$ is irreducible by the condition of the theorem, it admits the
unique invariant measure $\nu$ on $S_{0}$. Hence, Lemma \ref{trrmatconv}
implies that
\[
\lim_{N\rightarrow\infty}\frac{\mu_{N}(\xi_{N}^{x})}{\mu_{N}(\mathcal{E}_{N}(S_{0}))}=\nu(x)\;\;\text{for all }x\in S_{0}\;.
\]
Since condensation occurs on $S_{0}$, this implies that $\lim_{N\rightarrow\infty}\mu_{N}(\xi_{N}^{x})=\nu(x)$;
thus, $S_{\star}=S_{0}$. Therefore, Theorem \ref{t215} is proved
for the symmetric case. Finally, Theorem \ref{t214} is concluded
via Proposition \ref{main0}.
\end{proof}

\section{\label{sec7}Condensation under Condition \textbf{(UP)}}

In this section, we establish condensation of the inclusion process
under the condition \textbf{(UP)}, i.e., prove Theorem \ref{t311}.
With this result on the occurrence of condensation, the analysis of
the metastable behavior, as well as the characterization of $S_{\star}$
and asymptotic mass of the invariant measure, follows immediately
from the results obtained in Section \ref{sec6}. We mention that
we do not have an explicit formula of the invariant measure $\mu_{N}$
for this case as well, and hence all the proof should follow the ways
that have never been explored before.

We assume the condition \textbf{(UP)} throughout this section, i.e.,
$r(x,\,y)>0$ for all $x,\,y\in S$. We start by summarizing several
sets that are repeatedly used in the proof of the main result of this
section. We refer to Figure \ref{fig4} for the illustration of these
sets.
\begin{notation}
\label{no61}Let $R$ be a non-empty subset of $S$.
\begin{itemize}
\item Define the \textit{$R$-tube} as
\[
\mathcal{A}_{N}^{R}=\{\eta\in\mathcal{H}_{N}:\eta_{x}=0\text{ for all }x\in S\setminus R\}\;.
\]
For example, $\mathcal{A}_{N}^{S}=\mathcal{H}_{N}$, $\mathcal{A}_{N}^{\{x\}}=\mathcal{E}_{N}^{x}$,
and $\mathcal{A}_{N}^{\{x,\,y\}}=\mathcal{A}_{N}^{x,\,y}$ for all
$x,\,y\in S$. In view of the last example, we can regard $\mathcal{A}_{N}^{R}$
as a natural extension of the tube $\mathcal{A}_{N}^{x,\,y}$ introduced
in Notation \ref{nota}.
\item We decompose each $R$-tube $\mathcal{A}_{N}^{R}$ into its boundary
$\partial\mathcal{A}_{N}^{R}$ and the core $\mathcal{R}_{N}^{R}$
where
\begin{align*}
\partial\mathcal{A}_{N}^{R}=\; & \{\eta\in\mathcal{A}_{N}^{R}:\eta_{x}=0\text{ for some }x\in R\}\text{ and}\\
\mathcal{R}_{N}^{R}=\; & \{\eta\in\mathcal{A}_{N}^{R}:\eta_{x}>0\text{ for all }x\in R\}\;.
\end{align*}
 For example, we have $\partial\mathcal{A}_{N}^{\{x,\,y\}}=\mathcal{E}_{N}^{x}\cup\mathcal{E}_{N}^{y}$
and $\mathcal{R}_{N}^{\{x,\,y\}}=\widehat{\mathcal{A}}_{N}^{x,\,y}$.
\item We further decompose the core $\mathcal{R}_{N}^{R}$ into the inner
core $\mathcal{I}_{N}^{R}$ and the outer core $\mathcal{O}_{N}^{R}$
where
\begin{align*}
\mathcal{I}_{N}^{R} & =\{\eta\in\mathcal{R}_{N}^{R}:\eta_{x}>\epsilon\log N\text{ for all }x\in R\}\text{\;,}\\
\mathcal{O}_{N}^{R} & =\{\eta\in\mathcal{R}_{N}^{R}:\eta_{x}\le\epsilon\log N\text{ for some }x\in R\}\;,
\end{align*}
where $\epsilon$ is a small enough number that will be specified
later (cf. \eqref{epscondition}). We stress that $\epsilon$ does
not depend on $N$. For the convenience of notation, we assume in
this and the next subsections that $\epsilon\log N$ is an integer.
(For general case, it suffices to replace all $\epsilon\log N$ below
with $\lfloor\epsilon\log N\rfloor$.) For instance, a configuration
$\eta$ belonging to $\mathcal{I}_{N}^{R}$ does not have particles
at $S\setminus R$ while have more than $\epsilon\log N$ particles
at each site of $R$. Summing up, we decompose each $R$-tube $\mathcal{A}_{N}^{R}$
into the following disjoint union:
\begin{equation}
\mathcal{A}_{N}^{R}=\partial\mathcal{A}_{N}^{R}\cup\mathcal{O}_{N}^{R}\cup\mathcal{I}_{N}^{R}\;.\label{deca}
\end{equation}
\item Write $|S|=\kappa$. For $1\le k\le\kappa$, we define
\[
\mathcal{B}_{N}^{k}=\bigcup_{R\subseteq S,\,|R|=k}\mathcal{A}_{N}^{R}\;.
\]
Namely, $\mathcal{B}_{N}^{k}$ is a collection of configurations that
have at most $k$ sites with at least one particle. For instance,
$\mathcal{B}_{N}^{1}=\mathcal{E}_{N}$, $\mathcal{B}_{N}^{2}=\mathcal{A}_{N}$
(by the assumption \textbf{(UP)}), and $\mathcal{B}_{N}^{\kappa}=\mathcal{H}_{N}$.
\end{itemize}
\end{notation}

\begin{figure}
\includegraphics[scale=0.22]{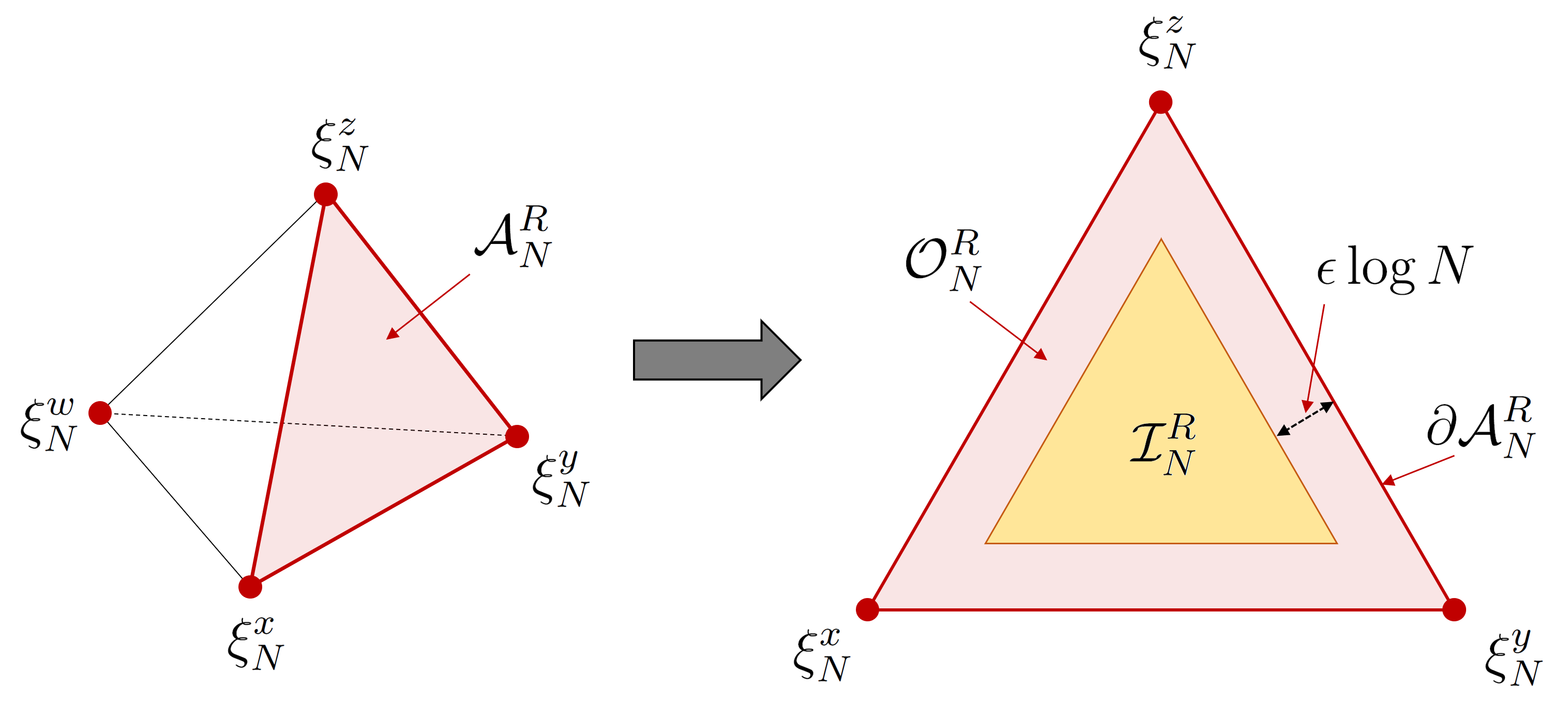}\caption{\label{fig4}Visualization of the notation introduced in Notation
\ref{no61} when $S=\{x,\,y,\,z,\,w\}$ and $R=\{x,\,y,\,z\}$. }
\end{figure}

In this section, we are mainly focusing on the following proposition.
\begin{prop}
\label{trrposprop}Suppose that \eqref{fcv} holds. Then, for all
$\ell\in\llbracket2,\,\kappa\rrbracket$, we have that
\begin{equation}
\lim_{N\rightarrow\infty}\frac{\mu_{N}(\mathcal{B}_{N}^{\ell})}{\mu_{N}(\mathcal{B}_{N}^{\ell-1})}=1\;.\label{trrpospropeq}
\end{equation}
\end{prop}

\begin{proof}
We explain the proof based on the results that will be proved in the
remaining part of this section. We prove this proposition by means
of the backward induction on $\ell$ from $\kappa$ to $2$. We note
that the initial case $\ell=\kappa$ is proved by Propositions \ref{rmndoutest}
and \ref{rmndinest} (cf. discussion between \eqref{eis} and \eqref{eis2}).
Then, by the induction step proved in Proposition \ref{indstepconcl},
the assertion of the proposition holds for all $\ell\in\llbracket2,\,\kappa\rrbracket$.
\end{proof}
With this proposition, Theorem \ref{t311} is immediate.
\begin{proof}[Proof of Theorem \ref{t311}]
Since $\mathcal{B}_{N}^{1}=\mathcal{E}_{N}$ and $\mathcal{B}_{N}^{\kappa}=\mathcal{H}_{N}$,
it suffices to check that
\begin{equation}
\lim_{N\rightarrow\infty}\frac{\mu_{N}(\mathcal{B}_{N}^{\kappa})}{\mu_{N}(\mathcal{B}_{N}^{1})}=1\;.\label{teq}
\end{equation}
This is immediate from \eqref{trrpospropeq} and we are done.
\end{proof}
Now, we explain our plan to prove the detailed ingredients appeared
in the proof of Proposition \ref{trrposprop}. The initial step $\ell=\kappa$
for the backward induction is proved in Sections \ref{sec71} and
\ref{sec72}, and the induction step is established in Section \ref{sec73}.
For the proof of these steps, an auxiliary Markov chain introduced
in Definition \ref{def68} of Section \ref{sec72} is crucially used.
As a by-product of our investigation of the hitting time of this Markov
chain carried out in Lemma \ref{hitest}, we prove the nucleation
result presented as Theorem \ref{t211} in Section \ref{sec74} as
well.

\subsection{\label{sec71}Initial step (1): negligibility of the outer core }

Now, we prove the case $\ell=\kappa$ for Proposition \ref{trrposprop}.
Since $\mathcal{B}_{N}^{\kappa}=\mathcal{A}_{N}^{S}(=\mathcal{H}_{N})$
and $\mathcal{B}_{N}^{\kappa-1}=\partial\mathcal{A}_{N}^{S}$ by the
definition of the boundary, it suffices to prove
\begin{equation}
\lim_{N\rightarrow\infty}\frac{\mu_{N}(\mathcal{A}_{N}^{S})}{\mu_{N}(\partial\mathcal{A}_{N}^{S})}=1\;.\label{eis}
\end{equation}
Since $\mu_{N}(\mathcal{A}_{N}^{S})=\mu_{N}(\partial\mathcal{A}_{N}^{S})+\mu_{N}(\mathcal{O}_{N}^{S})+\mu_{N}(\mathcal{I}_{N}^{S})$
by \eqref{deca}, it suffices to prove that
\begin{equation}
\lim_{N\rightarrow\infty}\frac{\mu_{N}(\mathcal{O}_{N}^{S})}{\mu_{N}(\partial\mathcal{A}_{N}^{S})}=0\;\;\;\text{and\;\;\;}\lim_{N\rightarrow\infty}\frac{\mu_{N}(\mathcal{I}_{N}^{S})}{\mu_{N}(\partial\mathcal{A}_{N}^{S})}=0\;.\label{eis2}
\end{equation}
The proof of The latter one is considered in the next subsection,
and we focus only on the former one in the current subsection. Thus,
the main object now is to prove the following proposition.
\begin{prop}
\label{rmndoutest}Suppose that $\lim_{N\rightarrow\infty}d_{N}N=0$.
Then, for sufficiently small $\epsilon>0$, we have that
\[
\lim_{N\rightarrow\infty}\frac{\mu_{N}(\mathcal{O}_{N}^{S})}{\mu_{N}(\partial\mathcal{A}_{N}^{S})}=0\;.
\]
\end{prop}

To prove this, we decompose the outer core $\mathcal{O}_{N}^{S}$
into more refined objects, and estimate each of them carefully.

\subsubsection*{Decomposition of outer core}

For $x\in R\subseteq S$ and $k\in\llbracket1,\,N\rrbracket$, we
define
\[
\mathcal{C}_{N}^{R}(x,\,k)=\{\eta\in\mathcal{A}_{N}^{R}:\eta_{x}=k\}\;.
\]
For instance, we have
\begin{equation}
\mathcal{C}_{N}^{R}(x,\,N)=\mathcal{E}_{N}^{x}\;\;\text{and\;\;}\mathcal{C}_{N}^{R}(x,\,0)=\mathcal{A}_{N}^{R\setminus\{x\}}\;.\label{exc}
\end{equation}
Then, it holds that
\begin{equation}
\mathcal{O}_{N}^{R}\subseteq\bigcup_{x\in R}\bigcup_{k=1}^{\epsilon\log N}\mathcal{C}_{N}^{R}(x,\,k)\;;\label{rmndoutsbst}
\end{equation}
thus,
\begin{equation}
\mu_{N}(\mathcal{O}_{N}^{S})\le\sum_{x\in S}\sum_{k=1}^{\epsilon\log N}\mu_{N}(\mathcal{C}_{N}^{S}(x,\,k))\;.\label{e67}
\end{equation}
Hence, it suffices to estimate $\mu_{N}(\mathcal{C}_{N}^{S}(x,\,k))$
for $k\in\llbracket1,\,\epsilon\log N\rrbracket$ and $x\in S$.

\subsubsection*{Estimation of $\mu_{N}(\mathcal{C}_{N}^{S}(x,\,k))$}

For $k\in\llbracket0,\,N-1\rrbracket$ and $x\in R\subseteq S$, we
define
\begin{align*}
\mathbf{F}_{N}^{R}(x;\,k\rightarrow k+1)= & \sum_{\eta\in\mathcal{C}_{N}^{R}(x,\,k),\,\,\zeta\in\mathcal{C}_{N}^{R}(x,\,k+1)}\mu_{N}(\eta)\mathbf{r}_{N}(\eta,\,\zeta)\;,\\
\mathbf{F}_{N}^{R}(x;\,k+1\rightarrow k)= & \sum_{\eta\in\mathcal{C}_{N}^{R}(x,\,k+1),\,\,\zeta\in\mathcal{C}_{N}^{R}(x,\,k),}\mu_{N}(\eta)\mathbf{r}_{N}(\eta,\,\zeta)\;.
\end{align*}

\begin{lem}
\label{flowsym} For all $k\in\llbracket0,\,N-1\rrbracket$ and $x\in S$,
it holds that
\[
\mathbf{F}_{N}^{S}(x;\,k\rightarrow k+1)=\mathbf{F}_{N}^{S}(x;\,k+1\rightarrow k)\;.
\]
\end{lem}

\begin{proof}
Since $\mu_{N}$ is the invariant measure for the inclusion process,
we have that,
\begin{equation}
\sum_{x,\,y\in S}\mu_{N}(\eta)\mathbf{r}_{N}(\eta,\,\sigma^{x,\,y}\eta)=\sum_{x,\,y\in S}\mu_{N}(\sigma^{x,\,y}\eta)\mathbf{r}_{N}(\sigma^{x,\,y}\eta,\,\eta)\;\;\text{for all }\eta\in\mathcal{H}_{N}\;.\label{invmsr}
\end{equation}
Here, we use the convention that $\mathbf{r}_{N}(\eta,\,\eta)=0$
for $\eta\in\mathcal{H}_{N}$. By summing \eqref{invmsr} over $\eta\in\mathcal{C}_{N}^{S}(x,\,k)$,
\begin{align}
 & \sum_{\eta\in\mathcal{C}_{N}^{S}(x,\,k)}\sum_{y,\,z\in S\setminus\{x\}}\mu_{N}(\eta)\mathbf{r}_{N}(\eta,\,\sigma^{y,\,z}\eta)+\sum_{\eta\in\mathcal{C}_{N}^{S}(x,\,k)}\sum_{y\in S\setminus\{x\}}\mu_{N}(\eta)\mathbf{r}_{N}(\eta,\,\sigma^{y,\,x}\eta)\nonumber \\
 & +\sum_{\eta\in\mathcal{C}_{N}^{S}(x,\,k)}\sum_{y\in S\setminus\{x\}}\mu_{N}(\eta)\mathbf{r}_{N}(\eta,\,\sigma^{x,\,y}\eta)\nonumber \\
=\; & \sum_{\eta\in\mathcal{C}_{N}^{S}(x,\,k)}\sum_{y,\,z\in S\setminus\{x\}}\mu_{N}(\sigma^{y,\,z}\eta)\mathbf{r}_{N}(\sigma^{y,\,z}\eta,\,\eta)+\sum_{\eta\in\mathcal{C}_{N}^{S}(x,\,k)}\sum_{y\in S\setminus\{x\}}\mu_{N}(\sigma^{y,\,x}\eta)\mathbf{r}_{N}(\sigma^{y,\,x}\eta,\,\eta)\nonumber \\
 & +\sum_{\eta\in\mathcal{C}_{N}^{S}(x,\,k)}\sum_{y\in S\setminus\{x\}}\mu_{N}(\sigma^{x,\,y}\eta)\mathbf{r}_{N}(\sigma^{x,\,y}\eta,\,\eta)\;.\label{slicetemp}
\end{align}
Note that the first summations in the respective sides of \eqref{slicetemp}
are canceled out with each other. Therefore, \eqref{slicetemp} can
be simply rewritten as
\begin{equation}
\mathbf{F}_{N}^{S}(x;\,k\rightarrow k+1)+\mathbf{F}_{N}^{S}(x;\,k\rightarrow k-1)=\mathbf{F}_{N}^{S}(x;\,k\rightarrow k+1)+\mathbf{F}_{N}^{S}(x;\,k-1\rightarrow k)\;,\label{slicetemp2}
\end{equation}
where $\mathbf{F}_{N}^{S}(x;\,-1\rightarrow0)$ and $\mathbf{F}_{N}^{S}(x;\,0\rightarrow-1)$
are defined to be $0$. Inserting $k=0$ to \eqref{slicetemp2} implies
\begin{equation}
\mathbf{F}_{N}^{S}(x;\,0\rightarrow1)=\mathbf{F}_{N}^{S}(x;\,1\rightarrow0)\;.\label{slicetemp3}
\end{equation}
Therefore, \eqref{slicetemp2} and \eqref{slicetemp3} along with
induction on $k$ finish the proof.
\end{proof}
\begin{lem}
\label{sliceest}For $k\in\llbracket0,\,N-1\rrbracket$ and $x\in S$,
we have that
\[
\mu_{N}(\mathcal{C}_{N}^{S}(x,\,k+1))\le\frac{R_{2}(k+d_{N})(N-k)}{R_{1}(k+1)(N-k-1+d_{N})}\mu_{N}(\mathcal{C}_{N}^{S}(x,\,k))\;,
\]
where the constants $R_{1}$ and $R_{2}$ are introduced in \eqref{const}.
\end{lem}

\begin{proof}
Looking at $\mathbf{F}_{N}^{S}(x;\,k\rightarrow k+1)$ more carefully,
we get the following bound:
\begin{align}
\mathbf{F}_{N}^{S}(x;\,k\rightarrow k+1) & =\sum_{\eta\in\mathcal{C}_{N}^{S}(x,\,k),\,\,\zeta\in\mathcal{C}_{N}^{S}(x,\,k+1)}\mu_{N}(\eta)\mathbf{r}_{N}(\eta,\,\zeta)\nonumber \\
 & =\sum_{\eta\in\mathcal{C}_{N}^{S}(x,\,k)}\,\sum_{y\in S\setminus\{x\}}\mu_{N}(\eta)\mathbf{r}_{N}(\eta,\,\sigma^{y,\,x}\eta)\nonumber \\
 & =\sum_{\eta\in\mathcal{C}_{N}^{S}(x,\,k)}\bigg[\mu_{N}(\eta)\sum_{y\in S\setminus\{x\}}\eta_{y}(d_{N}+\eta_{x})r(y,\,x)\bigg]\nonumber \\
 & =(k+d_{N})\sum_{\eta\in\mathcal{C}_{N}^{S}(x,\,k)}\mu_{N}(\eta)\sum_{y\in S\setminus\{x\}}r(y,\,x)\eta_{y}\nonumber \\
 & \le R_{2}(k+d_{N})(N-k)\,\mu_{N}(\mathcal{C}_{N}^{S}(x,\,k))\;.\label{slicetemp4}
\end{align}
Similarly, we can get
\begin{align}
\mathbf{F}_{N}^{S}(x;\,k+1\rightarrow k) & =\sum_{\eta\in\mathcal{C}_{N}^{S}(x,\,k+1)}\sum_{y\in S\setminus\{x\}}\mu_{N}(\eta)\mathbf{r}_{N}(\eta,\,\sigma^{x,\,y}\eta)\nonumber \\
 & =\sum_{\eta\in\mathcal{C}_{N}^{S}(x,\,k+1)}\bigg[\mu_{N}(\eta)\sum_{y\in S\setminus\{x\}}\eta_{x}(d_{N}+\eta_{y})r(x,\,y)\bigg]\nonumber \\
 & \ge R_{1}(k+1)(N-k-1+d_{N})\,\mu_{N}(\mathcal{C}_{N}^{S}(x,\,k+1))\;.\label{slicetemp5}
\end{align}
Combining \eqref{slicetemp4}, \eqref{slicetemp5} with Lemma \ref{flowsym},
we can complete the proof of the lemma.
\end{proof}
In the proof above, it is crucial to have $r(x,\,y)>0$ for all $x,\,y\in S$
to deduce \eqref{slicetemp5}. Hence, the condition\textbf{ (UP)}
is critically used.
\begin{lem}
\label{lem66}For sufficiently small $\epsilon>0$, we have that
\[
\sum_{k=1}^{\epsilon\log N}\mu_{N}(\mathcal{C}_{N}^{S}(x,\,k))\le O(Nd_{N})\,\mu_{N}(\mathcal{A}_{N}^{S\setminus\{x\}})\;.
\]
\end{lem}

\begin{proof}
Inserting $k=0$ to Lemma \ref{sliceest} yields that, for some constant
$C_{1}>0$,
\begin{equation}
\mu_{N}(\mathcal{C}_{N}^{S}(x,\,1))\le C_{1}d_{N}\mu_{N}(\mathcal{C}_{N}^{S}(x,\,0))\;,\label{ermn1}
\end{equation}
while inserting $k\in\llbracket1,\,N-2\rrbracket$ provides us that
for some constant $C_{2}>0$,
\begin{equation}
\mu_{N}(\mathcal{C}_{N}^{S}(x,\,k+1))\le C_{2}\mu_{N}(\mathcal{C}_{N}^{S}(x,\,k))\;.\label{ermn2}
\end{equation}
Let $C_{0}=\max\{C_{1},\,C_{2}\}$. Then, \eqref{ermn1} and \eqref{ermn2}
imply that
\begin{equation}
\mu_{N}(\mathcal{C}_{N}^{S}(x,\,k))\le C_{0}^{k}d_{N}\mu_{N}(\mathcal{C}_{N}^{S}(x,\,0))\;\;\text{for }k\in\llbracket1,\,N-1\rrbracket\;.\label{slicetemp6}
\end{equation}
Summing this up for $k=1,\,2,\,\dots,\,\epsilon\log N$, we get
\begin{equation}
\sum_{k=1}^{\epsilon\log N}\mu_{N}(\mathcal{C}_{N}^{S}(x,\,k))\le\frac{C_{0}^{(\epsilon\log N)+1}-C_{0}}{C_{0}-1}d_{N}\mu_{N}(\mathcal{C}_{N}^{S}(x,\,0))\;.\label{epsilon}
\end{equation}
Take $\epsilon$ small enough so that
\begin{equation}
\frac{C_{0}^{(\epsilon\log N)+1}-C_{0}}{C_{0}-1}=O(N)\;.\label{epscondition}
\end{equation}
The proof is completed since $\mathcal{C}_{N}^{S}(x,\,0)=\mathcal{A}_{N}^{S\setminus\{x\}}$
by \eqref{exc}
\end{proof}
Now, we are ready to prove the main goal of this subsection.
\begin{proof}[Proof of Proposition \ref{rmndoutest}]
 By \eqref{rmndoutsbst} and the previous lemma, we get
\[
\mu_{N}(\mathcal{O}_{N}^{S})\le\sum_{x\in S}\sum_{k=1}^{\epsilon\log N}\mu_{N}(\mathcal{C}_{N}^{S}(x,\,k))\le Cd_{N}N\sum_{x\in S}\mu_{N}(\mathcal{A}_{N}^{S\setminus\{x\}})\;.
\]
The proof is completed since
\[
\sum_{x\in S}\mu_{N}(\mathcal{A}_{N}^{S\setminus\{x\}})=\mu_{N}(\partial\mathcal{A}_{N}^{S})\;,
\]
and since $\lim_{N\rightarrow\infty}d_{N}N=0$ by the assumption of
the proposition.
\end{proof}

\subsection{\label{sec72}Initial step (2): negligibility of the inner core }

In this subsection, we prove the negligibility of the inner core $\mathcal{I}_{N}^{S}$
via the following proposition.
\begin{prop}
\label{rmndinest}Suppose that $\lim_{N\rightarrow\infty}d_{N}N^{\kappa+2}(\log N)^{\kappa-3}=0$.
Then, we have that
\[
\lim_{N\rightarrow\infty}\frac{\mu_{N}(\mathcal{I}_{N}^{S})}{\mu_{N}(\partial\mathcal{A}_{N}^{S})}=0\;.
\]
\end{prop}

The proof of this part is more demanding than that of the outer core,
and we have to introduce a sequence of new concepts.

Define the closure and the (outer) boundary of $\mathcal{I}_{N}^{R}$
for $R\subseteq S$ as
\begin{align*}
\overline{\mathcal{I}}_{N}^{R} & =\{\eta\in\mathcal{A}_{N}^{R}:\mathbf{r}_{N}(\zeta,\,\eta)>0\text{ for some }\zeta\in\mathcal{I}_{N}^{R}\}\;,\\
\partial\mathcal{I}_{N}^{R} & =\overline{\mathcal{I}}_{N}^{R}\setminus\mathcal{I}_{N}^{R}\;.
\end{align*}
Thus, $\overline{\mathcal{I}}_{N}^{S}$ consists of configurations
$\eta$ such that $\eta_{x}\ge\epsilon\log N$ for all $x$ and there
exists \textit{at most} one $x\in S$ with $\eta_{x}=\epsilon\log N$,
while $\partial\mathcal{I}_{N}^{S}$ consists of configurations $\eta$
such that $\eta_{x}\ge\epsilon\log N$ for all $x$ and there exist
\textit{exactly} one $x\in S$ with $\eta_{x}=\epsilon\log N$. Therefore,
we have the following decomposition for $\partial\mathcal{I}_{N}^{S}$
\[
\partial\mathcal{I}_{N}^{S}\subseteq\bigcup_{x\in S}\mathcal{C}_{N}^{S}(x,\,\epsilon\log N)\;.
\]
Therefore, by \eqref{slicetemp6} and \eqref{epscondition}, we have
that
\begin{equation}
\mu_{N}(\partial\mathcal{I}_{N}^{S})\le\sum_{x\in S}\mu_{N}(\mathcal{C}_{N}^{S}(x,\,\epsilon\log N))\le CNd_{N}\sum_{x\in S}\mu_{N}(\mathcal{C}_{N}^{S}(x,\,0))=O(Nd_{N})\,\mu_{N}(\partial\mathcal{A}_{N}^{S})\;.\label{e620}
\end{equation}
Hence, Proposition \ref{rmndinest} is the consequence of the following
proposition.
\begin{prop}
\label{p67}Suppose that $\lim_{N\rightarrow\infty}d_{N}N^{\kappa+2}(\log N)^{\kappa-3}=0$.
Then, we have that
\begin{equation}
\mu_{N}(\mathcal{I}_{N}^{S})=O(N^{\kappa-2}(\log N)^{\kappa})\mu_{N}(\partial\mathcal{I}_{N}^{S})\;.\label{e621}
\end{equation}
\end{prop}

\begin{proof}[Proof of Proposition \ref{rmndinest}]
By \eqref{e620} and \eqref{e621}, we get that
\[
\mu_{N}(\mathcal{I}_{N}^{S})=O(N^{\kappa-2}(\log N)^{\kappa})\mu_{N}(\partial\mathcal{I}_{N}^{S})=O(N^{\kappa-1}(\log N)^{\kappa}d_{N})\,\mu_{N}(\partial\mathcal{A}_{N}^{S})\;.
\]
By the condition $d_{N}N^{\kappa+2}(\log N)^{\kappa-3}=o_{N}(1)$,
we are done.
\end{proof}
The remaining part of this subsection is devoted to prove Proposition
\ref{p67}.

\subsubsection*{Auxiliary Markov chain $\widehat{\eta}_{N}^{R}(\cdot)$ and its hitting
time estimate}

The crucial ingredient in the proof of Proposition \ref{p67} is an
auxiliary discrete time Markov chain on $\overline{\mathcal{I}}_{N}^{S}=\mathcal{I}_{N}^{S}\cup\partial\mathcal{I}_{N}^{S}$
and the estimate of the hitting time of the set $\partial\mathcal{I}_{N}^{S}$
when the chain starts from $\mathcal{I}_{N}^{S}$. To use these results
at the induction step in Section \ref{sec73}, we will work on $\overline{\mathcal{I}}_{N}^{R}$
for $R\subseteq S$.
\begin{defn}
\label{def68}For $R\subseteq S$, let $(\widehat{\eta}_{N}^{R}(t))_{t\in\mathbb{N}}$
denote the discrete-time Markov chain on $\overline{\mathcal{I}}_{N}^{R}$
whose transition probability $\widehat{\mathbf{p}}_{N}^{R}$ is given
by
\begin{equation}
\widehat{\mathbf{p}}_{N}^{R}(\eta,\,\sigma^{x,\,y}\eta)=\frac{\eta_{y}(d_{N}+\eta_{x})\,r(y,\,x)}{\sum_{a,\,b\in R}\eta_{a}(d_{N}+\eta_{b})\,r(a,\,b)}\;\;\;\;\;\text{for}\;\eta,\,\sigma^{x,\,y}\eta\in\overline{\mathcal{I}}_{N}^{R}\;,\label{pnr}
\end{equation}
and set
\[
\widehat{\mathbf{p}}_{N}^{R}(\eta,\,\eta)=1-\sum_{\zeta:\zeta\neq\eta}\widehat{\mathbf{p}}_{N}^{R}(\eta,\,\zeta)\;\;\;\;\;\text{for }\eta\in\partial\mathcal{I}_{N}^{R}\;.
\]
In other words, $\widehat{\eta}_{N}^{R}(\cdot)$ is attained from
the discrete version of the inclusion process by changing the jump
rate $\eta_{x}(d_{N}+\eta_{y})r(x,\,y)$ to $\eta_{y}(d_{N}+\eta_{x})r(y,\,x)$
and then restrict to $\overline{\mathcal{I}}_{N}^{R}$. This chain
is well-defined since $\eta_{x},\,\eta_{y}>0$ for $\eta\in\overline{\mathcal{I}}_{N}^{R}$.

Let $\widehat{\mathbf{L}}_{N}^{R}$ denote the corresponding generator
and by $\widehat{\mathbb{E}}_{\eta}^{R}$ the expectation with respect
to the chain $\widehat{\eta}_{N}^{R}(\cdot)$ starting from $\eta\in\overline{\mathcal{I}}_{N}^{R}$.
Finally, let $\sigma_{R}:=\tau_{\partial\mathcal{I}_{N}^{R}}$ be
the hitting time the set $\partial\mathcal{I}_{N}^{R}$ by the chain
$\widehat{\eta}_{N}^{R}(\cdot)$. \textbf{\textit{Then, the primary
purpose is to estimate $\widehat{\mathbb{E}}_{\eta}^{R}\left[\sigma_{R}\right]$
for $\eta\in\mathcal{I}_{N}^{R}$.}} The crucial step for this estimate
is the following construction of a test function, which based on the
so-called Gordan's lemma.
\end{defn}

\begin{lem}
\label{lem710}Suppose that $R\subseteq S$ and $\lim_{N\rightarrow\infty}d_{N}\frac{N^{2}}{(\log N)^{2}}=0$.
Then, there exist a constant $C=C(\epsilon)>0$ and a test function
$\mathbf{f}_{0}=\mathbf{f}_{0}^{R}:\overline{\mathcal{I}}_{N}^{R}\rightarrow\mathbb{R}$
such that
\begin{align}
\max_{\overline{\mathcal{I}}_{N}^{R}}\mathbf{f}_{0}-\min_{\overline{\mathcal{I}}_{N}^{R}}\mathbf{f}_{0} & \le C\log N\;,\text{ and}\label{testfcncond}\\
(\widehat{\mathbf{L}}_{N}^{R}\mathbf{f}_{0})(\eta) & \ge\frac{\log N}{CN^{3}}\text{\; for all }\eta\in\mathcal{I}_{N}^{R}\;.\label{testfcncond3}
\end{align}
\end{lem}

\begin{proof}
Fix $R\subseteq S$ and consider a $|R|\times|R|$ skew-symmetric
matrix $\boldsymbol{Q}$ defined by
\[
\boldsymbol{Q}_{x,\,y}=r(x,\,y)-r(y,\,x)\;,\quad x\in R,\,y\in R\;.
\]
By Gordan's lemma stated in Lemma \ref{Gordan} at the appendix, we
have that
\begin{equation}
\text{\ensuremath{\exists}}\,\boldsymbol{\alpha}=(\alpha_{x})_{x\in R}\in\mathbb{R}^{|R|}\text{ such that }(\boldsymbol{Q}\boldsymbol{\alpha})_{1},\,\dots,\,(\boldsymbol{Q}\boldsymbol{\alpha})_{|R|}<0\;,\label{GordanQa}
\end{equation}
or
\begin{equation}
\exists\,\boldsymbol{\beta}(\neq\boldsymbol{0})=(\beta_{x})_{x\in R}\in\mathbb{R}^{|R|}\text{ so that }\beta_{1},\,\dots,\,\beta_{|R|}\le0\text{ and }\boldsymbol{Q}\boldsymbol{\beta}=0\;.\label{GordanQb}
\end{equation}
We consider these two cases separately.\\
\\
\textbf{(Case 1: }\eqref{GordanQa}\textbf{)} In this case, define
$\mathbf{f}_{0}:\overline{\mathcal{I}}_{N}^{R}\rightarrow\mathbb{R}$
as
\[
\mathbf{f}_{0}(\eta)=\sum_{x\in R}\alpha_{x}\bigg(1+\frac{1}{2}+\cdots+\frac{1}{\eta_{x}}\bigg)\;.
\]
Then, for each $\eta\in\overline{\mathcal{I}}_{N}^{R}$,
\[
|\mathbf{f}_{0}(\eta)|\le C\sum_{x\in R}|\alpha_{x}|\log\eta_{x}\le C'\log N\;;
\]
hence; the condition \eqref{testfcncond} follows immediately. To
check the condition \eqref{testfcncond3}, we define
\[
\mathbf{w}(\eta)=\sum_{a,\,b\in R}\eta_{a}(d_{N}+\eta_{b})\,r(a,\,b)\;,
\]
 so that
\begin{align}
(\widehat{\mathbf{L}}_{N}^{R}\mathbf{f}_{0})(\eta)=\; & \frac{1}{\mathbf{w}(\eta)}\sum_{x,\,y\in R}\eta_{y}(d_{N}+\eta_{x})\,r(y,\,x)\bigg(\frac{\alpha_{y}}{\eta_{y}+1}-\frac{\alpha_{x}}{\eta_{x}}\bigg)\nonumber \\
=\; & \frac{1}{\mathbf{w}(\eta)}\bigg\{\sum_{x,\,y\in R}r(y,\,x)\eta_{x}\eta_{y}\bigg(\frac{\alpha_{y}}{\eta_{y}+1}-\frac{\alpha_{x}}{\eta_{x}}\bigg)+O\big(d_{N}\frac{N}{\log N}\big)\bigg\}\nonumber \\
=\; & \frac{1}{\mathbf{w}(\eta)}\bigg\{\sum_{x,\,y\in R}r(y,\,x)\eta_{x}\eta_{y}\bigg(\frac{\alpha_{y}}{\eta_{y}}-\frac{\alpha_{x}}{\eta_{x}}\bigg)+O\big(\frac{N}{\log N}\big)+O\big(d_{N}\frac{N}{\log N}\big)\bigg\}\;.\label{elnf}
\end{align}
The seemingly not so serious last identity is indeed the main reason
that we introduced the inner core $\mathcal{I}_{N}^{R}.$ The error
coming from this identity is not able to control if $\eta_{y}$ is
close to $0$. In this case the bound $\eta_{y}\ge\epsilon\log N$
provides us the small error term of $O(N/\log N)$.

Now the last summation can be computed as
\begin{align*}
\sum_{x,\,y\in R}r(y,\,x)(\alpha_{y}\eta_{x}-\alpha_{x}\eta_{y})=\; & \sum_{x\in R}\bigg[\eta_{x}\sum_{y\in R}\alpha_{y}\{r(y,\,x)-r(x\,,y)\}\bigg]\\
=\; & \sum_{x\in R}\bigg[\eta_{x}\sum_{y\in R}-\boldsymbol{Q}_{x,\,y}\alpha_{y}\bigg]= \sum_{x\in R}\eta_{x}(-\boldsymbol{Q}\boldsymbol{\alpha})_{x}\ge\frac{N}{C}\;.
\end{align*}
where the last inequality is due to \eqref{GordanQa}. Since $\mathbf{w}(\eta)=O(N^{2})$,
applying the last inequality to \eqref{elnf} verifies the condition
\eqref{testfcncond3} because clearly $1/N\ge1/N^3$.\\
\\
\textbf{(Case 2: }\eqref{GordanQb}\textbf{)} Define $\mathbf{f}_{0}:\overline{\mathcal{I}}_{N}^{R}\rightarrow\mathbb{R}$
by
\[
\mathbf{f}_{0}(\eta)=\sum_{x\in R}\beta_{x}\bigg(1+\frac{1}{2}+\cdots+\frac{1}{\eta_{x}}\bigg)\;.
\]
Then, the condition \eqref{testfcncond} follows similarly as \textbf{(Case
1)}. By a similar calculation, for $\eta\in\mathcal{I}_{N}^{R}$,
\begin{align*}
 & (\widehat{\mathbf{L}}_{N}^{R}\mathbf{f}_{0})(\eta)\\
=\; & \frac{1}{\mathbf{w}(\eta)}\,\sum_{x,\,y\in R}\eta_{y}(d_{N}+\eta_{x})\,r(y,\,x)\bigg(\frac{\beta_{y}}{\eta_{y}+1}-\frac{\beta_{x}}{\eta_{x}}\bigg)\\
=\; & \frac{1}{\mathbf{w}(\eta)}\,\bigg[\sum_{x,\,y\in R}r(y,\,x)\,\eta_{x}\eta_{y}\,\bigg(\frac{\beta_{y}}{\eta_{y}}-\frac{\beta_{x}}{\eta_{x}}\bigg)+\sum_{x,\,y\in R}r(y,\,x)\frac{-\beta_{y}\eta_{x}}{\eta_{y}+1}+O\big(d_{N}\frac{N}{\log N}\big)\bigg]\\
=\; & \frac{1}{\mathbf{w}(\eta)}\,\bigg[\sum_{x\in R}\eta_{x}(-\boldsymbol{Q}\boldsymbol{\beta})_{x}+\sum_{x,\,y\in R}r(y,\,x)\frac{\eta_{x}}{\eta_{y}+1}(-\beta_{y})+O\big(d_{N}\frac{N}{\log N}\big)\bigg]\;.
\end{align*}
The first summation in the last line vanishes since $\boldsymbol{Q}\boldsymbol{\beta}=0$.
Hence by \eqref{GordanQb},
\begin{equation}
(\widehat{\mathbf{L}}_{N}^{R}\mathbf{f}_{0})(\eta)\ge\frac{1}{\mathbf{w}(\eta)}\bigg[\frac{1}{C}\sum_{x,\,y\in R}\frac{\eta_{x}}{\eta_{y}+1}(-\beta_{y})+O\big(d_{N}\frac{N}{\log N}\big)\bigg]\ge\frac{1}{C'}N^{-3}\log N\;,\label{ego1}
\end{equation}
where the last inequality holds because $\mathbf{w}(\eta)=O(N^{2})$,
and $\lim_{N\rightarrow\infty}d_N\frac{N^2}{(\log N)^2}=0$ along with $\eta\in\mathcal{I}_N^R$ imply that the first term inside the bracket dominates the second term.
\end{proof}
We remark that, at the first inequality of \eqref{ego1}, the condition\textbf{
(UP) }is strongly used again. Now, we estimate the expectation of
the hitting time $\sigma_{R}$ of the outer boundary $\partial\mathcal{I}_{N}^{R}$.
\begin{lem}
\label{hitest} Suppose that $R\subseteq S$ and that $\lim_{N\rightarrow\infty}d_{N}\frac{N^{2}}{(\log N)^{2}}=0$.
Then, there exists $C=C(\epsilon)>0$ such that
\[
\sup_{\eta\in\mathcal{I}_{N}^{R}}\widehat{\mathbb{E}}_{\eta}^{R}\left[\sigma_{R}\right]\le CN^{3}\;.
\]
\end{lem}

\begin{proof}
For $\mathbf{f}:\overline{\mathcal{I}}_{N}^{R}\rightarrow\mathbb{R}$,
we know that
\[
\mathscr{M}(n)=\mathbf{f}(\widehat{\eta}_{N}^{R}(n))-\mathbf{f}(\widehat{\eta}_{N}^{R}(0))-\sum_{k=0}^{n-1}(\widehat{\mathbf{L}}_{N}^{R}\mathbf{f})(\widehat{\eta}_{N}^{R}(k))\;\;\;;\;n\in\mathbb{N}
\]
is a discrete-time martingale with initial value $0$. Therefore,
by the optional stopping theorem, we have for all $\eta\in\mathcal{I}_{N}^{R}$
and $n\ge0$ that
\begin{equation}
\widehat{\mathbb{E}}_{\eta}^{R}\left[\mathbf{f}(\widehat{\eta}_{N}^{R}(\sigma_{R}\wedge n))\right]=\mathbf{f}(\eta)+\widehat{\mathbb{E}}_{\eta}^{R}\bigg[\sum_{k=0}^{(\sigma_{R}\wedge n)-1}(\widehat{\mathbf{L}}_{N}^{R}\mathbf{f})(\widehat{\eta}_{N}^{R}(k))\bigg]\;.\label{mgprob}
\end{equation}
Now, we insert $\mathbf{f}=\mathbf{f}_{0}$ where $\mathbf{f}_{0}$
is the test function obtained in Lemma \ref{lem710}. Using the bounds
in \eqref{testfcncond} and \eqref{testfcncond3}, it holds for all
$n\ge0$ that
\begin{align*}
C\log N\ge\; & \widehat{\mathbb{E}}_{\eta}^{R}\left[\mathbf{f}_{0}(\widehat{\eta}_{N}^{R}(\sigma_{R}\wedge n))\right]-\mathbf{f}_{0}(\eta)\\
=\; & \widehat{\mathbb{E}}_{\eta}^{R}\bigg[\sum_{k=0}^{(\sigma_{R}\wedge n)-1}(\widehat{\mathbf{L}}_{N}^{R}\mathbf{f}_{0})(\widehat{\eta}_{N}^{R}(k))\bigg]\ge\frac{\log N}{CN^{3}}\widehat{\mathbb{E}}_{\eta}^{R}\left[\sigma_{R}\wedge n\right]\;.
\end{align*}
Thus, the proof is completed by letting $n\rightarrow\infty$.
\end{proof}
\begin{rem}
\label{rem611}A careful reading of the proofs shows that Lemmas \ref{lem710}
and \ref{hitest} holds for any $\epsilon>0$.
\end{rem}

\begin{lem}
\label{gmp}Fix a set $R\subseteq S$ and a constant $\delta\ge0$.
Suppose that $\lim_{N\rightarrow\infty}d_{N}\frac{N^{2}}{(\log N)^{2}}=0$,
and that a function $\mathbf{f}:\overline{\mathcal{I}}_{N}^{R}\rightarrow\mathbb{R}$
satisfies
\begin{equation}
\mathbf{f}(\eta)\le\sum_{\zeta\in\overline{\mathcal{I}}_{N}^{R}}\widehat{\mathbf{p}}_{N}^{R}(\eta,\,\zeta)\mathbf{f}(\zeta)+\delta\text{ \;\;for all }\eta\in\mathcal{I}_{N}^{R}\;.\label{gmptemp}
\end{equation}
Then, for each $\eta\in\overline{\mathcal{I}}_{N}^{R}$, we have that
\[
\mathbf{f}(\eta)\le\max_{\partial\mathcal{I}_{N}^{R}}\mathbf{f}+CN^{3}\delta\,\;.
\]
\end{lem}

\begin{proof}
Define $\mathbf{g}:\overline{\mathcal{I}}_{N}^{R}\rightarrow\mathbb{R}$
by
\[
\mathbf{g}(\eta)=\widehat{\mathbb{E}}_{\eta}^{R}\left[\mathbf{f}(\widehat{\eta}_{N}^{R}(\sigma_{R}))+\delta\sigma_{R}\right]\;.
\]
For $\eta\in\mathcal{I}_{N}^{R}$, the Markov property gives us that
\begin{align}
\mathbf{g}(\eta) & =\sum_{\zeta\in\overline{\mathcal{I}}_{N}^{R}}\widehat{\mathbf{p}}_{N}^{R}(\eta,\,\zeta)\mathbb{E}_{\zeta}\left[\mathbf{f}(\widehat{\eta}_{N}^{R}(\sigma_{R}))+\delta(\sigma_{R}+1)\right]\nonumber \\
 & =\sum_{\zeta\in\overline{\mathcal{I}}_{N}^{R}}\widehat{\mathbf{p}}_{N}^{R}(\eta,\,\zeta)\mathbf{g}(\zeta)+\delta\;.\label{gmptemp2}
\end{align}
Let $\mathbf{h}=\mathbf{f}-\mathbf{g}$. Then, by \eqref{gmptemp}
and \eqref{gmptemp2}, we have that
\[
\mathbf{h}(\eta)\le\sum_{\zeta\in\overline{\mathcal{I}}_{N}^{R}}\widehat{\mathbf{p}}_{N}^{R}(\eta,\,\zeta)\mathbf{h}(\zeta)\text{ for all }\eta\in\mathcal{I}_{N}^{R}\;.
\]
On the other hand, we have $\mathbf{h}\equiv0$ on $\partial\mathcal{I}_{N}^{R}$
since $\sigma_{R}=0$ on $\partial\mathcal{I}_{N}^{R}$. Therefore,
since $\widehat{\eta}_{N}^{R}(\cdot)$ is irreducible, the maximum
principle implies that $\mathbf{h}\le0$, i.e., $\mathbf{f}\le\mathbf{g}$
on $\overline{\mathcal{I}}_{N}^{R}$. Since $\mathbf{g}(\eta)\le\max_{\partial\mathcal{I}_{N}^{R}}\mathbf{f}+\delta\,\widehat{\mathbb{E}}_{\eta}^{R}\left[\sigma_{R}\right]$
by the definition of $\mathbf{g}$, the proof is completed by Lemma
\ref{hitest}.
\end{proof}
Now, we define $\mathbf{m}:\mathcal{H}_{N}\rightarrow\mathbb{R}$
by
\begin{equation}
\mathbf{m}(\eta)=\mu_{N}(\eta)\prod_{x\in S}\eta_{x}\;.\label{funm}
\end{equation}
Then we can obtain the following estimate on $\mathbf{m}$ based on
the maximum principle given in Lemma \ref{gmp}.
\begin{lem}
\label{lem714}There exists $C=C(\epsilon)$ such that for each $\eta\in\mathcal{I}_{N}^{S}$,
\[
\mathbf{m}(\eta)\le\max_{\partial\mathcal{I}_{N}^{S}}\mathbf{m}+C\frac{d_{N}N^{3}}{(\log N)^{2}}\max_{\overline{\mathcal{I}}_{N}^{S}}\mathbf{m}\;.
\]
\end{lem}

\begin{proof}
We can deduce from \eqref{invmsr} that, for each $\eta\in\mathcal{I}_{N}^{S}$,
\[
\sum_{x,\,y\in S}\mu_{N}(\eta)\eta_{x}(d_{N}+\eta_{y})r(x,\,y)=\sum_{x,\,y\in S}\mu_{N}(\sigma^{y,\,x}\eta)(\eta_{x}+1)(d_{N}+\eta_{y}-1)r(x,\,y)\;.
\]
Inserting $\mu_{N}(\eta)=\mathbf{m}(\eta)\left(\prod_{x\in S}\eta_{x}\right)^{-1}$
and rearranging it yield that
\[
\mathbf{m}(\eta)=\sum_{x,\,y\in S}\frac{\eta_{x}(d_{N}+\eta_{y})r(x,\,y)\frac{\eta_{y}(d_{N}+\eta_{y}-1)}{(d_{N}+\eta_{y})(\eta_{y}-1)}}{\sum_{a,\,b\in S}\eta_{a}(d_{N}+\eta_{b})r(a,\,b)}\mathbf{m}(\sigma^{y,\,x}\eta)\;.
\]
By recalling the definition of $\widehat{\mathbf{p}}_{N}^{S}$ (cf.
\eqref{pnr}), we can rewrite the last identity as
\begin{align}
\mathbf{m}(\eta)= & \sum_{x,\,y\in S}\left[1+\frac{d_{N}}{(d_{N}+\eta_{y})(\eta_{y}-1)}\right]\widehat{\mathbf{p}}_{N}^{S}(\eta,\,\sigma^{y,\,x}\eta)\mathbf{m}(\sigma^{y,\,x}\eta)\;.\label{em1}
\end{align}
For $\eta\in\mathcal{I}_{N}^{S}$, we have
\begin{equation}
\sum_{x,\,y\in S}\frac{d_{N}}{(d_{N}+\eta_{y})(\eta_{y}-1)}\widehat{\mathbf{p}}_{N}^{S}(\eta,\,\sigma^{y,\,x}\eta)\mathbf{m}(\sigma^{y,\,x}\eta)\le\frac{Cd_{N}}{(\log N)^{2}}\max_{\overline{\mathcal{I}}_{N}^{S}}\mathbf{m}\label{coretemp3}
\end{equation}
since $\eta_{x}\ge\epsilon\log N$ for all $x\in S$. By \eqref{em1}
and \eqref{coretemp3}, $\mathbf{m}$ satisfies
\begin{align}
\mathbf{m}(\eta)\le & \sum_{\zeta\in\overline{\mathcal{I}}_{N}^{S}}\widehat{\mathbf{p}}_{N}^{S}(\eta,\,\zeta)\mathbf{m}(\zeta)+\frac{Cd_{N}}{(\log N)^{2}}\max_{\overline{\mathcal{I}}_{N}^{S}}\mathbf{m}\;.\label{em1-1}
\end{align}
Hence, the proof is complete by Lemma \ref{gmp} with $R=S$ and $\mathbf{f}=\mathbf{m}$.
\end{proof}
Now, we are ready to prove Proposition \ref{p67} by combining results
obtained in Lemmas \ref{lem710}-\ref{lem714}.
\begin{proof}[Proof of Proposition \ref{p67}]
By Lemma \ref{lem714},
\begin{align}
\mu_{N}(\mathcal{I}_{N}^{S})=\; & \sum_{\eta\in\mathcal{I}_{N}^{S}}\mu_{N}(\eta)=\sum_{\eta\in\mathcal{I}_{N}^{S}}\frac{\mathbf{m}(\eta)}{\prod_{x\in S}\eta_{x}}\nonumber \\
\le\; & \sum_{\eta\in\mathcal{I}_{N}^{S}}\frac{1}{\prod_{x\in S}\eta_{x}}\,\bigg\{\max_{\partial\mathcal{I}_{N}^{S}}\mathbf{m}+\frac{d_{N}}{(\log N)^{2}}CN^{3}\max_{\overline{\mathcal{I}}_{N}^{S}}\mathbf{m}\bigg\}\label{rmndintemp}\\
\le\; & CN^{-1}(\log N)^{\kappa-1}\,\bigg\{\max_{\partial\mathcal{I}_{N}^{S}}\mathbf{m}+\frac{d_{N}}{(\log N)^{2}}CN^{3}\max_{\overline{\mathcal{I}}_{N}^{S}}\mathbf{m}\bigg\}\;,\nonumber
\end{align}
where the last line follows from Lemma \ref{prelem} (note that $\kappa=|S|$).
Recall the definition of $\mathbf{m}$ from \eqref{funm} and note
that
\[
\prod_{x\in S}\eta_{x}=\begin{cases}
O(N^{\kappa-1}\log N) & \text{for }\eta\in\partial\mathcal{I}_{N}^{S}\text{ and},\\
O(N^{\kappa}) & \text{for }\eta\in\overline{\mathcal{I}}_{N}^{S}\;.
\end{cases}
\]
Based on this, we can further deduce from \eqref{rmndintemp} that
\begin{align*}
\mu_{N}(\mathcal{I}_{N}^{S}) & \le CN^{\kappa-2}(\log N)^{\kappa}\max_{\partial\mathcal{I}_{N}^{S}}\mu_{N}+Cd_{N}N^{\kappa+2}(\log N)^{\kappa-3}\max_{\overline{\mathcal{I}}_{N}^{S}}\mu_{N}\\
 & \le CN^{\kappa-2}(\log N)^{\kappa}\mu_{N}(\partial\mathcal{I}_{N}^{S})+Cd_{N}N^{\kappa+2}(\log N)^{\kappa-3}\mu_{N}(\overline{\mathcal{I}}_{N}^{S})\;.
\end{align*}
By the condition on $d_{N}$ given at the statement of the proposition,
we complete the proof.
\end{proof}

\subsection{\label{sec73}Induction step}

Next, we consider the induction step. We shall prove the following
two statements together by the backward induction: there exists $C>0$
such that, for all $i\in\llbracket2,\,\kappa\rrbracket$,
\begin{equation}
\lim_{N\rightarrow\infty}\frac{\mu_{N}(\mathcal{B}_{N}^{i})}{\mu_{N}(\mathcal{B}_{N}^{i-1})}=1\;,\label{indhyp}
\end{equation}
and for all $R\subseteq S\text{ with }|R|=i\text{, and }\forall z\in R$,
\begin{equation}
\mu_{N}(\mathcal{C}_{N}^{R}(z,\,1))\le Cd_{N}\mu_{N}(\mathcal{B}_{N}^{i-1})\;.\label{indhyp2}
\end{equation}
Note that the initial case $i=\kappa$ for \eqref{indhyp} is proven
in Propositions \ref{rmndoutest}, \ref{rmndinest}, and for \eqref{indhyp2}
is proven in Lemma \ref{sliceest}.

Now, we will assume the following condition throughout this subsection:
\begin{equation}
\lim_{N\rightarrow\infty}d_{N}N^{\kappa+2}(\log N)^{\kappa-3}=0\;.\label{condd}
\end{equation}

\begin{prop}
\label{indstepconcl}Suppose that the induction hypotheses \eqref{indhyp}
and \eqref{indhyp2} hold for $i=\ell+1$. Then, \eqref{indhyp} and
\eqref{indhyp2} hold for $i=\ell$ as well.
\end{prop}

The overall outline of the proof is similar to the initial step, but
several additional technical difficulties arise in the course of the
proof. As before, we investigate the outer core and inner core separately
in Lemmas \ref{rmndoutest'} and \ref{rmndinest'}, respectively.

\subsubsection*{Estimation of the outer core}

For the outer core $\mathcal{O}_{N}^{R}$ with $R\subseteq S$, we
will prove the following bound.
\begin{lem}
\label{rmndoutest'}For all $R\subseteq S$, it holds that
\[
\mu_{N}(\mathcal{O}_{N}^{R})=o_{N}(1)\left[\mu_{N}(\partial\mathcal{A}_{N}^{R})+\mu_{N}(\mathcal{A}_{N}^{R})\right]\;.
\]
\end{lem}

We first prove two preliminary lemmas before proving this lemma. Recall
the notions introduced after Proposition \ref{rmndinest}.
\begin{lem}
\label{l617}For all $R\subseteq S$, $x\in R$, and $j\in\llbracket0,\,N-1\rrbracket$,
it holds that
\begin{equation}
\mathbf{F}_{N}^{R}(x;\,j+1\rightarrow j)-\mathbf{F}_{N}^{R}(x;\,j\rightarrow j+1)\le Cd_{N}N\,\mu_{N}(\mathcal{A}_{N}^{R})\;.\label{flowse}
\end{equation}
\end{lem}

\begin{proof}
By summing \eqref{invmsr} over $\eta\in\mathcal{C}_{N}^{R}(x,\,k)$,
\begin{align}
 & \sum_{\eta\in\mathcal{C}_{N}^{R}(x,\,k)}\sum_{y,\,z\in R\setminus\{x\}}\mu_{N}(\eta)\mathbf{r}_{N}(\eta,\,\sigma^{y,\,z}\eta)+\sum_{\eta\in\mathcal{C}_{N}^{R}(x,\,k)}\sum_{y\in R\setminus\{x\}}\mu_{N}(\eta)\mathbf{r}_{N}(\eta,\,\sigma^{y,\,x}\eta)\nonumber \\
 & +\sum_{\eta\in\mathcal{C}_{N}^{R}(x,\,k)}\sum_{y\in R\setminus\{x\}}\mu_{N}(\eta)\mathbf{r}_{N}(\eta,\,\sigma^{x,\,y}\eta)+\sum_{\eta\in\mathcal{C}_{N}^{R}(x,\,k)}\sum_{y\in R,\,z\in R^{c}}\mu_{N}(\eta)\mathbf{r}_{N}(\eta,\,\sigma^{y,\,z}\eta)\nonumber \\
=\, & \sum_{\eta\in\mathcal{C}_{N}^{R}(x,\,k)}\sum_{y,\,z\in R\setminus\{x\}}\mu_{N}(\sigma^{y,\,z}\eta)\mathbf{r}_{N}(\sigma^{y,\,z}\eta,\,\eta)+\sum_{\eta\in\mathcal{C}_{N}^{R}(x,\,k)}\sum_{y\in R\setminus\{x\}}\mu_{N}(\sigma^{y,\,x}\eta)\mathbf{r}_{N}(\sigma^{y,\,x}\eta,\,\eta)\nonumber \\
 & +\sum_{\eta\in\mathcal{C}_{N}^{R}(x,\,k)}\sum_{y\in R\setminus\{x\}}\mu_{N}(\sigma^{x,\,y}\eta)\mathbf{r}_{N}(\sigma^{x,\,y}\eta,\,\eta)+\sum_{\eta\in\mathcal{C}_{N}^{R}(x,\,k)}\sum_{y\in R,\,z\in R^{c}}\mu_{N}(\sigma^{y,\,z}\eta)\mathbf{r}_{N}(\sigma^{y,\,z}\eta,\,\eta)\;.\label{slicetemp'}
\end{align}
Compared to the corresponding computations in Lemma \ref{flowsym},
the last terms in both sides of \eqref{slicetemp'} should be handled
in addition. The term in the left-hand side is bounded above by
\begin{align*}
\sum_{\eta\in\mathcal{C}_{N}^{R}(x,\,k)}\sum_{y\in R,\,z\in R^{c}}\mu_{N}(\eta)\mathbf{r}_{N}(\eta,\;\sigma^{y,\,z}\eta)\, & =\,O(d_{N}N)\sum_{\eta\in\mathcal{C}_{N}^{R}(x,\,k)}\mu_{N}(\eta)\\
 & =\,O(d_{N}N)\,\mu_{N}(\mathcal{C}_{N}^{R}(x,\,k))\;.
\end{align*}
The term in the right-hand side of \eqref{slicetemp'} is bounded
below by $0$. Hence, we can obtain from \eqref{slicetemp'} that
\begin{align*}
 & \left[\mathbf{F}_{N}^{R}(x;\,k+1\rightarrow k)-\mathbf{F}_{N}^{R}(x;\,k\rightarrow k+1)\right]-\left[\mathbf{F}_{N}^{R}(x;\,k\rightarrow k-1)-\mathbf{F}_{N}^{R}(x;\,k-1\rightarrow k)\right]\\
\le\; & Cd_{N}N\,\mu_{N}(\mathcal{C}_{N}^{R}(x,\,k))\;.
\end{align*}
By summing the bound over $k=0,\,1,\,\dots,\,j$, we obtain \eqref{flowse}.
\end{proof}
\begin{lem}
\label{l618}There exists $C>0$ such that for all $R\subseteq S$,
$x\in R$, and $k\in\llbracket0,\,N-1\rrbracket$, we have
\begin{align*}
\mu_{N}(\mathcal{C}_{N}^{R}(x,\,k+1))\le\; & C\frac{(k+d_{N})(N-k)}{(k+1)(N-k-1+d_{N})}\mu_{N}(\mathcal{C}_{N}^{R}(x,\,k))\\
 & +C\frac{d_{N}N}{(k+1)(N-k-1+d_{N})}\mu_{N}(\mathcal{A}_{N}^{R})\;.
\end{align*}
\end{lem}

\begin{proof}
The proof is identical to Lemma \ref{sliceest} if we replace the
role of Lemma \ref{flowsym} with that of Lemma \ref{l617}.
\end{proof}
Now, we prove Lemma \ref{rmndoutest'} based on Lemmas \ref{l617}
and \ref{l618}.
\begin{proof}[Proof of Lemma \ref{rmndoutest'}]
Fix $x\in R$. Inserting $k=0$ in Lemma \ref{l618} implies that
there exists a constant $C_{1}>1$ such that
\begin{equation}
\mu_{N}(\mathcal{C}_{N}^{R}(x,\,1))\le C_{1}\,d_{N}\,\mu_{N}(\mathcal{C}_{N}^{R}(x,\,0))+C_{1}\,d_{N}\,\mu_{N}(\mathcal{A}_{N}^{R})\;.\label{e161}
\end{equation}
On the other hand, inserting $k\in\llbracket1,\,N-1\rrbracket$ to
Lemma \ref{l618} implies that there exists a constant $C_{2}>1$
such that
\begin{equation}
\mu_{N}(\mathcal{C}_{N}^{R}(x,\,k+1))\le C_{2}\,\mu_{N}(\mathcal{C}_{N}^{R}(x,\,k))+C_{2}\,d_{N}\,\mu_{N}(\mathcal{A}_{N}^{R})\;.\label{e162}
\end{equation}
Let $C_{0}=\max\{C_{1},\,C_{2}\}$. Then, by \eqref{e161} and \eqref{e162},
we obtain that
\begin{equation}
\mu_{N}(\mathcal{C}_{N}^{R}(x,\,k))\le C_{0}^{k}\,d_{N}\,\mu_{N}(\mathcal{C}_{N}^{R}(x,\,0))+\frac{C_{0}^{k+1}-C_{0}}{C_{0}-1}\,d_{N}\,\mu_{N}(\mathcal{A}_{N}^{R})\;\;\;;\;k\in\llbracket1,\,N-1\rrbracket\;.\label{slicetemp6'}
\end{equation}
Summing this up for for $k\in\llbracket1,\,\epsilon\log N\rrbracket$,
it holds that
\[
\sum_{k=1}^{\epsilon\log N}\mu_{N}(\mathcal{C}_{N}^{R}(x,\,k))\le C\times C_{0}^{\epsilon\log N}d_{N}\left[\mu_{N}(\mathcal{C}_{N}^{R}(x,\,0))+\mu_{N}(\mathcal{A}_{N}^{R})\right]\;.
\]
Take $\epsilon$ small enough so that
\begin{equation}
C_{0}^{\epsilon\log N}\ll N\;.\label{epscond2}
\end{equation}
Therefore, by \eqref{rmndoutsbst},
\begin{align*}
\mu_{N}(\mathcal{O}_{N}^{R}) & \le\sum_{x\in R}\sum_{k=1}^{\epsilon\log N}\mu_{N}(\mathcal{C}_{N}^{R}(x,\,k))\le Cd_{N}N\sum_{x\in R}\left\{ \mu_{N}(\mathcal{A}_{N}^{R\setminus\{x\}})+\mu_{N}(\mathcal{A}_{N}^{R})\right\} \\
 & =O(d_{N}N)\left[\mu_{N}(\partial\mathcal{A}_{N}^{R})+\mu_{N}(\mathcal{A}_{N}^{R})\right]\;.
\end{align*}
This finishes the proof.
\end{proof}

\subsubsection*{Estimation of the inner core}

Next, we control the inner core $\mathcal{I}_{N}^{R}$. The proof
of the following lemma also relies on Lemma \ref{hitest} regarding
the estimate of the hitting time.
\begin{lem}
\label{rmndinest'} Suppose that \eqref{indhyp2} holds for $i=\ell+1$.
Then, for all $R\subseteq S$ with $|R|=\ell$, we have that
\[
\mu_{N}(\mathcal{I}_{N}^{R})=o_{N}(1)\left[\mu_{N}(\partial\mathcal{A}_{N}^{R})+\mu_{N}(\mathcal{B}_{N}^{\ell})\right]\;.
\]
\end{lem}

\begin{proof}
Fix $R\subseteq S$ and define $\mathbf{m}^{R}:\mathcal{H}_{N}\rightarrow\mathbb{R}$
by
\begin{equation}
\mathbf{m}^{R}(\eta)=\mu_{N}(\eta)\prod_{x\in R}\eta_{x}\;.\label{met2}
\end{equation}
Similarly to Lemma \ref{lem714}, for $\eta\in\mathcal{I}_{N}^{R}$,
we get
\[
\mu_{N}(\eta)=\sum_{y\in R,\,x\in S\setminus\{y\}}\frac{(\eta_{x}+1)(d_{N}+\eta_{y}-1)r(x,\,y)}{\sum_{a\in R,\,b\in S\setminus\{a\}}\eta_{a}(d_{N}+\eta_{b})r(a,\,b)}\,\mu_{N}(\sigma^{y,\,x}\eta)\;.
\]
In the denominator of the right-hand side, we discard the transitions escaping $R$ to get the following upper bound for $\mu_N(\eta)$:
\begin{align}
\mu_{N}(\eta)\le & \sum_{x,\,y\in R}\frac{(\eta_{x}+1)(d_{N}+\eta_{y}-1)r(x,\,y)}{\sum_{a,\,b\in R}\eta_{a}(d_{N}+\eta_{b})r(a,\,b)}\,\mu_{N}(\sigma^{y,\,x}\eta)\nonumber \\
 & +\sum_{y\in R,\,x\in S\setminus R}\frac{(d_{N}+\eta_{y}-1)r(x,\,y)}{\sum_{a,\,b\in R}\eta_{a}(d_{N}+\eta_{b})r(a,\,b)}\,\mu_{N}(\sigma^{y,\,x}\eta)\;.\label{coretemp'}
\end{align}
By the assumption that \eqref{indhyp2} holds for $i=\ell+1$, the
last term is bounded by
\[
\sum_{x\in S\setminus R}\frac{C}{\log N}\,\mu_{N}(\mathcal{C}_{N}^{R\cup\{x\}}(x,\,1))\le\frac{Cd_{N}}{\log N}\,\mu_{N}(\mathcal{B}_{N}^{\ell})\;.
\]
Inserting this to \eqref{coretemp'} and using the formula \eqref{met2}
of $\mathbf{m}^{R}$, we can deduce that
\begin{equation}
\mathbf{m}^{R}(\eta)\le\sum_{x,\,y\in R}\frac{\eta_{x}(d_{N}+\eta_{y})r(x,\,y)\frac{\eta_{y}(d_{N}+\eta_{y}-1)}{(d_{N}+\eta_{y})(\eta_{y}-1)}}{\sum_{a,\,b\in R}\eta_{a}(d_{N}+\eta_{b})r(a,\,b)}\,\mathbf{m}^{R}(\sigma^{y,\,x}\eta)+\frac{Cd_{N}N^{\ell}}{\log N}\mu_{N}(\mathcal{B}_{N}^{\ell})\;.\label{coretemp2'}
\end{equation}
Now, as in the proof of Lemma \ref{lem714} (cf. \eqref{em1}, \eqref{coretemp3},
and \eqref{em1-1}), we can obtain from the previous inequality that
\[
\mathbf{m}^{R}(\eta)\le\sum_{x,\,y\in R}\mathbf{p}_{N}^{R}(\eta,\sigma^{y,\,x}\eta)\mathbf{m}^{R}(\sigma^{y,\,x}\eta)+\frac{2d_{N}}{(\log N)^{2}}\max_{\overline{\mathcal{I}}_{N}^{R}}\mathbf{m}^{R}+\frac{Cd_{N}N^{\ell}}{\log N}\mu_{N}(\mathcal{B}_{N}^{\ell})\;.
\]
Therefore, Lemma \ref{gmp} with $\mathbf{f}=\mathbf{m}^{R}$ and
Lemma \ref{hitest} give that,
\begin{equation}
\mathbf{m}^{R}(\eta)\le\max_{\partial\mathcal{I}_{N}^{R}}\mathbf{m}^{R}+\max_{\overline{\mathcal{I}}_{N}^{R}}\mathbf{m}^{R}\frac{Cd_{N}N^{3}}{(\log N)^{2}}+\frac{Cd_{N}N^{\ell+3}}{\log N}\mu_{N}(\mathcal{B}_{N}^{\ell})\label{coretemp4'}
\end{equation}
for all $\eta\in\mathcal{I}_{N}^{R}$. Now recalling the definition
\eqref{met2} and applying Lemma \ref{prelem},
\begin{align*}
\mu_{N}(\mathcal{I}_{N}^{R})\le\; & \sum_{\eta\in\mathcal{I}_{N}^{R}}\frac{1}{\prod_{x\in R}\eta_{x}}\bigg\{\max_{\partial\mathcal{I}_{N}^{R}}\mathbf{m}^{R}+\max_{\overline{\mathcal{I}}_{N}^{R}}\mathbf{m}^{R}\frac{Cd_{N}N^{3}}{(\log N)^{2}}+\frac{Cd_{N}N^{\ell+3}}{\log N}\mu_{N}(\mathcal{B}_{N}^{\ell})\bigg\}\\
\le\; & \frac{C(\log N)^{\ell-1}}{N}\bigg\{ N^{\ell-1}\log N\max_{\partial\mathcal{I}_{N}^{R}}\mu_{N}+d_{N}\frac{N^{\ell+3}}{(\log N)^{2}}\max_{\overline{\mathcal{I}}_{N}^{R}}\mu_{N}+d_{N}\frac{N^{\ell+3}}{\log N}\mu_{N}(\mathcal{B}_{N}^{\ell})\bigg\}\\
=\; & CN^{\ell-2}(\log N)^{\ell}\mu_{N}(\partial\mathcal{I}_{N}^{R})+Cd_{N}N^{\ell+2}(\log N)^{\ell-3}\mu_{N}(\overline{\mathcal{I}}_{N}^{R})\\
 & +Cd_{N}N^{\ell+2}(\log N)^{\ell-2}\mu_{N}(\mathcal{B}_{N}^{\ell})\;.
\end{align*}
By \eqref{condd}, we can finally deduce that
\begin{equation}
\mu_{N}(\mathcal{I}_{N}^{R})=O(N^{\ell-2}(\log N)^{\ell})\mu_{N}(\partial\mathcal{I}_{N}^{R})+O(d_{N}N^{\ell+2}(\log N)^{\ell-2})\mu_{N}(\mathcal{B}_{N}^{\ell})\;.\label{rmndintemp4'}
\end{equation}
Recall the notation defined after Proposition \ref{rmndoutest} to
see that
\[
\partial\mathcal{I}_{N}^{R}\subseteq\bigcup_{x\in R}\mathcal{C}_{N}^{R}(x,\,\epsilon\log N)\;.
\]
Therefore by \eqref{slicetemp6'},
\begin{equation}
\mu_{N}(\partial\mathcal{I}_{N}^{R})\le\sum_{x\in R}\mu_{N}(\mathcal{C}_{N}^{R}(x,\,\epsilon\log N))=O(d_{N}N)\left[\mu_{N}(\partial\mathcal{A}_{N}^{R})+\mu_{N}(\mathcal{A}_{N}^{R})\right]\;.\label{rmndintemp5'}
\end{equation}
\eqref{rmndintemp4'} and \eqref{rmndintemp5'} give
\[
\mu_{N}(\mathcal{I}_{N}^{R})=O(d_{N}N^{\ell-1}(\log N)^{\ell})\mu_{N}(\partial\mathcal{A}_{N}^{R})+O(d_{N}N^{\ell+2}(\log N)^{\ell-2})\mu_{N}(\mathcal{B}_{N}^{\ell})\;.
\]
By \eqref{condd}, we finish the proof.
\end{proof}
\begin{proof}[Proof of Proposition \ref{indstepconcl}]
Take $R\subseteq S$ with $|R|=\ell$. Since $\mathcal{R}_{N}^{R}$
is decomposed into $\mathcal{O}_{N}^{R}$ and $\mathcal{I}_{N}^{R}$,
and since $\partial\mathcal{A}_{N}^{R}\subseteq\mathcal{B}_{N}^{\ell-1}$,
we can derive from Propositions \ref{rmndoutest'} and \ref{rmndinest'}
that
\begin{align*}
\mu_{N}(\mathcal{R}_{N}^{R}) & =o_{N}(1)\left[\mu_{N}(\partial\mathcal{A}_{N}^{R})+\mu_{N}(\mathcal{B}_{N}^{\ell})\right]\le o_{N}(1)\left[\mu_{N}(\mathcal{B}_{N}^{\ell-1})+\mu_{N}(\mathcal{B}_{N}^{\ell})\right]\;.
\end{align*}
Summing the last bound over all $R\subseteq S$ with $|R|=\ell$ yields
that
\[
\mu_{N}(\mathcal{B}_{N}^{\ell}\setminus\mathcal{B}_{N}^{\ell-1})=o_{N}(1)\left[\mu_{N}(\mathcal{B}_{N}^{\ell-1})+\mu_{N}(\mathcal{B}_{N}^{\ell})\right]\;.
\]
We can deduce \eqref{indhyp} with $i=\ell$ from here. On the other
hand, we can verify \eqref{indhyp2} with $i=\ell$ from \eqref{indhyp}
and \eqref{slicetemp6'}. To be more specific, for $x\in R$, inserting
$k=1$ in \eqref{slicetemp6'} gives us
\[
\mu_{N}(\mathcal{C}_{N}^{R}(x,\,1))\le Cd_{N}\mu_{N}(\mathcal{C}_{N}^{R}(x,\,0))+Cd_{N}\mu_{N}(\mathcal{A}_{N}^{R})\;.
\]
Since $\mathcal{C}_{N}^{R}(x,\,0)\subseteq\mathcal{B}_{N}^{\ell-1}$
and $\mu_{N}(\mathcal{A}_{N}^{R})\le\mu_{N}(\mathcal{B}_{N}^{\ell})=(1+o_{N}(1))\mu_{N}(\mathcal{B}_{N}^{\ell-1})$
by \eqref{indhyp}, we conclude that $\mu_{N}(\mathcal{C}_{N}^{R}(x,\,1))\le Cd_{N}\mu_{N}(\mathcal{B}_{N}^{\ell-1})$
and thus conclude the proof of Proposition \ref{indstepconcl}.
\end{proof}
\begin{rem}
We remark that the final step in \eqref{trrpospropeq}, i.e., $\ell=2$,
can be proved in a completely independent way without assumption \textbf{(UP)},
and with a much weaker assumption on $d_{N}$. To be more specific,
we can prove the following result:
\begin{thm*}
Suppose that $\lim_{N\rightarrow\infty}d_{N}N\log N=0$. Then, we
have
\[
\lim_{N\rightarrow\infty}\frac{\mu_{N}(\mathcal{E}_{N})}{\mu_{N}(\mathcal{A}_{N})}=1\;.
\]
\end{thm*}
Note that under condition \textbf{(UP)}, this is exactly the case
$\ell=2$ in \eqref{trrpospropeq}. We omit the proof of this statement,
and only remark that it can be proved by tracing the original process
on $\mathcal{A}_{N}$ and calculating the transition rates of the
trace process, as done in Section \ref{sec43}.
\end{rem}

\subsection{\label{sec74}Proof of Theorem \ref{t211}}

Now, we explain the proof of Theorem \ref{t211} whose main idea of
proof is nearly identical to that of Lemma \ref{hitest}. Slight difference
is that here we are dealing with the original continuous-time chain
$\eta_{N}(\cdot)$, instead of the reversed discrete-time chain $\hat{\eta}_{N}^{R}(\cdot)$.
\begin{proof}[Proof of Theorem \ref{t211}]
 We recall the definition of $\mathcal{U}_{N}$ from the display
before Theorem \ref{t211}. Let us identify $\epsilon$ in the definition
of $\mathcal{I}_{N}^{S}$ with $\delta$ in the definition of $\mathcal{U}_{N}$.
Then, in the terminology introduced in this section, we have $\mathcal{U}_{N}=(\mathcal{I}_{N}^{S})^{c}$
and thus $\tau_{\mathcal{U}_{N}}=\tau_{\partial\mathcal{I}_{N}^{S}}$
provided that the chain starts in $\mathcal{U}_{N}^{c}$. Thus a deduction
similar to that in Lemma \ref{lem710} guarantees the existence of
test function $\mathbf{g}_{0}:\overline{\mathcal{I}}_{N}^{S}\rightarrow\mathbb{R}$
such that
\[
\max_{\overline{\mathcal{I}}_{N}^{S}}\mathbf{g}_{0}-\min_{\overline{\mathcal{I}}_{N}^{S}}\mathbf{g}_{0}\le C\log N\;\;\text{ and}\;\;\;(\mathcal{L}_{N}\mathbf{g}_{0})(\eta)\ge\frac{\log N}{CN}\text{\; for all }\eta\in\mathcal{I}_{N}^{S}\;.
\]
Here, the denominator of the lower bound of $(\mathcal{L}_{N}\mathbf{g}_{0})(\eta)$
is $CN$ instead of $CN^{3}$, since there is no $\mathbf{w}(\eta)$
term as in Lemma \ref{lem710} in the calculation of the continuous-time
generator $\mathcal{L}_{N}$. Let us consider an arbitrary extension
of $\mathbf{g}_{0}$ to a function on $\mathcal{H}_{N}$ and then
consider the continuous-time martingale
\[
M_{\mathbf{g}_{0}}(t):=\mathbf{g}_{0}(\eta_{N}(t))-\mathbf{g}_{0}(\eta_{N}(0))-\int_{0}^{t}(\mathcal{L}_{N}\mathbf{g}_{0})(\eta_{N}(s))ds\;\;\;;\;t\ge0\;.
\]
Then, proceeding as in Lemma \ref{hitest}, we can conclude that $\mathbb{E}_{\eta}[\tau_{\mathcal{U}_{N}}]\le CN$.
\end{proof}

\section{\label{sec8}Inclusion Processes in Thermodynamic Limit Regime}

In this section, we consider the inclusion process in the thermodynamic
limit regime and prove the condensation (Theorem \ref{tlmain}) and
the metastable behavior (Theorems \ref{tt1}-\ref{tt3}).

\subsubsection*{Organization of the section}

In Section \ref{sec81}, we prove the existence of the condensation
(Theorem \ref{tlmain}), which is indeed not very far from that of
the fixed $L$ case under \textbf{(UI)}. On the other hand, the metastable
behavior is more delicate than the fixed $L$ case, mainly because
the limiting dynamic is now a continuous process on $\mathbb{T}^{d}$,
while the trace process is a jump process on $\mathbb{T}_{L}^{d}$.
The proof of this convergence is based on two ingredients: the convergence
of the generator (Proposition \ref{maint}) and the tightness (Proposition
\ref{tightness}). These ingredients are obtained in Sections \ref{sec82}
and \ref{sec83}, respectively. Finally, we prove Theorems \ref{tt1}-\ref{tt3}
in Section \ref{sec84}.

\subsection{\label{sec81}Condensation}

We first establish condensation by proving Theorem \ref{tlmain}.
This should be distinguished from the former cases by the fact that
the graph grows along with the number of particles. Although the proof
is given in \cite[Proposition 2]{C-G-J}, we present a proof here
for the completeness of the article.
\begin{proof}[Proof of Theorem \ref{tlmain}]
Recall $\mathcal{E}_{L}$ from \eqref{tel}. Then, it suffices to
show that
\begin{equation}
\lim_{L\rightarrow\infty}\frac{\mu_{L}(\mathcal{H}_{L}\setminus\mathcal{E}_{L})}{\mu_{L}(\mathcal{E}_{L})}=0\;.\label{tdin}
\end{equation}
Since the inclusion process that we consider here satisfies the condition\textbf{
(UI)}, thanks to Proposition \ref{p13}, the invariant measure of
the process denoted by $\mu_{L}$ can be expressed explicitly by
\begin{equation}
\mu_{L}(\eta)=\frac{1}{Z_{L}}\prod_{x\in\mathbb{T}_{L}^{d}}w_{L}(\eta_{x})\;,\quad\eta\in\mathcal{H}_{L}\;,\label{tinv-1}
\end{equation}
where
\[
w_{L}(n)=\frac{\Gamma(n+d_{L})}{n!\Gamma(d_{L})}\;,\;n\in\mathbb{N}\text{ \;\;and\;\; }Z_{L}=\sum_{\eta\in\mathcal{H}_{L}}\prod_{x\in\mathbb{T}_{L}^{d}}w_{L}(\eta_{x})\;.
\]
We recall the following elementary inequality from \cite[Lemma 3.1]{B-D-G}:
\[
\frac{d_{L}}{d_{L}+k}\frac{1}{\Gamma(d_{L}+1)}\le w_{L}(k)\le\frac{d_{L}}{k}e^{d_{L}\log L}\text{ for all }k\in\llbracket1,\,L\rrbracket\;.
\]
Since we assumed that $\lim_{L\rightarrow\infty}d_{L}L^{d}\log L=0$,
the previous inequality implies that
\begin{equation}
w_{L}(k)=(1+o_{L}(1))\frac{d_{L}}{k}\;\;\text{uniformly for }k\in\llbracket1,\,L\rrbracket\;.\label{hecy1}
\end{equation}
Decompose
\begin{equation}
\mathcal{H}_{L}\setminus\mathcal{E}_{L}=\bigcup_{i=2}^{L}\Delta_{i}\label{hede}
\end{equation}
where, for each $i\in\llbracket2,\,L\rrbracket$,
\[
\Delta_{i}=\{\eta\in\mathcal{H}_{L}:\text{exactly }i\text{ coordinates of }\eta=(\eta_{x})_{x\in\mathbb{T}_{L}}\text{ are non-zero}\}\;.
\]
By \eqref{hecy1} and the definition of $S_{n,\,k}$ in Lemma \ref{prelem},
for large enough $L$,
\[
\mu_{L}(\Delta_{i})\le\frac{1}{Z_{L}}(2d_{L})^{i}S_{N,\,i}\times\binom{L^{d}}{i}\;,
\]
where the last term appears since there are $\binom{L}{i}$ ways to
select $i$ coordinates that are non-zero. By Lemma \ref{prelem},
it holds for all large enough $L$ that
\begin{align}
\mu_{L}(\Delta_{i})\le\; & \frac{1}{Z_{L}}\frac{1}{3N\log(N+1)}(6d_{L}\log(N+1))^{i}\binom{L^{d}}{i}\nonumber \\
\le\; & \frac{1}{Z_{L}}\frac{1}{N\log N}(7d_{L}\log L^{d})^{i}\binom{L^{d}}{i}\;.\label{hede1}
\end{align}
For convenience, write $u_{L}=7d_{L}\log L^{d}$. Then, by combining
\eqref{hede} and \eqref{hede1}, we obtain for all large enough $L$
that
\begin{align*}
\mu_{L}(\mathcal{H}_{L}\setminus\mathcal{E}_{L})=\sum_{i=2}^{L}\mu_{L}(\Delta_{i}) & \le\frac{1}{Z_{L}}\frac{1}{N\log N}\{(1+u_{L})^{L^{d}}-1-L^{d}u_{L}\}\\
 & \le\frac{1}{Z_{L}}\frac{1}{N\log N}\{e^{L^{d}u_{L}}-1-L^{d}u_{L}\}\\
 & \le\frac{1}{Z_{L}}\frac{1}{N\log N}(L^{d}u_{L})^{2}\;,
\end{align*}
where the last inequality follows because $\lim L^{d}u_{L}=0$. Thus,
\begin{equation}
\mu_{L}(\mathcal{H}_{L}\setminus\mathcal{E}_{L})\le\frac{C}{Z_{L}}d_{L}^{2}L^{d}\log L\;.\label{tltemp2}
\end{equation}
On the other hand, by the explicit formula \eqref{tinv-1} and the
asymptotic \eqref{LNlim}, we have that
\begin{equation}
\mu_{L}(\mathcal{E}_{L})=L^{d}\times\frac{1}{Z_{L}}w_{L}(N)w_{L}(0)^{L^{d}-1}=(1+o_{L}(1))L^{d}\times\frac{1}{Z_{L}}\frac{d_{L}}{N}=(1+o_{L}(1))\frac{1}{\rho Z_{L}}d_{L}\;.\label{tltemp3}
\end{equation}
Now, \eqref{tdin} is straightforward from \eqref{tltemp2} and \eqref{tltemp3}.
\end{proof}

\subsection{\label{sec82}Convergence of the generator}

Now, we consider the metastable behavior associated with the condensation
proved above. The generator $\mathcal{L}^{\mathbb{T}^{d}}$ associated
with the limiting object presented in Theorems \ref{tt1}-\ref{tt3}
can be written as, for all sufficiently smooth $f:\mathbb{T}^{d}\rightarrow\mathbb{R}$,
\begin{align}
 & (\mathcal{L}^{\mathbb{T}^{d}}f)(x)\nonumber \\
 & =\begin{cases}
\rho\left(\sum_{y\in\mathbb{Z}^{d}}h(y)y\right)\cdot\nabla f(x) & \text{for totally asym. case},\\
\frac{\rho}{2}\sum_{y\in\mathbb{Z}^{d}:h(y)>h(-y)}(h(y)-h(-y))y^{\dagger}\left[\nabla^{2}f(x)\right]y & \text{for mean-zero asym. case},\\
\frac{1}{2}\sum_{y\in\mathbb{Z}^{d}}h(y)y^{\dagger}\left[\nabla^{2}f(x)\right]y & \text{for symmetric case},
\end{cases}\label{gent}
\end{align}
where $(\nabla^{2}f)(x)$ denotes the Hessian of $f$ at $x$. The
main objective of this subsection is to prove the convergence of the
generator of the trace process to the generator $\mathcal{L}^{\mathbb{T}^{d}}$
in an appropriate sense as $L\rightarrow\infty$ (cf. Proposition
\ref{convgen}). The proof of this result again relies on the asymptotics
of the mean-jump rate.

\subsubsection*{Asymptotics of mean-jump rate}

We start by introducing several notations related to the mean-jump
rate. Recall that $\eta_{L}^{\mathcal{E}_{L}}(\cdot)$ denotes the
trace process of $\eta_{L}(\cdot)$ on the set $\mathcal{E}_{L}$.
We let $\mathbf{r}_{L}^{\mathcal{E}_{L}}(\cdot,\,\cdot)$, $\mathcal{L}_{L}^{\mathcal{E}_{L}}$,
and $\mu_{L}^{\mathcal{E}_{L}}$ denote the jump rate, the infinitesimal
generator and the invariant measure of the trace process $\eta_{L}^{\mathcal{E}_{L}}(\cdot)$,
respectively. For $x,\,y\in\mathbb{T}_{L}^{d}$, we write
\begin{equation}
\mathbf{b}_{L}(x,\,y)=\mathbf{r}_{L}^{\mathcal{E}_{L}}(\mathcal{\xi}_{L}^{x},\,\xi_{L}^{y})\;.\label{ebl}
\end{equation}
With these notation, we summarize the asymptotic relations for $\mathbf{b}_{L}(\cdot,\,\cdot)$
which are immediate from Proposition \ref{main}.
\begin{prop}
\label{maint}The followings hold for the inclusion process on $\mathbb{T}_{L}^{d}$
with $N\simeq\rho L^{d}$ particles:
\begin{enumerate}
\item for (either totally or mean-zero) asymmetric case, we have that
\begin{equation}
\mathbf{b}_{L}(x,\,x+y)=\begin{cases}
(1+O(d_{L}\log L+q^{N}))\,d_{L}N(h(y)-h(-y)) & \text{if }h(y)>h(-y),\\
O(d_{L}\log L+q^{N})\,d_{L}N & \text{otherwise}.
\end{cases}\label{asymjr}
\end{equation}
\item for symmetric case, we have that
\begin{equation}
\mathbf{b}_{L}(x,\,x+y)=\begin{cases}
\left(h(y)+O(d_{L}L^{d}\log L+L^{d}q^{N})\right)\,d_{L} & \text{if }h(y)=h(-y)>0,\\
O(d_{L}L^{d}\log L+L^{d}q^{N})\,d_{L} & \text{otherwise}.
\end{cases}\label{symjr}
\end{equation}
\end{enumerate}
\end{prop}

\subsubsection*{Convergence of generator of speeded-up trace process}

Now, we are ready to proceed to the main result regarding the convergence
of the generator. We are primarily interested in the convergence of
the speeded-up (Markov) process defined by
\begin{equation}
W_{L}(t)=Y_{L}(\theta_{L}t)\;,\label{wl-1}
\end{equation}
where
\[
\theta_{L}=\begin{cases}
1/(d_{L}L^{d-1}) & \text{for totally asymmetric case},\\
1/(d_{L}L^{d-2}) & \text{for mean-zero asymmetric case},\\
L^{2}/d_{L} & \text{for symmetric case}.
\end{cases}
\]
Let $\mathcal{L}^{W_{L}}$ denote the infinitesimal generator associated
with the continuous-time Markov chain $W_{L}(\cdot)$. Then, we can
write this generator as, for all $F:\mathbb{T}^{d}\rightarrow\mathbb{R}$,
\begin{align}
(\mathcal{L}^{W_{L}}F)\left(\frac{x}{L}\right) & =\theta_{L}\mathcal{L}_{L}^{\mathcal{E}_{L}}(F\circ\Theta_{L})(\xi_{L}^{x})\nonumber \\
 & =\theta_{L}\sum_{y\in\mathbb{T}_{L}^{d}}\mathbf{b}_{L}(x,\,x+y)\left\{ F\left(\frac{x+y}{L}\right)-F\left(\frac{x}{L}\right)\right\} \;.\label{igen}
\end{align}
The following is the main result of the current subsection.
\begin{prop}
\label{convgen}Under the conditions of Theorems \ref{tt1}-\ref{tt3},
it holds for all $f\in C^{3}(\mathbb{T}^{d})$ that
\[
\lim_{L\rightarrow\infty}\sup_{x\in\mathbb{T}_{L}^{d}}\left|(\mathcal{L}^{W_{L}}f)\left(\frac{x}{L}\right)-(\mathcal{L}^{\mathbb{T}^{d}}f)\left(\frac{x}{L}\right)\right|=0\;.
\]
\end{prop}

\begin{proof}
We fix $f\in C^{3}(\mathbb{T}^{d})$ and consider three cases separately.\\
\\
\textbf{(Case 1: Totally asymmetric case)} For this case, $\theta_{L}=1/(d_{L}L^{d-1})$.
Hence, by \eqref{igen} and by part (1) of Proposition \ref{maint},
we can deduce that
\begin{align*}
 & (\mathcal{L}^{W_{L}}f)(\xi_{L}^{x})-(\mathcal{L}^{\mathbb{T}^{d}}f)\Big(\frac{x}{L}\Big)\\
=\; & \frac{1}{d_{L}L^{d-1}}\sum_{y\in\mathbb{T}_{L}^{d}}\mathbf{b}_{L}(x,\,x+y)\left\{ f\Big(\frac{x+y}{L}\Big)-f\Big(\frac{x}{L}\Big)\right\} -\rho\sum_{y\in\mathbb{Z}^{d}}h(y)y\cdot\nabla f\Big(\frac{x}{L}\Big)\\
=\; & \sum_{y\in\mathbb{Z}^{d}:h(y)>h(-y)}\frac{N}{L^{d}}(h(y)-h(-y))y\cdot\nabla f\Big(\frac{x}{L}\Big)+o_{L}(1)-\rho\sum_{y\in\mathbb{T}_{L}^{d}}h(y)y\cdot\nabla f\Big(\frac{x}{L}\Big)\;.
\end{align*}
The second equality holds by the first-order Taylor expansion and the condition
$\lim_{L\rightarrow\infty}d_{L}L^{d+1}\log L=0$. Since $N/L^{d}\rightarrow\rho$,
the last line converges to $0$ as $L\rightarrow\infty$ and we are
done. \\
\\
\textbf{(Case 2: Mean-zero asymmetric case)} For this case, $\theta_{L}=1/(d_{L}L^{d-2})$;
thus, by \eqref{igen} and part (1) of Proposition \ref{maint}, we
obtain that
\begin{align*}
 & (\mathcal{L}^{W_{L}}f)(\xi_{L}^{x})-(\mathcal{L}^{\mathbb{T}^{d}}f)\Big(\frac{x}{L}\Big)\\
=\; & \frac{1}{d_{L}L^{d-2}}\sum_{y\in\mathbb{T}_{L}^{d}}\mathbf{b}_{L}(x,\,x+y)\bigg[f\Big(\frac{x+y}{L}\Big)-f\Big(\frac{x}{L}\Big)\bigg]-(\mathcal{L}^{\mathbb{T}^{d}}f)\Big(\frac{x}{L}\Big)\\
=\; & \frac{1}{d_{L}L^{d-2}}\sum_{y\in\mathbb{Z}^{d}:h(y)>h(-y)}d_{L}N(h(y)-h(-y))\bigg[f\Big(\frac{x+y}{L}\Big)-f\Big(\frac{x}{L}\Big)\bigg]\\
&+o_{L}(1)-(\mathcal{L}^{\mathbb{T}^{d}}f)\Big(\frac{x}{L}\Big)\;.
\end{align*}
In this case, unlike in \textbf{(Case 1)}, the first-order terms at
the Taylor expansion cancel out each other. Thus, we apply the second-order
Taylor expansion to get
\[
\frac{N}{2L^{d}}\sum_{y\in\mathbb{Z}^{d}:h(y)>h(-y)}(h(y)-h(-y))y^{\dagger}\nabla^{2}f\Big(\frac{x}{L}\Big)y-(\mathcal{L}^{\mathbb{T}^{d}}f)\Big(\frac{x}{L}\Big)+o_{L}(1)\;.
\]
This concludes the proof for this case since $N/L^{d}\rightarrow\rho$.\\
\\
\textbf{(Case 3: Symmetric case)} For this case, $\theta_{L}=L^{2}/d_{L}$.
Thus by \eqref{igen} and by part (2) of Proposition \ref{maint},
we obtain

\begin{align*}
 & (\mathcal{L}^{W_{L}}f)(\xi_{L}^{x})-(\mathcal{L}^{\mathbb{T}^{d}}f)\Big(\frac{x}{L}\Big)\\
=\; & \frac{L^{2}}{d_{L}}\sum_{y\in\mathbb{T}_{L}^{d}}\mathbf{b}_{L}(x,\,x+y)\bigg[f\Big(\frac{x+y}{L}\Big)-f\Big(\frac{x}{L}\Big)\bigg]-(\mathcal{L}^{\mathbb{T}^{d}}f)\Big(\frac{x}{L}\Big)\\
=\; & \frac{L^{2}}{2d_{L}}\sum_{y\in\mathbb{Z}^{d}}d_{L}h(y)\bigg[f\Big(\frac{x+y}{L}\Big)+f\Big(\frac{x-y}{L}\Big)-2f\Big(\frac{x}{L}\Big)\bigg]-(\mathcal{L}^{\mathbb{T}^{d}}f)\Big(\frac{x}{L}\Big)+o_{L}(1)\;.
\end{align*}
Note that the last error term is $o_{L}(1)$, since $\lim_{L\rightarrow\infty}d_{L}L^{2d+2}\log L=0$.
Hence, we apply the second-order Taylor expansion to deduce that the
last expression is equal to
\[
\frac{1}{2}\sum_{y\in\mathbb{Z}^{d}}h(y)y^{\dagger}\nabla f^{2}\Big(\frac{x}{L}\Big)y-(\mathcal{L}^{\mathbb{T}^{d}}f)\Big(\frac{x}{L}\Big)+o_{L}(1)\;.
\]

This finishes the proof the definition of $\mathcal{L}^{\mathbb{T}^{d}}$.
\end{proof}

\subsection{\label{sec83}Tightness}

The last ingredient for the proof of the convergence stated in part
(1) of Definition \ref{def218} is the tightness of the process $W_{L}(t)=Y_{L}(\theta_{L}t)$.
Let $\mathbb{Q}_{\eta}^{L}$, $\eta\in\mathcal{E}_{L}$ denote the
law of the process $W_{L}(\cdot)$ on the path space $D([0,\,\infty),\,\mathbb{T}^{d})$
when the inclusion process starts from $\eta$, i.e., associated with
the law $\mathbb{P}_{\eta}^{L}$.
\begin{prop}
\label{tightness}Let $(x_{L})_{L=1}^{\infty}$ be a sequence such
that $x_{L}\in\mathbb{T}_{L}^{d}$ for all $L\ge1$. Then, under the
conditions of Theorems \ref{tt1}-\ref{tt3}, the sequence $\{\mathbb{Q}_{\xi_{L}^{x_{L}}}^{L}\}_{L\ge1}$
of path measures is tight in $D([0,\,\infty),\,\mathbb{T}^{d})$.
\end{prop}

The natural way of proving this proposition is to use the Aldous criterion.
Of course, we found a proof of the tightness based on this criterion,
but controlling errors coming from the non-regularity of distance
function $d(x,\,0)=|x|$ around $0$ requires complicated computations
based on the large-deviation principle and the local central limit
theorem for the random walk on the discrete torus. Instead, we realized
that the criterion presented as Proposition \ref{tightcrt} is more
adequate to apply, in that it only considers smooth functions $F$,
which guarantees sufficiently small error terms via Taylor expansion.
\begin{proof}[Proof of Proposition \ref{tightness}]
The condition (1) of Proposition \ref{tightcrt} is straightforward,
since $\mathbb{T}^{d}$ is compact. Now let us check the condition
(2). To this end, fix $f\in C_{c}^{\infty}(\mathbb{T}^{d})$ and $\delta>0$.
Then, by the martingale problem associated with the Markov chain $W_{L}(\cdot)$,
we know that the process given by
\begin{equation}
M_{f}^{L}(t)=f(W_{L}(t))-f(W_{L}(0))-\int_{0}^{t}(\mathcal{L}^{W_{L}}f)(W_{L}(s))ds\label{tmart}
\end{equation}
is a $\mathbb{Q}_{\xi_{L}^{x_{L}}}^{L}$-martingale. Let $(\mathcal{\mathscr{F}}_{t}^{L})_{t\ge0}$
denote the canonical filtration associated with the process $W_{L}(\cdot)$
and by $\mathbf{E}_{\eta}^{L}$ the expectation associated with $\mathbb{Q}_{\eta}^{L}$.
Then, the previous observation implies that, for all $t\ge0$ and
$0\le u\le\delta$, we have that
\[
\mathbf{E}_{\xi_{L}^{x_{L}}}^{L}\left[f(W_{L}(t+u))-f(W_{L}(t))\bigg|\mathcal{\mathscr{F}}_{t}^{L}\right]=\mathbf{E}_{\xi_{L}^{x_{L}}}^{L}\left[\int_{t}^{t+u}(\mathcal{L}^{W_{L}}f)(W_{L}(s))ds\bigg|\mathcal{\mathscr{F}}_{t}^{L}\right]\;.
\]
Hence, in view of Proposition \ref{tightcrt}, it suffices to check
\begin{equation}
\lim_{\delta\rightarrow0}\limsup_{L\rightarrow\infty}\mathbf{E}_{\xi_{L}^{x_{L}}}^{L}\sup_{0\le u\le\delta}\left|\int_{t}^{t+u}(\mathcal{L}^{W_{L}}f)(W_{L}(s))ds\right|=0\;.\label{che2}
\end{equation}
By Proposition \ref{convgen}, it suffices to prove that
\[
\lim_{\delta\rightarrow0}\limsup_{L\rightarrow\infty}\mathbf{E}_{\xi_{L}^{x_{L}}}^{L}\sup_{0\le u\le\delta}\left|\int_{t}^{t+u}(\mathcal{L}^{\mathbb{T}^{d}}f)(W_{L}(s))ds\right|=0\;.
\]
This is obvious since $\mathcal{L}^{\mathbb{T}^{d}}f$ is a bounded
function on $\mathbb{T}^{d}$.
\end{proof}

\subsection{\label{sec84}Proof of the main results}
\begin{proof}[Proof of Theorems \ref{tt1}-\ref{tt3}]
Fix a sequence $(x_{L})_{L=1}^{\infty}$ that satisfies $x_{L}\in\mathbb{T}_{L}^{d}$
for all $L\ge1$ and $\lim_{L\rightarrow\infty}(x_{L}/L)=u$, as in
part (1) of Definition \ref{def218}. For simplicity, we write $\mathbb{Q}^{L}=\mathbb{Q}_{\xi_{L}^{x_{L}}}^{L}$
and $\mathbf{E}^{L}=\mathbf{E}_{\xi_{L}^{x_{L}}}^{L}$.

Let us first identify the limit points of the sequence $\{\mathbb{Q}^{L}\}_{L\ge1}$.
Let $\mathbb{Q}$ denote an arbitrary limit point of $\{\mathbb{Q}^{L}\}_{L\ge1}$.
Fix $f\in C^{3}(\mathbb{T}^{d})$ and consider
\[
M_{f}(t)=f(\omega(t))-f(\omega(0))-\int_{0}^{t}(\mathcal{L}^{\mathbb{T}^{d}}f)(\omega(s))ds\;\;;\;t\ge0\;,
\]
where $\omega(t)$ is the canonical coordinate process on $D([0,\,\infty),\,\mathbb{T}^{d})$.
Then, we claim that $(M_{f}(t))_{t\ge0}$ is a $\mathbb{Q}$-martingale,
i.e.,
\begin{equation}
\mathbb{E}^{\mathbb{Q}}\left[g((\omega(u):0\le u\le s))\left(M_{f}(t)-M_{f}(s)\right)\right]=0\label{claim}
\end{equation}
for all $0\le s\le t$ and for all bounded, continuous function $g$
on $D([0,\,s],\,\mathbb{T}^{d})$. To prove
\eqref{claim}, we recall the $\mathbb{Q}^{L}$-martingale $M_{f}^{L}(t)$
defined in \eqref{tmart} so that we have
\begin{equation}
\mathbf{E}^{L}\left[g((\omega(u):0\le u\le s))\left(M_{f}^{L}(t)-M_{f}^{L}(s)\right)\right]=0\;.\label{e817}
\end{equation}
By Proposition \ref{convgen}, we have
\begin{equation}
\lim_{L\rightarrow\infty}\left|M_{f}^{L}(t)-\left[f(W_{L}(t))-f(W_{L}(0))-\int_{0}^{t}(\mathcal{L}^{\mathbb{T}^{d}}f)(W_{L}(s))ds\right]\right|=0\label{e818}
\end{equation}
for all $t\ge0$, and hence by \eqref{e817} and \eqref{e818}, we
obtain that
\begin{equation}
\lim_{L\rightarrow\infty}\mathbf{E}^{L}\left[g((\omega(u):0\le u\le s))\left(M_{f}(t)-M_{f}(s)\right)\right]=0\;.\label{e817-1}
\end{equation}
Therefore, the proof of \eqref{claim} is completed if we can establish
the following limit:
\begin{align}
 & \mathbf{E}^{L}\left[g((\omega(u):0\le u\le s))\left(M_{f}(t)-M_{f}(s)\right)\right]\nonumber \\
 & \qquad\qquad\rightarrow\mathbb{E}^{\mathbb{Q}}\left[g((\omega(u):0\le u\le s))\left(M_{f}(t)-M_{f}(s)\right)\right]\text{\;\;as }L\rightarrow\infty\;.\label{e819}
\end{align}
This is not trivial since the map $H:\omega\mapsto g((\omega(u):0\le u\le s))\left(M_{f}(t)-M_{f}(s)\right)$
is not continuous on $D([0,\,\infty),\,\mathbb{T}^{d})$. However,
in \cite[Proposition 3.2]{A-G-L}, this limiting procedure has been
robustly confirmed and can be applied to our situation as well. Thus,
the claim is proved. It completes the identification of limit points
since the solution of the martingale problem is unique and since $C^{3}(\mathbb{T}^{d})$
consists the core of the generator $\mathcal{L}^{\mathbb{T}^{d}}$
given in \eqref{gent} because $\mathbb{T}^{d}$ is compact. Finally,
along with the tightness established in Proposition \ref{tightness},
we can conclude the convergence of the process $W_{L}(\cdot)$ to
$Y(\cdot)+u$ where $Y(\cdot)$ is the process generated by $\mathcal{L}_{\mathbb{T}^{d}}$
and starting at $0$. This finally completes the verification of part
(1) of Definition \ref{def218}.

Now, we turn to part (2) of Definition \ref{def218}, i.e., we prove
\begin{equation}
\lim_{L\rightarrow\infty}\sup_{\eta\in\mathcal{E}_{L}}\mathbb{E}_{\eta}^{L}\left[\int_{0}^{t}\mathbf{1}_{\mathcal{H}_{L}\setminus\mathcal{E}_{L}}(\eta_{L}(\theta_{L}s))ds\right]=0\text{\;\,\,for all }t>0\;.\label{e821}
\end{equation}
To this end, let us first fix $x\in\mathbb{T}_{L}^{d}$ and $t>0$.
Then, by the translation invariance of the model, we have
\[
\mathbb{E}_{\xi_{L}^{x}}^{L}\left[\int_{0}^{t}\mathbf{1}_{\mathcal{H}_{L}\setminus\mathcal{E}_{L}}(\eta_{L}(\theta_{L}s))ds\right]=\mathbb{E}_{\mu_{L}^{\mathcal{E}_{L}}}^{L}\left[\int_{0}^{t}\mathbf{1}_{\mathcal{H}_{L}\setminus\mathcal{E}_{L}}(\eta_{L}(\theta_{L}s))ds\right]
\]
since the invariant measure $\mu_{L}^{\mathcal{E}_{L}}(\cdot)$ of
the trace process is a uniform measure on $\mathcal{E}_{L}=\{\xi_{L}^{x}:x\in\mathbb{T}_{L}^{d}\}$.
Now, we can deduce from Fubini theorem that
\begin{align*}
\mathbb{E}_{\mu_{L}^{\mathcal{E}_{L}}}^{L}\left[\int_{0}^{t}\mathbf{1}_{\mathcal{H}_{L}\setminus\mathcal{E}_{L}}(\eta_{L}(\theta_{L}s))ds\right] & \le\frac{1}{\mu_{L}(\mathcal{E}_{L})}\mathbb{E}_{\mu_{L}}^{L}\left[\int_{0}^{t}\mathbf{1}_{\mathcal{H}_{L}\setminus\mathcal{E}_{L}}(\eta_{L}(\theta_{L}s))ds\right]\\
 & =\frac{1}{\mu_{L}(\mathcal{E}_{L})}\,t\,\mu_{L}(\mathcal{H}_{L}\setminus\mathcal{E}_{L})\;.
\end{align*}
Thus, \eqref{e821} follows from static condensation established in
Theorem \ref{tlmain}.
\end{proof}

\section{Appendix}

In the appendix, we collect several known results for the completeness
of the article.

\subsection{A lemma on the sum of reciprocals}

The following elementary lemma is repeatedly used throughout the article.
\begin{lem}
\label{prelem}For integers $n\ge k\ge1$, define
\[
A_{n,\,k}=\bigg\{(a_{1},\,\dots,\,a_{k})\in\mathbb{N}^{k}:a_{1},\,\dots,\,a_{k}\ge1\;\;\text{and\;\;}\sum_{i=1}^{k}a_{i}=n\bigg\}\;,
\]
and define
\[
S_{n,\,k}=\sum_{(a_{1},\,\dots,\,a_{k})\in A_{n,\,k}}\prod_{i=1}^{k}\frac{1}{a_{i}}\;.
\]
Then, it holds that
\begin{equation}
S_{n,\,k}\le\frac{(3\log(n+1))^{k-1}}{n}\text{ for all }n\ge k\ge1\;.\label{prelemeq}
\end{equation}
\end{lem}

\begin{proof}
We proceed by the mathematical induction on $k$. Note that the inequality
\eqref{prelemeq} is trivial for the initial case $k=1$. Now, we
fix $k\ge2$ and assume that \eqref{prelemeq} holds for $S_{n,\,\ell}$
with $\ell=k-1$ and $n\ge\ell$. Then, look at the inequality for
$S_{n,\,k}$ for some fixed $n$.

Since $a_{k}$ can take values from $1$ to $n-(k-1)$, we can write
\[
S_{n,\,k}=\sum_{m=1}^{n-(k-1)}\sum_{(a_{1},\,\dots,\,a_{k-1})\in A_{n-m,\,k-1}}\frac{1}{m}\prod_{i=1}^{k}\frac{1}{a_{i}}=\sum_{m=1}^{n-(k-1)}\frac{1}{m}S_{n-m,\,k-1}\;.
\]
Thus, by the induction hypothesis, we get that
\begin{equation}
S_{n,\,k}\le\sum_{m=1}^{n-(k-1)}\frac{1}{m}\frac{(3\log(n-m+1))^{k-2}}{n-m}\le(3\log(n+1))^{k-2}\sum_{m=1}^{n-(k-1)}\frac{1}{m(n-m)}\;.\label{prelemtemp2}
\end{equation}
The proof of the inequality \eqref{prelemeq} is completed since the
last summation can be estimated by
\begin{equation}
\sum_{m=1}^{n-(k-1)}\frac{1}{m(n-m)}=\frac{1}{n}\sum_{m=1}^{n-(k-1)}\left(\frac{1}{m}+\frac{1}{n-m}\right)\le\frac{3}{n}\log(n+1)\;.\label{prelemtemp3}
\end{equation}
Inserting \eqref{prelemtemp3} to \eqref{prelemtemp2} finishes the
proof of the induction step, and thus concludes the proof.
\end{proof}

\subsection{Gordan's lemma}

The following elementary lemma is used in the proof of Lemma \ref{lem710}.
This lemma has many equivalent statements, which include the one known
as Farkas' lemma.
\begin{lem}[Gordan's lemma]
\label{Gordan} Let $\boldsymbol{A}$ be an $m\times n$ matrix for
integers $m,\,n\ge1$. Then, exactly one of the following statements
holds.
\begin{itemize}
\item There exists a vector $\boldsymbol{\alpha}\in\mathbb{R}^{m}$ such
that all the components of $\boldsymbol{A}^{\dagger}\boldsymbol{\alpha}$
are positive.
\item There exists a vector $\boldsymbol{0}\neq\boldsymbol{\beta}\in\mathbb{R}^{n}$
such that all the components of $\boldsymbol{\beta}$ are non-positive
and such that $\boldsymbol{A\beta}=0$.
\end{itemize}
\end{lem}

\begin{proof}
We refer to e.g., \cite[Section 3]{Dax}.
\end{proof}

\subsection{A criterion for the tightness}

We introduce a criterion for the tightness of the random process which
is used in the proof of tightness of the speeded-up trace process
in the thermodynamic limit case in Section \ref{sec8}. This criterion
is thoroughly explained in \cite{Whitt}, and is also used in \cite{Gr-R-V 13}
to prove the metastable behavior of symmetric inclusion processes.
\begin{prop}
\label{tightcrt}For each $N\ge1$, let $X_{\cdot}^{N}$ be a continuous-time
Markov chain on $\Omega=\mathbb{R}^{d}$ or $\mathbb{T}^{d}$, and
let $\mathcal{F}_{t}^{N}$, $t\ge0$ be its natural filtration. Fix
$\{x_{N}\}_{N\ge1}\subseteq\Omega$ and let $\mathbb{P}_{x_{N}}$
and $\mathbb{E}_{x_{N}}$ denote the law and expectation of $X_{\cdot}^{N}$
starting at $x_{N}$, respectively. Then, the collection of laws $\{\mathbb{P}_{x_{N}}\}_{N\ge1}$
is tight in the path space $D([0,\,\infty);\,\Omega)$ provided that
both of the following conditions hold.
\begin{enumerate}
\item The sequence $\{X_{\cdot}^{N}\}_{N\ge1}$ is stochastically bounded
in $D([0,\,\infty);\,\Omega)$.
\item For all $F\in C_{c}^{\infty}(\Omega)$, there exists a family of non-negative
random variables $Z_{N}(\delta,\,F)$, $\delta>0$, such that, for
all $t\ge0$ and $0\le u\le\delta$,
\begin{equation}
\left|\,\mathbb{E}_{x_{N}}\left[F(X_{t+u}^{N})-F(X_{t}^{N})\bigg|\mathcal{F}_{t}^{N}\right]\right|\le\mathbb{E}_{x_{N}}\left[Z_{N}(\delta,\,F)\bigg|\mathcal{F}_{t}^{N}\right]\;\;\;\;\;\;\mathbb{P}_{x_{N}}\text{-a.s.}\;,\label{check1}
\end{equation}
and
\begin{equation}
\lim_{\delta\rightarrow0+}\limsup_{N\rightarrow\infty}\mathbb{E}_{x_{N}}Z_{N}(\delta,\,F)=0\;.\label{check2}
\end{equation}
\end{enumerate}
\end{prop}

\begin{proof}
See \cite[Lemma 3.11]{Whitt} for the proof
for the Euclidean case, i.e., $\Omega=\mathbb{R}^{d}$. The proof
for the case $\Omega=\mathbb{T}^{d}$ is obviously the same with that
of the Euclidean space.
\end{proof}
\begin{acknowledgement*}
SK was supported by the National Research Foundation of Korea (NRF)
grant funded by the Korea government (MSIT) (No. 2018R1C1B6006896)
and NRF-2019-Fostering Core Leaders of the Future Basic Science Program/Global
Ph.D. Fellowship Program. IS was supported by the National Research
Foundation of Korea (NRF) grant funded by the Korea government (MSIT)
(No. 2018R1C1B6006896 and No. 2017R1A5A1015626). The authors wish
to thank Professor Claudio Landim and Professor Kenkichi Tsunoda for
providing valuable comments regarding the first draft of the manuscript
that helped a lot to clarify the presentation of the article.
\end{acknowledgement*}

\end{document}